\documentclass[a4paper, reqno]{amsart}

\usepackage[usenames,dvipsnames]{color}
\usepackage{amsthm,amsfonts,amssymb,amsmath,amsxtra}
\usepackage{bm}
\usepackage{tikz}
\usepackage{tkz-euclide}
\usepackage[all]{xy}
\SelectTips{cm}{}
\usepackage{xr-hyper}
\usepackage[
   citecolor=Black,
   linkcolor=Red,
   urlcolor=Blue]{hyperref}
\usepackage{verbatim}
\usepackage{stmaryrd}
\usepackage{enumerate} 

\usepackage[margin=1.25in]{geometry}
\usepackage{mathrsfs}

\RequirePackage{xspace}
\RequirePackage{etoolbox}
\RequirePackage{varwidth}
\RequirePackage{enumitem}
\RequirePackage{tensor}
\RequirePackage{mathtools}
\RequirePackage{longtable}
\RequirePackage{multirow}

\usepackage{environ}

\def\<{\langle}
\def\>{\rangle}
\def\ge{\geqslant}
\def\le{\leqslant}
\def\a{\alpha}
\def\b{\beta}

\def\G{\Gamma}

\def\umu{\underline \mu}
\def\o{\omega}

\def\s{\sigma}
\def\t{\tau}
\def\th{\theta}
\def\k{\kappa}
\def\l{\lambda}

\def\i{^{-1}}

\def\<{\langle}
\def\>{\rangle}

\newcommand{\fka}{\ensuremath{\mathfrak{a}}\xspace}

\newcommand{\bG}{\mathbf G}
\newcommand{\bJ}{\mathbf J}

\def\brF{\breve F}

\def\brI{\breve \CI}
\def\brK{\breve \CK}

\def\brQp{\breve{\mathbb Q}_p}

\def\tSS{\tilde{\mathbb S}}

\newcommand{\BA}{\ensuremath{\mathbb {A}}\xspace}

\newcommand{\BE}{\ensuremath{\mathbb {E}}\xspace}
\newcommand{\BF}{\ensuremath{\mathbb {F}}\xspace}
\newcommand{{\BG}}{\ensuremath{\mathbb {G}}\xspace}

\newcommand{\BJ}{\ensuremath{\mathbb {J}}\xspace}
\newcommand{{\BK}}{\ensuremath{\mathbb {K}}\xspace}

\newcommand{\BM}{\ensuremath{\mathbb {M}}\xspace}

\newcommand{\BP}{\ensuremath{\mathbb {P}}\xspace}
\newcommand{\BQ}{\ensuremath{\mathbb {Q}}\xspace}

\newcommand{\BS}{\ensuremath{\mathbb {S}}\xspace}

\newcommand{\BZ}{\ensuremath{\mathbb {Z}}\xspace}

\newcommand{\CA}{\ensuremath{\mathcal {A}}\xspace}
\newcommand{\CB}{\ensuremath{\mathcal {B}}\xspace}

\newcommand{\CD}{\ensuremath{\mathcal {D}}\xspace}

\newcommand{\CF}{\ensuremath{\mathcal {F}}\xspace}

\newcommand{\CI}{\ensuremath{\mathcal {I}}\xspace}
\newcommand{\CJ}{\ensuremath{\mathcal {J}}\xspace}
\newcommand{\CK}{\ensuremath{\mathcal {K}}\xspace}
\newcommand{\CL}{\ensuremath{\mathcal {L}}\xspace}
\newcommand{\CM}{\ensuremath{\mathcal {M}}\xspace}
\newcommand{\CN}{\ensuremath{\mathcal {N}}\xspace}
\newcommand{\CO}{\ensuremath{\mathcal {O}}\xspace}

\newcommand{\CS}{\ensuremath{\mathcal {S}}\xspace}

\newcommand{\Ad}{{\mathrm{Ad}}}
\newcommand{\ad}{{\mathrm{ad}}}

\DeclareMathOperator{\diag}{diag}

\def\fin{\rm fin}

\DeclareMathOperator{\Adm}{Adm}
\newcommand{\Admmu}{\Adm(\mu)}
\newcommand{\BGmu}{B(\mathbf G, \mu)}
\newcommand{\XmubK}[2]{X(\mu, #1)_{#2}}
\newcommand{\KAdmmu}{{}^K\! \!\Adm(\mu)}

\newcommand{\Latt}{\mathop{\mathcal L \rm att}\nolimits}
\newcommand{\boldtau}{{\boldsymbol\tau}}
\DeclareMathOperator{\alc}{alc}

\DeclareMathOperator{\Gal}{Gal}

\newcommand{\GL}{\mathrm{GL}}

\DeclareMathOperator{\Hom}{Hom}

\newcommand{\id}{\ensuremath{\mathrm{id}}\xspace}

\newcommand{\Ind}{{\mathrm{Ind}}}

\newcommand{\inv}{{\mathrm{inv}}}

\newcommand{\loc}{\ensuremath{\mathrm{loc}}\xspace}

\newcommand{\PGL}{{\mathrm{PGL}}}

\DeclareMathOperator{\rad}{rad}

\DeclareMathOperator{\Res}{Res}

\DeclareMathOperator{\Spec}{Spec}
\DeclareMathOperator{\Spf}{Spf}

\newcommand{\U}{\mathrm{U}}

\DeclareMathOperator{\vol}{vol}

\newcommand{\wh}{\widehat}

\newcommand{\ov}{\overline}

\newcommand{\nass}{\noalign{\smallskip}}

\def\tW{\tilde W}

\DeclareMathOperator{\supp}{supp}

\newtheorem{theorem}{Theorem}

\newtheorem{proposition}[theorem]{Proposition}
\newtheorem{lemma}[theorem]{Lemma}

\newtheorem{corollary}[theorem]{Corollary}

\theoremstyle{definition}
\newtheorem{definition}[theorem]{Definition}
\newtheorem{example}[theorem]{Example}

\newtheorem{situation}[theorem]{Situation}
\newtheorem{condition}[theorem]{Condition}
\newtheorem{remark}[theorem]{Remark}

\newenvironment{altenumerate}
   {\begin{list}
      {\textup{(\theenumi)} }
      {\usecounter{enumi}
       \setlength{\labelwidth}{0pt}
       \setlength{\labelsep}{0pt}
       \setlength{\leftmargin}{0pt}
       \setlength{\itemsep}{\the\smallskipamount}
       \renewcommand{\theenumi}{\roman{enumi}}
      }}
   {\end{list}}
\newenvironment{altitemize}
   {\begin{list}
      {$\bullet$}
      {\setlength{\labelwidth}{0pt}
	   \setlength{\itemindent}{5pt}
       \setlength{\labelsep}{5pt}
       \setlength{\leftmargin}{0pt}
       \setlength{\itemsep}{\the\smallskipamount}
      }}
   {\end{list}}

\numberwithin{equation}{section}
\numberwithin{theorem}{section}

\setitemize[0]{leftmargin=*,itemsep=\the\smallskipamount}
\setenumerate[0]{leftmargin=*,itemsep=\the\smallskipamount}

\renewcommand{\to}{%
   \ifbool{@display}{\longrightarrow}{\rightarrow}%
   }
\let\shortmapsto\mapsto
\renewcommand{\mapsto}{%
   \ifbool{@display}{\longmapsto}{\shortmapsto}%
   }
\newcommand{\isoarrow}{%
   \ifbool{@display}{\overset{\sim}{\longrightarrow}}{\xrightarrow\sim}%
   }

\def\sqrttwo{1.4142135}

\begin{document}
\date{\today}
\title{Extremal cases of Rapoport-Zink spaces}
\author[U.~G\"{o}rtz]{U. G\"{o}rtz}
\address{Ulrich G\"{o}rtz\\Institut f\"ur Experimentelle Mathematik\\Universit\"at Duisburg-Essen\\45117 Essen\\Germany}
\email{ulrich.goertz@uni-due.de}

\author{X. He}
\address{Xuhua He\\The Institute of Mathematical Sciences and Department of Mathematics\\The Chinese University of Hong Kong\\Shatin, N.T.\\Hong Kong }
\email{xuhuahe@math.cuhk.edu.hk}

\author{M. Rapoport}
\address{Michael Rapoport\\Mathematisches Institut der Universit\"at Bonn, Endenicher Allee 60, 53115 Bonn, Germany, and Department of Mathematics, University of Maryland, College Park, MD 20742, USA}
\email{rapoport@math.uni-bonn.de}


\date{\today}

\begin{abstract}
We investigate qualitative properties of the underlying scheme of Rapoport-Zink formal moduli spaces of $p$-divisible groups, resp. Shtukas. We single out those cases when the dimension of this underlying scheme is zero, resp. those where the dimension is maximal possible. The model case for the first alternative is the Lubin-Tate moduli space, and the model case for the second alternative is the Drinfeld moduli space. We exhibit a complete list in both cases. 
\end{abstract}

\maketitle

\tableofcontents

\section{Introduction}
Let $F$ be a nonarchimedean local field and let $\bG$ be a connected reductive group over $F$. Let $\mu$ be a conjugacy class of cocharacters of $\bG$ (over the algebraic closure $\overline{F}$), and let $b\in\bG(\breve F)$, where $\breve F$ denotes the completion of the maximal unramified extension of $F$. The main character of this paper is the set 
\begin{equation}\label{Xintro}
\XmubK{b}{K}=X^{\bG}(\mu, b)_K:=\{g \brK \in \bG(\breve F)/\brK\mid g \i b \s(g) \in \brK\Admmu\brK\}.
\end{equation}
Here $K$ denotes a $F$-rational parahoric level structure of $\bG$, with corresponding standard parahoric subgroup $\brK\subset \bG(\breve F)$. Also, $\Admmu$ denotes the \emph{$\mu$-admissible subset} of the Iwahori-Weyl group of $\bG$. See Section~\ref{sec:preliminaries} for details on this notion and other notation used here. By \cite{He-KR}, $\XmubK{b}{K}$ is non-empty if and only if $[b]\in \BGmu$ (i.e., $[b]$ is \emph{neutral acceptable}), which we assume from now on.

The set \eqref{Xintro} has a geometric structure: if $F$ is a function field, then $\XmubK{b}{K}$ is a finite-dimensional closed  subscheme  of the partial affine flag variety $\bG(\breve F)/\brK$, locally of finite type over the algebraic closure of the residue field of $F$. If $F$ is $p$-adic, then the partial affine flag variety and its finite-dimensional closed subscheme $\XmubK{b}{K}$ have to be understood in the sense of  Bhatt--Scholze and Zhu \cite{BS, Zhu} as a perfect scheme.

The interest in the set \eqref{Xintro} comes from the fact that in the case of a $p$-adic field and when $\mu$ is \emph{minuscule}, sets of this form arise as the set of geometric points of the \emph{underlying reduced set} of a Rapoport-Zink formal moduli space of $p$-divisible groups, cf. \cite{RV}. Something analogous holds in the function field case for formal moduli spaces of \emph{Shtukas}, cf. \cite{HV} (in the latter case, the minuscule hypothesis can be dropped).  Both classes of formal schemes  are very mysterious. In fact, we know explicitly these formal schemes essentially only in two cases: the \emph{Lubin-Tate case} and the \emph{Drinfeld case}. In the first case, the formal scheme is a disjoint union of formal spectra of formal power series rings with coefficients in $O_{\breve F}$, hence the underlying reduced scheme is just a disjoint union of points. In the second case, the formal scheme is $\pi$-adic and the underlying reduced set is a disjoint union of special fibers of the Deligne-Drinfeld formal model of the $p$-adic halfspace corresponding to the local field $F$.

 In the present paper, we address the question of classifying the cases when $\XmubK{b}{K}$ has minimal dimension zero (as in the Lubin-Tate case) or maximal dimension $\langle \mu, 2\rho\rangle$ (as in the Drinfeld case). 

\smallskip

Let us first discuss our  results pertaining to the case of dimension zero.
\begin{theorem}[comp.~Theorem~\ref{thm-dimension-0}]\label{Main1}
Assume that $\bG$ is quasi-simple over $F$ and that $\mu$ is non-central.  Let $b$ be basic, and let $K$ be a $F$-rational parahoric level structure. Then $\XmubK{b}{K}$ is zero-dimensional if and only if $\bG_\ad$ is isomorphic to $\Res_{\tilde F/F}(\PGL_n)$, for some $n$ and some finite extension $\tilde F$ of $F$, and there exists a unique $F$-embedding $\varphi_0\colon \tilde F\to \ov F$ such that  $\mu_{\ad, {\varphi}}$ is trivial  for $\varphi\neq\varphi_0$ and $\mu_{\ad, {\varphi_0}}= \o^\vee_1$. 

{\rm Here we write, for any $\tilde F$-group $\tilde \bG$,   a cocharacter $\mu$ of $\Res_{\tilde F/F}(\tilde \bG)$ as $\mu=(\mu_{\varphi})_\varphi$ for cocharacters $\mu_\varphi$ of $\tilde \bG$, where $\varphi$ runs over $\Hom_F(\tilde F, \ov F)$. }
\end{theorem}
In particular, if $\bG$ is absolutely quasi-simple, then the Lubin-Tate case (Example~\ref{ex-lubin-tate}) is the only one when the dimension of $\XmubK{b}{K}$ is zero. In general, when the dimension of $\XmubK{b}{K}$ is zero, then  $\mu$ is   automatically minuscule. Also, the statement that the dimension of $\XmubK{b}{K}$ is zero is independent of the choice of $K$. The case $(\bG, \mu)$ that appears in Theorem \ref{Main1} is called the \emph{extended Lubin-Tate case} (we use the term \emph{extended} because there is an extension $\tilde F/F$ involved). 

When we vary $K$, we obtain the transition morphisms $\pi_{K, K'}\colon \XmubK{b}{K}\to \XmubK{b}{K'}$, whenever $K\subset K'$. In the extended Lubin-Tate case, the fibers of $\pi_{K, K'}$ are finite for any $K\subsetneqq K'$. For the next statement, let us exclude this case. 

\begin{theorem}[comp.~Theorem~\ref{thm-fixed-s}]\label{Main2}
Assume that $\bG$ is quasi-simple over $F$ and that $\mu$ is non-central. Let $b$ be basic. Also, exclude the extended Lubin-Tate case discussed in the previous theorem. Fix a pair $K\subsetneqq K'$ of $F$-rational parahoric level structures.

Then the fibers of $\pi_{K, K'}$ are all finite  if and only if $\bG_\ad$ is isomorphic to $\Res_{\tilde F/F}(\tilde \bG_\ad)$, where $\tilde F$ is a finite extension of $F$, and where  $\tilde \bG_\ad$ is  the adjoint group of a unitary group associated to a \emph{split} $\tilde F'/\tilde F$-hermitian vector space $V$ for an \emph{unramified} quadratic extension $\tilde F'/\tilde F$,  and the following two conditions are satisfied:

\begin{altitemize}
\item there exists a unique $F$-embedding $\varphi_0\colon \tilde F\to \ov F$ such that $\mu_{\ad, {\varphi}}$ is trivial for $\varphi\neq\varphi_0$ and $\mu_{\ad, {\varphi_0}}= \o^\vee_1$;
\item the pair $(K, K')$ satisfies: let  the maximal unramified subextension $F_d$  of $\tilde F/F$ have degree $d$. Correspondingly write $K$ and $K'$ as $K=(K_1, \ldots, K_d)$ and $K'=(K'_1, \ldots, K'_d)$, where the entries  are parahoric subgroups of $\Res_{\tilde F/F_d}(\tilde \bG_\ad)$\footnote{Note that $\Res_{\tilde F/F_d}(\tilde \bG_\ad)$  has  affine Dynkin type $\tilde A_{n-1}$; we use standard notation for the simple reflections in this case.}.   Then $K_1'\setminus K_1 \subset \{s_0, s_{\frac{n}{2}}\}$, and  if $s_i \in K_1'\setminus K_1$, then $s_{i+1} \notin K_1$.
\end{altitemize} 
\end{theorem}

Both implications of the theorem are interesting. Indeed, in the case singled out by the theorem,  assume for simplicity that $\tilde F=F$, and consider a  maximal selfdual periodic lattice chain 
$$
\{\ldots\subset\Lambda_{-2}\subset\Lambda_{-1}\subset\Lambda_0\subset\Lambda_1\subset\Lambda_2\subset\ldots\}
$$
in $V$. The case when $K'\setminus K=\{s_0\}$ is given as follows: $K'$ stabilizes a subchain $\Lambda_I$ which   contains $\Lambda_1$ but not  the selfdual lattice $\Lambda_0$, and $K$ stabilizes $\Lambda_0$ in addition to $\Lambda_I$. Under these conditions, the theorem states the following. Let $N$ be a $\breve F$-vector space of dimension $2\dim V$, equipped with an action of $\tilde F$ and an alternating bilinear form $\langle\, , \, \rangle$ which is hermitian with respect to the $\tilde F$-action. Let $\phi$ be a $\sigma$-linear automorphism of $N$ which commutes with the $\tilde F$-action and which is isoclinic of slope $1/2$ and  such that $\langle\phi(x), \phi(y)\rangle=\pi\sigma(\langle x, y\rangle)$, for all $x, y\in N$. Here $\pi$ denotes a uniformizer in $F$. Let $\CM_I$ be a self-dual  chain of $O_{\breve F}$-lattices in $N$ which are invariant under $O_{\tilde F}$, \emph{ of  type $\Lambda_I$}. Assume that $\pi\CM_i\subset \phi(\CM_i)\subset^1\CM_i$ for all $i\in I$. Then there are only finitely many ways of completing the chain $\CM_I$  to  a self-dual chain by adding a self-dual lattice $\CM_0$  such that $\pi\CM_0\subset \phi(\CM_0)\subset^1\CM_0$.

The case when $K'\setminus K=\{s_m\}$ when $n=2m$ is similar (with a selfdual lattice replaced by a lattice which is selfdual up to a scalar); and the case when $K'\setminus K =\{s_0, s_m\}$ when $n=2m$ is a concatenation of the previous cases.

From a global perspective, i.e., the point of view of Shimura varieties, Theorem~\ref{Main1} implies that the only cases where the basic locus is $0$-dimensional are those which at the fixed prime $p$ give rise to the extended Lubin-Tate case. This is the situation considered by Harris and Taylor in~\cite{HT}.

\smallskip

Now let us discuss our  results pertaining to the case of maximal  dimension. First, we have the following well-known upper bound on the dimension of $\XmubK{b}{K}$, cf. \cite{He-CDM}. As usual, $\rho$ denotes the half sum of all positive roots, and by $ \langle\mu, 2 \rho\rangle$ we mean the value of $2\rho$ on a dominant representative of $\mu$. 

\begin{proposition}[Corollary~\ref{cor:bound}]\label{Mainbound}
The dimension of $X(\mu, b)_K$ is bounded as 
$$\dim X(\mu, b)_K\leq \langle\mu, 2 \rho\rangle.$$
 If equality holds, then $b$ is basic. 
\end{proposition}

It is thus a natural question to ask in which cases this upper bound is attained. A well-known example is the Drinfeld case, but there are other cases, too.

\begin{theorem}[Theorem~\ref{max-dim}]\label{Main3}
    Assume that $\bG$ is quasi-simple over $F$ and that $\mu$ is not central. If $\dim \XmubK{b}{K}=\langle\mu, 2\rho\rangle$, then $b$ is basic,  the $\sigma$-centralizer group ${\bf J}_b$ is a quasi-split inner form of $\bG$ and $\umu$ is minuscule (in the \'echelonnage root system\footnote{The latter condition implies that $\mu$ is minuscule but is slightly stronger if $\bG$ does not split over $\breve F$.}, see Section~\ref{subsec:admissible-set}). If $K=\emptyset$ is the Iwahori level, the converse holds. 

For a general parahoric level, 
$\dim \XmubK{b}{K}=\langle\mu, 2\rho\rangle$ if and only if $b$ is basic and $W(\umu)_{K, \fin}\neq\emptyset$.  In this case, the orbits of the action of ${\bf J}_b(F)$ on the set of  irreducible components of $\XmubK{b}{K}$ of dimension $\<\mu, 2 \rho\>$ are parametrized by  the finite set $ W(\umu)_{K,\fin}$.

\smallskip

{\rm We refer to \eqref{Wfin} for the definition of  $W(\umu)_{K, \fin}$, a finite set of translation elements, which is related to Drinfeld's notion of critical index (see Proposition~\ref{W-fin}). }
\end{theorem}

The constraints on $(\bG, \mu, b, K)$ imposed by Theorem \ref{Main3} are in fact quite weak. For instance, if $(\bG, \mu, b)$ is such that $\mu$ is minuscule  and $b$ basic, and such that $\bG$ is split over $\breve F$, then there always exists an inner form $\bf H$ of $\bG$
such that $\dim X^{\bf H}(\mu, b)_\emptyset=\langle\mu, 2\rho\rangle$.
\smallskip

On the other hand, the condition that $\dim \XmubK{b}{K}$ be equi-dimensional of maximal dimension is much stronger.
\begin{theorem}[Theorem~\ref{thmmaxequi}]\label{Main4}
Assume that $\bG$ is quasi-simple over $F$ and that $\mu$ is not central. Let $b\in\bG(\breve F)$ be a representative of the unique basic element in $\BGmu$. Then $\XmubK{b}{K}$ is equi-dimensional of dimension equal to $\<\mu, 2 \rho\>$ if and only if the triple $(\bG_\ad, \mu_\ad, K)$ is isomorphic to  one of the following. 
	
	\begin{enumerate}
		\item $\big(\Res_{\tilde F/F}(D^\times_{1/n})_\ad, \o^\vee_1(\varphi_0),\emptyset\big)$.
		
		\item $\big(\Res_{\tilde F/F}\PGL_2(D_{1/2}), \o^\vee_2(\varphi_0),  \emptyset\big)$.

		\item $\big(\Res_{\tilde F/F}(\PGL_n),  \mu,  \emptyset\big)$.
		
	\end{enumerate}
			
\end{theorem}

		{\rm Here  $\tilde F$ denotes a finite extension of $F$ and, for an adjoint reductive group $\tilde\bG$ over $\tilde F$ and a cocharacter $\tilde\mu$ of $\tilde\bG$ and an embedding $\varphi_0\colon\tilde F\to\ov F$, we denote by $\tilde\mu(\varphi_0)$ the cocharacter $\mu$  of $\Res_{\tilde F/F}(\tilde\bG)$ with $\mu_{\varphi}=0$ for $\varphi\neq\varphi_0$ and $\mu_{\varphi_0}=\tilde\mu$. Furthermore, $D_{1/n}$ denotes the central division algebra over $\tilde F$ with invariant $1/n$, and $D^\times_{1/n}$ the algebraic group over $\tilde F$ associated to its multiplicative group. In (3), there are two embeddings $\varphi_0, \varphi_1\colon\tilde F\to\ov F$ such that their restrictions  to the maximal unramified subextension of $\tilde F/F$ are distinct, and  the cocharacter $\mu$ is given as follows: $\mu_{\varphi_0}=\o^\vee_1$ and $\mu_{\varphi_1}= \o^\vee_{n-1}$ and $\mu_\varphi=0$ for $\varphi\notin\{\varphi_0, \varphi_1\}$.   }

\smallskip
			
The case (1) is the \emph{extended Drinfeld case}. The case (2)  is somewhat surprising and was unknown to us before. The case (3) in the case of an unramified quadratic extension $\tilde F/F$ is the \emph{Hilbert-Blumenthal case}. It was discovered by Stamm \cite{St} in the case  $\bG=\Res_{\tilde F/F}\GL_2$. 

It is remarkable that in all three cases the parahoric level structure $K$ is the Iwahori level. This implies the following characterization of the Drinfeld case.
\begin{corollary}[comp. Corollary~\ref{charDrin}]\label{MainvaryingK}
Assume that $\bG$ is quasi-simple over $F$ and that $\mu$ is not central. Then $\XmubK{\t}{K}$ is equi-dimensional of dimension equal to $\<\mu, 2 \rho\>$ for every $F$-rational parahoric level structure $K$ if and only if   $(\bG_\ad, \mu_\ad)$ is isomorphic to $\big(\Res_{\tilde F/F}(D^\times_{1/n})_\ad, \o^\vee_1(\varphi_0)\big)$.
\end{corollary}
One of our motivations of this paper was to characterize the Drinfeld case. Scholze suggested to characterize it  through the dimension of its underlying reduced scheme. Theorem \ref{Main4} shows that this is not quite possible. But Corollary \ref{MainvaryingK} shows that this  is  possible when $K$ is varying. 

As a consequence of Corollary \ref{MainvaryingK}, we can characterize the Drinfeld case as the only  Rapoport-Zink space which is a $\pi$-adic formal scheme. We place ourselves in the context of \cite[\S 4]{HPR}; in particular, in the rational RZ-data 
$(F, B, V, (\, , \, ), *, G, \{\mu\}, [b])$, the first entry $F$ is a field. Also, RZ-spaces are modelled on the local models of \cite[\S 2.6]{HPR}. Hence we make a tame ramification hypothesis, cf.~loc.~cit. 
\begin{theorem}\label{charpadicunif}
Let $\CD_{\BZ_p}$ be integral RZ-data  such that the associated reductive group $\bG$ is connected and quasi-simple over $\BQ_p$, and the associated cocharacter $\mu$ is non-central. Let $E$ be its reflex field, and let $\CM_{\CD_{\BZ_p}}$ be the associated RZ-space, a   formal scheme  flat over $\Spf O_{\breve E}$. Then $\CM_{\CD_{\BZ_p}}$ is a $\pi$-adic formal scheme if only if ${\CD_{\BZ_p}}$ is of extended Drinfeld type, in which case $\CM_{\CD_{\BZ_p}}$ is isomorphic to the disjoint sum of copies of  $\wh{\Omega}^n_{ E}\wh{\otimes}_{O_E}O_{\breve E}$, where $\wh{\Omega}^n_{ E}$ is the Deligne-Drinfeld formal model of the Drinfeld half space attached to $E$. 
\end{theorem}

Here the integral RZ-data are said to be of extended Drinfeld type if the rational RZ-data are of type (EL) with $B=D_{1/n}$, $\dim_B(V)=1$, $\mu=\o_1^\vee(\varphi_0)$ and $b$ basic, and the integral RZ-data are given by a complete periodic  $O_B$-lattice chain in $V$.   

Through Rapoport-Zink uniformization, this theorem implies that there is no $p$-adic uniformization of Shimura varieties beyond the Drinfeld case. Note that the characterization of $p$-adic uniformization through the fact that the basic Newton stratum makes up the whole special fiber leads to Kottwitz's determination of all \emph{uniform} pairs $(\bG, \mu) $, cf. \cite[\S 6]{Ko2} and Section~\ref{ss:unif}. It appears interesting to us that one  can also characterize the Drinfeld case in a \emph{purely local} way,  without relating it to a Shimura variety.

\smallskip
The lay-out of the paper is as follows. The paper consists of three parts. In the first part, we provide the necessary background and introduce the terminology used. The second part is devoted to the case of dimension zero. In Section \ref{s:reszero}, we discuss the main results of this part. Sections \ref{s:onedirzero} and \ref{s:otherdirzero} are devoted to the proofs. In section \ref{s:lattice-interpretation}, we explain in lattice-theoretic terms the minimal cases of Theorems \ref{Main1} and \ref{Main2}. In Section \ref{finalpart2}, we give the proofs of Theorems \ref{Main1} and \ref{Main2} above. The third part is devoted to the case of maximal dimension. In Section \ref{s:dimADLV},   we recall the dimension theory of some subsets of $\breve G$ and prove Proposition \ref{Mainbound}. In Section \ref{s:resmax}, we discuss the main results of this part. Section \ref{s:critind} is preparatory for the proof but it also contains results on Drinfeld's critical index set which are of independent interest (in particular, we solve a problem posed 20 years ago in \cite{RZ:indag}). In Section \ref{s:max}, we give the proof of Theorem \ref{Main3}, and in Section \ref{s:equimax} the proof of Theorem \ref{Main4}. In Section \ref{s:lattice-interpretation-max}, we explain the equi-maximal cases in lattice-theoretic terms. In Section \ref{s:padic}, we discuss various ways of singling out the Drinfeld case among the three cases occurring in the classification of Theorem \ref{Main4}.  Section \ref{finalpart3} gives the proofs of the results stated above for the case of maximal dimension. 

\smallskip

\noindent {\bf Acknowledgements:} We thank P.~Scholze for raising the question of characterizing the Drinfeld case through the dimension of its underlying reduced scheme and insisting that we state our results in their natural generality; similarly, we thank N.~Ramachandra for (implicitly) raising the question of characterizing the Lubin-Tate space. We also  thank L.~Fargues for helpful discussions and H.~Wang for pointing out a mistake in a previous version.

U.G. was supported by the grant SFB/TR 45 from the Deutsche Forschungsgemeinschaft. X.H. was partially supported by NSF grant DMS-1801352. M.R. was supported by the grant SFB/TR 45 from the Deutsche Forschungsgemeinschaft and by funds connected with the Brin E-Nnovate Chair  at the University of Maryland.

\smallskip

\noindent {\bf Notation:} For a local field $F$, we denote by $O_F$ its ring of integers and by $k$ its residue field. We denote by $\breve F$ the completion of the maximal unramified extension, by $O_{\breve F}$ or $\breve{O}_F$ its ring of integers, and by $\s$ its Frobenius generator of $\Gal(\breve F/F)$. 
\part{Background}
\section{Preliminaries}\label{sec:preliminaries}

\subsection{The  Iwahori-Weyl group} Let $F$ be a nonarchimedean local field and $\brF$ be the completion of the maximal unramified extension $F^{\rm un}$ of $F$. We denote by $\s$ its Frobenius morphism, and by $\pi\in O_F$ a uniformizer. Let $\bG$ be a connected reductive group over $F$. We fix a $\s$-stable Iwahori subgroup $\brI$ of $\breve G=\bG(\brF)$.

We fix a maximal torus $T$ which after extension of scalars is contained in a Borel subgroup of $\bG \otimes_F \brF$, and such that $\brI$ is the Iwahori subgroup fixing an alcove $\fka$ in the apartment attached to the split part of $T$. The {\it Iwahori Weyl group} is defined by $$\tilde W=N(\brF)/(T(\brF)\cap \brI),$$
cf.~\cite{Tits:Corvallis}, \cite{Haines-Rapoport}.
Let $W_0=N(\breve F)/T(\breve F)$. Then we have 
\begin{equation}
\tilde W=X_*(T)_{\Gamma_0} \rtimes W_0,
\end{equation} where $\Gamma_0=\Gal(\ov F/F^{\rm un})$. The splitting depends on the choice of a special vertex of the base alcove $\fka$ that we fix in the sequel. When considering an element $\l\in X_*(T)_{\Gamma_0}$ as an element of $\tilde W$, we write $t^\l$. 

Let $\tSS$ be the set of simple reflections in $\tilde W$ determined by the base alcove $\fka$ and $\BS=\tSS \cap W_0$. For any subset $K$ of $\tSS$, we denote by $W_K$ the subgroup of $\tW$ generated by simple reflections in $K$. We also denote by $^K \tilde W$ the set of representatives of minimal length of the cosets $W_K\backslash \tilde W$. If $W_K$ is a finite group, we denote by $\brK$ the corresponding standard parahoric subgroup.

The Iwahori-Weyl group is a quasi-Coxeter group. More precisely, 
\begin{equation}
\tilde W=W_a \rtimes \Omega,
\end{equation}
where $W_a$ is the affine Weyl group with set $\tilde \BS$ as simple reflections, and $\Omega$ is the set of elements stabilizing the base alcove $\frak a$, cf. \cite[\S 2.2]{He-CDM}. The length function on $W_a$ is extended to $\tilde W$ by $\ell(w\tau)=\ell(w)$, for $w\in W_a$ and $\tau\in\Omega$. For $w\in\tilde W$, we denote by $\tau(w)$ its image in $\Omega$. 

\subsection{Admissible sets and acceptable sets}\label{subsec:admissible-set} Let $\mu$ be a conjugacy class of cocharacters of $\bG$. We can always choose a $\breve F$-rational representative $\mu_+$ in this conjugacy class. We make a definite choice as follows. We identify $X_*(T)_{\Gamma_0, \mathbb R}$ with the standard apartment (the apartment attached to the split part of $T$), using our choice of special vertex of $\mathfrak a$. We then fix the unique Weyl chamber containing $\mathfrak a$, which we declare to be the dominant Weyl chamber. Then $\mu_+$ is to be chosen such that $t^{\mu_+} \mathfrak a = \mu_+ + \mathfrak a$ is contained in the dominant Weyl chamber. We denote by $\underline \mu$ the image in $X_*(T)_{\Gamma_0} $ of  $\mu_+$. 
\begin{remark}\label{rmk-dominant-antidominant}
The choice of dominant Weyl chamber determines a Borel subgroup $B$ of
$\bG\otimes_F\brF$ containing $T$. Note that  $\umu$ is equal to the
image  in $X_*(T)_{\Gamma_0, \mathbb R}$ of the \emph{$B$-anti-dominant} representative of the conjugacy class
$\mu \subset X_*(T)$! This phenomenon is
already visible when $\bG$ is split and is reflected by the minus sign in
equation~(5), p.~31, of~\cite{Tits:Corvallis}. The minus sign in turn is forced
upon us by loc.~cit., equation~(4) which could not possibly extend to the
non-commutative normalizer if the left hand side was replaced by $sX_\alpha
s\i$. It means that for $\lambda\in X_*(T)$, the element $\lambda(\pi)$ acts on
the apartment for $T$ by translation by $-\lambda$, i.e., as the element
$t^{-\lambda}$.

It also means that even for a split group the values $\<\mu,\a\>$ and $\<\umu,\a\>$ for a root $\a$ differ by sign.
\end{remark}

If $\umu$ is minuscule, then $\mu$ is minuscule; but the converse does not hold, comp. the table right before Lemma 5.4 in \cite{HPR}. More precisely, we have
\begin{lemma}\label{lemma-mu-minuscule}
 Write $\bG_\ad=\Res_{\tilde F/F}(\tilde \bG_\ad)$, where the $\tilde F$-group $\tilde\bG_\ad$ is absolutely simple. Let the maximal unramified subextension $F_d$ of $\tilde F/F$ have degree $d$, and write correspondingly $\umu=(\umu_1,\ldots,\umu_d)$, where the entries $\umu_i$ correspond to the various embeddings $\iota_i\colon F_d\to\brF$. If $\umu$ is minuscule, then for every $i$ there exists an embedding $\varphi_{i,0}\colon\tilde F\to \ov F$ inducing $\iota_i$ such that $\mu_\varphi=0$ for every $\varphi\neq\varphi_{i, 0}$ inducing $\iota_i$ and with $\mu_{\varphi_{i, 0}}$ minuscule. 
\end{lemma}
\begin{proof} One is immediately reduced to the case where $\tilde F/F$ is totally ramified, i.e., $d=1$; therefore, we may drop the index $i$.  Let $\tilde T$ be a maximal torus of $\tilde\bG$ which after extension to $\brF$ is contained in a Borel subgroup, and let $T=\Res_{\tilde F/F}(\tilde T)$. The sum homomorphism 
$X_*(T)=\Ind_{\Gamma_0}^{\tilde\Gamma_0} \big(X_*(\tilde T)\big)\to X_*(\tilde T)$
induces an identification 
\begin{equation}\label{sumcoinv}
 X_*( T)_{\Gamma_0}= X_*(\tilde T) .
\end{equation}
Here $\tilde \Gamma_0=\Gal(\ov F/\tilde F^{\rm un})$. Under the identification of \eqref{sumcoinv}, we have  $\umu=\sum_{\varphi} \mu_{+, \varphi}$.  From this the claim follows easily.
\end{proof}

Furthermore, we have 

\begin{lemma}\label{lemma-mu-central}
    With notation as above, $\mu$ is central if and only if $\umu$ is central.
\end{lemma}

\begin{proof}
    If $\mu$ is central, then clearly $\umu$ is central. Conversely, assuming that $\umu$ is central, we need to show that $\langle \mu_+, \alpha \rangle = 0$ for every (absolute) root $\alpha$. Assume by contradiction that $\langle \mu_+, \alpha \rangle <0$ for some $\alpha$ (comp.~Remark~\ref{rmk-dominant-antidominant}).  Let us write $[\umu]$, when considering $\umu$ as an  element   of 
    $X_*(T)_\BQ^{\Gamma_0}$. We want to show that the relative root ${\rm res}(\alpha)$ defined by $\alpha$ by restriction to $X_*(T)_\BQ^{\Gamma_0}$ takes a strictly positive value on $[\umu]$.  However, with $\mu_+$ also every Galois translate of $\mu_+$ under an element of $\Gamma_0$ is anti-dominant; and $[\umu]$ is the average over the $\Gamma_0$-orbit of $\mu_+$. But then ${\rm res}(\alpha)$  takes a strictly positive value on $[\umu]$, and this contradicts the assumption that $\umu$ is central.
    \end{proof}

The {\it $\mu$-admissible set} is defined by
\begin{equation}
\Admmu=\{w \in \tilde W\mid w \le t^{x(\underline \mu)} \text{ for some }x \in W_0\} ,
\end{equation}
cf.~\cite{Ra}. For $\lambda$ a cocharacter (rather than a conjugacy class of cocharacters), we denote by $\Adm(\lambda)$ the admissible set of the conjugacy class of $\lambda$.
Let $B(\bG)$ be the set of $\s$-conjugacy classes in $\breve G$. Kottwitz \cite{Ko1, Ko2} gave a description of the set $B(\bG)$. It uses the {\it Kottwitz map},
\begin{equation}\label{kotmap}
\kappa\colon B(\bG)\to  \pi_1(\bG)_{\G} ,
\end{equation}
 where $\G$ is the Galois group of $\overline F$ over $F$. Any $\s$-conjugacy class $[b]$ is determined by two invariants: 
\begin{itemize}
	\item The element $\k([b]) \in \pi_1(\bG)_{\G}$; 
	
	\item The Newton point $\nu_b$ in the dominant chamber of $X_*(T)_{\G_0} \otimes \BQ$. 
\end{itemize}

The set of \emph{neutrally acceptable} $\s$-conjugacy classes is defined by 
\begin{equation}
\BGmu=\{[b] \in B(\bG)\mid \k([b])=\k(\mu), \nu_b \le \mu^\diamond\},
\end{equation}
where $\mu^\diamond = [\Gamma : \mathop{\rm Stab}_\Gamma(\mu_+)]^{-1}\sum_{\gamma \in \Gamma/\mathop{\rm Stab}_\Gamma(\mu_+)} \gamma(\mu_+)$ is the Galois average of $\mu_+$, an element of $X_*(T)^{\G_0} \otimes \BQ\cong X_*(T)_{\G_0} \otimes \BQ$. 

\subsection{Affine Deligne-Lusztig varieties}

The \emph{affine Deligne-Lusztig variety} (for the Iwahori subgroup) associated to  $w \in \tW$ and $b \in \breve G$  is
\begin{equation}
X_w(b)=\{g \brI \in \breve G/\brI\mid g \i b \s(g) \in \brI w \brI\},
\end{equation}
cf.~\cite{Ra}. Then $X_w(b)$ is (the set of $\overline{\mathbb F}_p$-points of) a locally closed subscheme of the affine flag variety of $\bf G$, locally of finite type over $\overline{\mathbb F}_p$ and of finite dimension; this follows from~\cite{RZ:indag}. If $F$ is of equal characteristic, then by affine flag variety we mean the ``usual'' affine flag variety; in the case of mixed characteristic, this notion should be understood in the sense of perfect schemes, as developed by Zhu~\cite{Zhu} and by Bhatt and Scholze~\cite{BS}.

Denote by ${\bf J}_b$ the  $\s$-centralizer group of $b$, an algebraic group over $F$ with $F$-rational points
\begin{equation}\label{scentrali}
{\bf J}_b(F)=\{g\in \bG(\breve F)\mid g^{-1}b\s(g)=b \}.
\end{equation}
Then ${\bf J}_b(F)$ acts on $X_w(b)$. 
Let $K \subset \tSS$ such that $W_K$ is finite, with corresponding standard parahoric subgroup $\brK\subset \breve G$. Here, and whenever we consider the space $\XmubK{b}{K}$ below, we assume that $\s(K)=K$. We set 
\begin{equation}
\XmubK{b}{K}=\{g \brK \in \breve G/\brK\mid g \i b \s(g) \in \brK\Admmu\brK\}.
\end{equation}

For $K=\emptyset$, we write simply $\XmubK{b}{}$ for $\XmubK{b}{K}$. Then $\XmubK{b}{}$ is a union of affine Deligne-Lusztig varieties. 

We will need the following result (conjectured in \cite{KR03, Ra}).

\begin{theorem}[\cite{He-KR}] \label{KR-conj}
    Let $K \subset \tSS$ such that $\s(K)=K$ and $W_K$ is finite.  Then $\XmubK{b}{K} \neq \emptyset$ if and only if $[b] \in \BGmu$. \qed
\end{theorem}

\subsection{Fine affine Deligne-Lusztig varieties}\label{fine}

We recall the definition of fine affine Deligne-Lusztig varieties inside the partial affine flag variety $\breve G/\brK$, cf. \cite[\S 3.4]{GH}.  
For $K\subset \BS$, $w \in {}^K \tW$ and $b \in \breve G$, the associated \emph{fine affine Deligne-Lusztig variety} is
\begin{equation}\label{def-fine-ADLV}
X_{K, w}(b)=\{g \brK\mid g \i b \s(g) \in \brK \cdot_\s \brI w \brI\}.
\end{equation} 

Note that we have the decomposition of the partial affine flag variety $\breve G/\brK$ into  ordinary affine Deligne-Lusztig varieties (for the parahoric subgroup associated to $K$),  $$\breve G/\brK=\bigsqcup_{x \in W_K\backslash \tW/W_K} \{g \brK\mid g \i b \s(g) \in \brK x \brK\} .$$ An  ordinary affine Deligne-Lusztig variety decomposes in turn  into a disjoint sum of  fine affine Deligne-Lusztig varieties, 
\begin{equation}
\{g \brK\mid g \i b \s(g) \in \brK x \brK\}=\bigsqcup_{w \in {}^K \tW\cap W_K x W_K} X_{K, w}(b),
\end{equation}
cf. \cite[\S 3.4]{GH}. 

\subsection{The decomposition of $\XmubK{b}{K}$}\label{subsec:EKOR} We set
\[
\KAdmmu=\Admmu \cap {}^K \tW.
\]
 It is proved in \cite[Thm. 6.1]{He-KR} that $\KAdmmu=W_K \Admmu W_K \cap {}^K \tW$. Hence 
\begin{equation}\label{decXmu}
\XmubK{b}{K}=\bigsqcup_{w \in \KAdmmu} X_{K, w}(b).
\end{equation}

We can read the definition~\eqref{def-fine-ADLV} as saying that $X_{K, w}(b)$ is the image of $X_w(b)$ under the projection map $\breve G/\brI \to \breve G/\brK$.
We call this decomposition the \emph{EKOR stratification}, and accordingly call the subsets $X_{K, w}(b)$ the \emph{EKOR strata} inside $\XmubK{b}{K}$. If $K=\emptyset$, we speak of the \emph{KR stratification} and of \emph{KR strata} instead. These stratifications are the ``local analogues'' of the stratifications defined in the global context in~\cite{HR}. But since here we always fix a $\s$-conjugacy class $[b]$, an EKOR stratum in our context really corresponds to the intersection of a global EKOR stratum with the Newton stratum attached to $[b]$. In~\cite{GH}, EKOR strata were called EO strata (see loc.~cit., Section 5.1).

\subsection{Tits data}\label{subsec:tits-data}

We recall the notion of Tits data and Coxeter data from \cite[Def.~5.3]{HPR}.
For an affine Coxeter system $(W_a, \tSS)$, we denote by $W_0$ the finite Weyl group, and by $\tW$ the associated extended affine Weyl group, and by $X_*$ the translation lattice of $\tW$.

\begin{definition}
\begin{altenumerate}
\item
A \emph{Tits datum (over $\brF$)} is a pair $(\tilde{\Delta}, \lambda)$, where $\tilde{\Delta}$ is a local Dynkin diagram, and $\lambda$ is a $W_0$-conjugacy class in $X_*$.
\item
A \emph{Coxeter datum} (over $\brF$) is a pair $((W_a, \tSS), \lambda)$, where $(W_a, \tSS)$ is an affine Coxeter system and $\lambda$ is a $W_0$-conjugacy class in $X_*$.
\end{altenumerate}
\end{definition}

A Tits datum yields a Coxeter datum by forgetting the arrows in the Dynkin diagram. In general, different Tits data may give rise to the same Coxeter datum. However, in type $A$ and, more generally, for any simply laced Dynkin diagram, the Coxeter datum determines the Tits datum uniquely.

We need to generalize this notion as follows, to cover also the situation over $F$. Over $\brF$, simple adjoint groups are classified up to isomorphism by their (absolute) local Dynkin diagram; cf.~the table in~\cite[\S 4.2]{Tits:Corvallis}. Over $F$, we need to take into account the case of groups which are not residually split. In~\cite[\S 4.3]{Tits:Corvallis}, Tits gives the classification in terms of the ``local index'' and ``relative local Dynkin diagram''. Here we choose to work instead with the absolute local Dynkin diagram (i.e., the affine Dynkin diagram attached to $\bG$ over $\brF$) together with the diagram automorphism induced by Frobenius. This datum is determined by $\bG/F$ (up to isomorphism), and determines the group $\bG$ over $F$ up to isomorphism.

\begin{definition}\label{def-tits-data}
\begin{altenumerate}
\item
    A \emph{Tits datum over $F$} is a triple $(\tilde{\Delta}, \delta, \lambda)$, where $\tilde{\Delta}$ is an absolute  local Dynkin diagram,  and $\delta$ is a diagram automorphism of $\tilde{\Delta}$, and $\lambda$ is a $W_0$-conjugacy class in the coweight lattice $X_*$ of $\tilde{\Delta}$.
\item
A \emph{Coxeter datum over $F$} is a tuple $((W_a, \tSS), \delta, \lambda)$, where $(W_a, \tSS)$ is an affine Coxeter system,  and $\delta$ is a length preserving automorphism of $W_a$, and $\lambda$ is a $W_0$-conjugacy class in $X_*$.
\end{altenumerate}
\end{definition}
Note that a Tits datum over $F$ gives rise to a Coxeter datum over $F$. In~\cite{HPR}, the notion of \emph{enhanced} Tits and Coxeter data was used, where an enhanced datum in addition specifies a parahoric level structure. Note that for an enhanced Coxeter datum $((W_a, \tSS), \lambda, K)$ in the sense of~\cite{HPR}, the associated parahoric subgroup is the one generated by the Iwahori and all simple affine reflections which are \emph{not} contained in $K$, a convention opposite to the one used in this text.

Next we explain the notion of restriction of scalars of Dynkin types over $F$ (i.e., Dynkin types together with a diagram automorphism) along an unramified field extension. It models the form of the extended affine Weyl group of a group which arises as such a restriction of scalars. Let $F_d/F$ denote the unramified extension of degree $d$, and let $(\tilde{\Delta}, \delta_d)$ be a local Dynkin diagram with diagram automorphism $\delta_d$. We then define $\Res_{F_d/F}(\tilde{\Delta}, \delta_d)$ as the Dynkin type
\[
    \tilde{\Delta}_1 \times \cdots \times \tilde{\Delta}_d
\]
with diagram automorphism $\delta$ where $\tilde{\Delta}_i = \tilde{\Delta}$ for all $i$, and $\delta$ is given by $\id\colon \tilde{\Delta}_i \to \tilde{\Delta}_{i+1}$ for $i=1, \dots, d-1$ and $\delta_d \colon \tilde{\Delta}_d \to \tilde{\Delta}_{1}$. So $\delta$ permutes the components cyclically, and the restriction of $\delta^d$ to any component is equal to $\delta_d$.

Specifying a translation element for $\Res_{F_d/F}(\tilde{\Delta}, \delta_d)$ amounts to giving a tuple $(\lambda_1, \dots, \lambda_d)$ consisting of $d$ translation elements for $\tilde{\Delta}$. It is central (or minuscule), if and only if all the $\lambda_i$ are central (or minuscule, respectively).

\begin{example}[The Lubin-Tate case]\label{ex-lubin-tate} This is the case with Tits datum $(\tilde{A}_{n-1}, \id, \omega_1^\vee)$. The corresponding group is $\GL_n$. This is a fully Hodge-Newton decomposable case (Section~\ref{s:fullHN}), and is even of Coxeter type in the sense of~\cite{GH} (and in this case the Coxeter property holds for arbitrary parahoric level). See Section~\ref{ss:LT} for a discussion of this case as a ``minimal dimension'' case. See Section~\ref{s:lattice-interpretation-LT} for a ``lattice description'' of the Lubin-Tate case.

    Similarly, we have the \emph{extended Lubin-Tate case} $(\Res_{F_d/F}(\tilde{A}_{n-1}, \id), (\omega_1^\vee, 0, \dots, 0))$. 
\end{example}

\begin{example}[The Drinfeld case]\label{ex-drinfeld}
    Here we consider the Tits datum $(\tilde{A}_{n-1}, \varrho_{n-1}, \omega_1^\vee)$, where $\varrho_{n-1}$ denotes rotation by $n-1$ steps, $\varrho_{n-1}(s_0) = s_{n-1}$, etc. The corresponding algebraic group is the group of units of a central division algebra of invariant $1/n$. This is a fully Hodge-Newton decomposable case (Section~\ref{s:fullHN}), and even a ``Coxeter'' case (for arbitrary parahoric level). See Section~\ref{s:lattice-interpretation-max}  for a ``lattice description'' of the Drinfeld case.

    Similarly, we have the \emph{extended Drinfeld case} $(\Res_{F_d/F}(\tilde{A}_{n-1}, \varsigma_1), (\omega_1^\vee, 0, \dots, 0))$. 
\end{example}

\subsection{Reduction to $\brF$-simple groups}\label{subsec:ghn-2.4-construction}

Let us recall the construction of~\cite{GHN} 3.4. Given an $F$-simple group $\bG$ of adjoint type together with a conjugacy class $\mu$ of cocharacters, we can decompose
\[
    \bG_{\brF} = \bG_1 \times \cdots \times \bG_d,
\]
where the $\bG_i$ are simple algebraic groups over $\brF$ and where the Frobenius $\s$ induces maps $\bG_i\to \bG_{i+1}$ (with indices viewed in $\mathbb Z/d$). Let $F_d$ denote the unramified extension of $F$ of degree $d$ in $\breve F$. We denote by $\bG'$ the algebraic group over $F_d$, with $(\bG')_{\brF} = \bG_1$, with Frobenius given by $(\s^d)_{|G_1}$. In other words, we write $\bG=\Res_{F_d/F}(\bG')$ for a quasi-simple group over $F_d$ which stays quasi-simple over $\breve F$. Correspondingly, the Tits datum of $\bG$ arises by restriction of scalars along $F_d/F$ as defined in Section~\ref{subsec:tits-data}.

We also define $\mu' = \sum_{i=1}^d \s_0^i(\mu_+)$, where $\s_0$ denotes the $L$-action (cf.~\cite{GHN} Def.~2.1), i.e., the Frobenius action corresponding to the quasi-split inner form of $\bG$.

Now suppose that $K = (K_1, \dots, K_d)$ is an $F$-rational parahoric level structure for $\bG$. Then $K_1$ is a $F_d$-rational parahoric level structure for $\bG'$.

We now consider the special situation that $\mu = (\mu_1, \dots, \mu_d)$ is a conjugacy class of cocharacters of $\bG$ where $\mu_i$ is central for all $i > 1$. Let $\t = (\t_1, \dots, \t_d)$ be a $\s$-conjugacy class in $\BGmu$; we may choose $\t_i$ central for all $i>1$.
Let $\t'=\Pi \t_i$ (this is well-defined as only one of $\t_i$ is noncentral).

Then it is easy to see that projection to the first factor induces an isomorphism $X^{\bG}(\mu, \t)_K \cong X^{\bG'}(\mu', \t')_{K_1}$. Examples of this situation are the extended Lubin-Tate case and the extended Drinfeld case mentioned in the examples above.

Moreover, if $K' = (K'_1, \dots, K'_d)$ is another $F$-rational parahoric level and $K\subseteq K'$, then we likewise have $X^{\bG}(\mu, \t)_{K'} \cong X^{\bG'}(\mu', \t')_{K'_1}$ and we obtain a commutative diagram
\[
    \xymatrix{X^{\bG}(\mu, \t)_K \ar[d] \ar[r]^-{\cong}& X^{\bG'}(\mu', \t')_{K_1} \ar[d] \\
        X^{\bG}(\mu, \t)_{K'} \ar[r]^-{\cong} & X^{\bG'}(\mu', \t')_{K'_1},
}
\]
where the vertical maps are the natural projections.

\section{Fully Hodge-Newton decomposable case}\label{s:fullHN}
\subsection{The $\s$-support}
For $w \in W_a$, we denote by $\supp(w)$ the support of $w$, i.e., the set of $i \in \tSS$ such that $s_i$ appears in some (or equivalently, every) reduced expression of $w$. For any length preserving automorphism $\theta$ of $\tW$, we set 
\begin{equation}\label{def-sigma-support}
\supp_\theta(w \t)=\bigcup_{n \in \BZ} (\Ad(\t)\circ \theta)^n(\supp(w)).
\end{equation} 
This applies in particular to the Frobenius action $\sigma$. Then $\supp_\s(w \t)$ is the minimal $\Ad(\t)\s$-stable subset $J$ of $\tSS$ such that $w \t \s\in W_J \rtimes \<\t \s\>$.

\subsection{Definition and classification of fully Hodge-Newton decomposable pairs $(\bG, \mu)$} In \cite{GHN}, the notion of fully Hodge-Newton decomposable pair $(\bG, \mu)$ is introduced. We refer to~\cite[Def.~3.1]{GHN} for the definition. Here, we use the following equivalent characterizations \cite[Thm. B, Thm.~3.3]{GHN}. 
\begin{theorem}\label{GHN-2.3}
    Let $(\bG, \mu)$ be a pair  as above, with $\bG$ quasi-simple over $F$, and let $K\subset\tSS$ with $\s(K)=K$ and $W_K$ finite.  The following are equivalent:
\begin{enumerate}
\item
The pair $(\bG, \mu)$ is {\it fully Hodge-Newton decomposable}.
\item
For each $w \in \Admmu$, there exists a unique $[b] \in \BGmu$ such that $\brI w \brI \subset [b]$.
\item
For each $w\in \KAdmmu$ with $X_{K, w}(\t)\ne \emptyset$, the set $W_{\supp_\s(w)}$ is finite. 

{\rm Here $\t$ denotes a representative of the unique basic element $[\t]$ in $\BGmu$.}
\end{enumerate}
In particular, condition~(3) is independent of $K$.\qed
\end{theorem}

In particular, in this case, for any $K \subset \tSS$ with $W_K$ finite and any $w \in \KAdmmu$, there exists a unique $[b] \in \BGmu$ such that $\brK \cdot_\s \brI w \brI \subset [b]$. This gives us a natural map 
\begin{equation}\label{mapHN}
\KAdmmu \to \BGmu, \quad w\mapsto [w] .
\end{equation}

We will later use the following statement. 
\begin{proposition}[{\cite[Prop.~5.6, Lemma 5.8]{GHN}}]\label{GHN-4.6}
Let $x\in \tW$. The following are equivalent:
\begin{enumerate}
\item
 $\brK \cdot_\s \brI x \brI \subset [\t]$,
\item
$\k(x)=\k(\t)$ and $W_{\supp_{\s}(x)}$ is finite.
\item
$\k(x)=\k(\t)$ and  $\Ad(x)\circ \s$ fixes a point in the closure of the base alcove.\qed
\end{enumerate}
\end{proposition}

In the next two theorems, we give the classification of the fully Hodge-Newton decomposable cases following~\cite{GHN} Theorem~3.5.

\begin{theorem}\label{f-HN}
Assume that $\bG$ over $F$ is absolutely simple and that $\mu$ is not central. Then $(\bG, \mu)$ is fully Hodge-Newton decomposable if and only if the associated Tits datum is one of the following:
\renewcommand{\arraystretch}{1.7}
\[\begin{tabular}{|c|c|c|}
\hline
$(\tilde A_{n-1}, \id,  \o^\vee_1)$ & $(\tilde A_{n-1}, \varrho_{n-1}, \o^\vee_1)$ & $(\tilde A_{n-1}, \varsigma_0, \o^\vee_1 )$
\\
\hline
$(\tilde A_{2m-1}, \varrho_1 \varsigma_0, \o^\vee_1)$ & $(\tilde A_{n-1}, \id, \o^\vee_1+\o^\vee_{n-1})$ & 
\\
\hline
$(\tilde A_3, \id, \o^\vee_2)$ & $(\tilde A_3, \varsigma_0, \o^\vee_2)$ & $(\tilde A_3, \varrho_2, \o^\vee_2)$
\\
\hline
\hline
$(\tilde B_n, \id, \o^\vee_1)$ & $(\tilde B_n, \Ad(\t_1), \o^\vee_1)$ &
\\
\hline\hline
$(\tilde C_n, \id, \o^\vee_1)$ & $(\tilde C_2, \id, \o^\vee_2)$ & $(\tilde C_2, \Ad(\t_2), \o^\vee_2)$
\\
\hline\hline
$(\tilde D_n, \id, \o^\vee_1)$ &
$(\tilde D_n, \varsigma_0, \o^\vee_1)$ & 
\\
\hline
\end{tabular}
\]
\end{theorem}

\begin{theorem}\label{f-HN-2}
    Assume that $\bG$ is quasi-simple over $F$ and that $\mu$ is not central. Then the pair $(\bG, \mu)$ is fully Hodge-Newton decomposable if and only if the associated Tits datum is of type $(\Res_{F_d/F}(\tilde{\Delta}, \delta), (\mu_1, \dots, \mu_d))$ where one of the following two possibilities occur. 
    \begin{enumerate}
        \item 
            There is a unique $i$ such that $\mu_i$ is non-central, and $(\tilde{\Delta}, \delta, \mu_i)$ is one of the triples listed in Theorem~\ref{f-HN}, or
        \item
            $(\tilde{\Delta}, \delta) = (\tilde{A}_{n-1}, \id)$ and there exist $i \ne i'$ such that  $\mu_i = \omega_1^\vee$, $\mu_{i'} = \omega_{n-1}^\vee$, and $\mu_j$ is central for all $j\ne i, i'$.
    \end{enumerate}
\end{theorem}

Here we use the same labeling of the Coxeter graph as in Bourbaki \cite{Bour}. If $\o^\vee_i$ is minuscule, we denote the element  $\tau(t^{\omega^\vee_i})\in \Omega$ by $\t_i$; conjugation by $\t_i$ is a length preserving automorphism of $\tW$ which we denote by $\Ad( \t_i)$. For type $A_n$, $\Ad(\t_i)$ is the rotation of the affine Dynkin diagram by $i$ steps (i.e., $s_0$ is mapped to $s_i$, $s_1$ is mapped to $s_{i+1}$, etc.), and we denote it by $\varrho_i$ instead.  Let $\varsigma_0$ be the unique nontrivial diagram automorphism for the finite Dynkin diagram if $W_0$ is of type $A_n, D_n$ (with $n \ge 5$) or $E_6$. For type $D_4$, we also denote by $\varsigma_0$ the diagram automorphism which interchanges $\a_3$ and $\a_4$.

If we assume that $\mu$ is non-central in every component of the affine Dynkin diagram, the fully Hodge-newton decomposable cases are the cases in Theorem~\ref{f-HN} and the Hilbert-Blumenthal case $(\tilde A_{n-1} \times \tilde A_{n-1}, {}^1 \varsigma_0, (\o^\vee_1, \o^\vee_{n-1}))$, where the automorphism ${}^1 \varsigma_0$ on $\tilde A_{n-1} \times \tilde A_{n-1}$ is the automorphism which exchanges the two factors.

To derive Theorem~\ref{f-HN-2} from Theorem~\ref{f-HN}, note that for a group $\bG$ which is quasi-simple over $F$, but not over $\brF$, we can apply the construction in~\cite[Section 3.4]{GHN}, cf.~Section~\ref{subsec:ghn-2.4-construction}. We then have that $\bG'$ is quasi-simple over $\brF$, and that $\mu$ is minute if and only if $\mu'$ is minute (comp.~\cite{GHN} Def.~3.2 and Section 3.4). Applying Theorem~\ref{f-HN} to $(\bG', \mu')$, one obtains Theorem~\ref{f-HN-2}.

\subsection{Basic case} Let $\t=\t(\mu)\in\Omega$ be the length-$0$ element in $\tW$ such that $\Admmu \subset W_a \t$. Then $[\t]$ is the unique basic $\s$-conjugacy class in $\BGmu$. 

Set 
\begin{equation}\label{adm0}
\KAdmmu_0=\{w \in \KAdmmu\mid W_{\supp_\s(w)} \text{ is finite}\}.
\end{equation}
If $(\bG, \mu)$ is fully Hodge-Newton decomposable,  the set $\KAdmmu_0$ is just the fiber over the unique basic element of $\BGmu$ of the map \eqref{mapHN}.

The following result is proved in \cite{GHN}.

\begin{theorem}\label{fine-dec}
Suppose that $(\bG, \mu)$ is a fully Hodge-Newton decomposable pair. Then $$\XmubK{\t}{K}=\bigsqcup_{w \in \KAdmmu_0} X_{K, w}(\t),$$
and $X_{K, w}(\t)\ne \emptyset$ for all $w \in \KAdmmu_0$.
\qed
\end{theorem}
\part{Minimal dimension}
In this part we determine those cases when $\XmubK{b}{K}$ is zero-dimensional, in case that $b$ is basic. When $b$ is basic, we also determine the cases when the transition morphism $\XmubK{b}{K}\to \XmubK{b}{K'}$ has finite fibers. 
\section{Statement of results}\label{s:reszero}

\subsection{Change of parahoric} In this section, we are concerned with the following two theorems.

\begin{theorem}\label{thm-dimension-0}
Assume that $\bG$ is quasi-simple over $F$ and that $\mu$ is not central.  Let $K \subsetneqq \tSS$ be $\s$-stable. The following are equivalent:
\begin{enumerate}
\item
$\dim \XmubK{\t}{K}=0$.
\item
$(\bG, \mu)$ is of extended Lubin-Tate type, i.e., $(\tilde{\Delta}, \s, \mu) = (\Res_{F_d/F}(\tilde{A}_{n-1}, \id), (\omega_1^\vee, 0, \dots, 0))$ for a finite unramified extension $F_d/F$.
\end{enumerate}
\end{theorem}

See Example~\ref{ex-lubin-tate} for a discussion of the (extended) Lubin-Tate case.
We will prove a stronger version of this theorem below, see Theorem~\ref{thm-all-k}.  

For any $\s$-stable subsets $K \subsetneqq K' \subset \tSS$, we denote by 
\begin{equation}
\pi_{K, K'}: \XmubK{\t}{K} \to \XmubK{\t}{K'}
\end{equation}
the projection map.

\begin{theorem}\label{thm-fixed-s}
    Assume that $\bG$ is quasi-simple over $F$ and that $\mu$ is not central. Let
    $K \subsetneqq K' \subsetneqq \tSS$ be $\s$-stable parahoric level
    structures. Write the Tits datum of $(\bG, \mu)$ in the form  $(\Res_{F_d/F}(\tilde{\Delta}, \s),
    (\mu_1, \dots, \mu_d))$, and correspondingly write the parahoric level structures as $K=(K_1, K_2, \dots,
    K_d)$, $K'=(K'_1, K'_2, \dots, K'_d)$.
    Then the following are equivalent:
\begin{enumerate}
	\item
The projection $\XmubK{\t}{K} \to \XmubK{\t}{K'}$ has discrete fibers. 
	\item
        There exists a unique $j$ such that $\mu_j$ is non-central, we have
        $\mu_j = \omega_1^\vee$, and
     \begin{itemize}
     	\item $\s$ acts as $\id$ on the affine Dynkin diagram, or 
        \item $n \ge 3$ and the action of $\s$  on $\tilde{A}_{n-1}$ preserves
            $s_0$ and induces the nontrivial diagram automorphism $\varsigma_0$
            on ${A}_{n-1}$. Furthermore, the pair $(K_1, K'_1)$ satisfies
            Condition \ref{star}.  
     \end{itemize}
\end{enumerate}
\end{theorem}
Here is the Condition \ref{star} that appears in (2).
\begin{condition}\label{star}
    \emph{Every element of $K'_1\setminus K_1$ is fixed by $\s_d$, and if $s_i \in
    K_1'\setminus K_1$, then $s_{i+1} \notin K_1$.}
\end{condition}

Note that $K$ and $K'$ are assumed to be $\s$-stable, so requiring that the inclusion $K'\subsetneqq \tSS$ be strict implies that in each connected component of $\tSS$ there extists a vertex not lying in $K'$, and similarly for the inclusion $K\subsetneqq K'$.

\begin{remark}
Let us enumerate the cases for the second alternative in
Theorem~\ref{thm-fixed-s}, (2) when $\bG$ is quasi-simple over $\brF$. By
assumption $K$ and $K'$ are $\s$-stable; also, the corresponding algebraic group
is a quasi-split unitary group which splits over an unramified quadratic
extension.
\begin{altitemize}
\item {\it $n$ odd}: In this case $\s(s_0)=s_0$ and $\s(s_1)=s_{n-1}$. Then $K\subset \tSS\setminus \{s_0, s_1, s_{n-1}\}$ is $\s$-stable and $K'=K\cup\{s_0\}$. 

\smallskip
\noindent{\it Extreme case $n=3$}; then $K=\emptyset$, $K'=\{s_0\}$. 
\item{\it $n=2m$ even}: In this case $\s(s_0)=s_0, \s(s_m)=s_m $ and $\s(s_{m+1})=s_{m-1}$. Then the following three possibilities occur. 

\noindent (i) $K\subset \tSS\setminus \{s_0, s_1, s_{n-1}\}$ is $\s$-stable and $K'=K\cup\{s_0\}$.

\noindent (ii) $K\subset \tilde\BS\setminus \{s_{m-1}, s_m, s_{m+1}\}$ is $\s$-stable and $K'=K\cup\{s_m\}$.

\noindent (ii) $K\subset \tilde\BS\setminus \{s_0, s_1, s_{m-1}, s_m, s_{m+1}, s_{n-1}\}$ is $\s$-stable and $K'=K\cup\{s_0, s_m\}$.

\smallskip

\noindent{\it Extreme case $n=4, m=2$}; then for $(K, K')$ the following possibilities occur:  $(\emptyset,\{s_0\})$,  or $(\emptyset, \{s_2\})$, or $(\emptyset,\{s_0, s_2\})$, or $(\{s_2\}, \{s_0, s_2\})$, or $(\{s_0\},\{s_0, s_2\})$. 

\end{altitemize}
\end{remark}
The proof of Theorem~\ref{thm-fixed-s} will occupy the next two sections. In the rest of this section, we give more details on the two alternatives of the above theorem. 
\subsection{The Lubin-Tate case}\label{ss:LT}

\begin{theorem}\label{thm-all-k}
	Assume that $\bG$ is quasi-simple over $F$ and that $\mu$ is not central. The following are equivalent:
	\begin{enumerate}
		\item
            The pair $(\bG, \mu)$ is of extended Lubin-Tate type (cf.~the statement of Theorem~\ref{thm-dimension-0}).
		\item
$\dim \XmubK{\t}{K}=0$ for some parahoric $K$.
		\item
$\dim \XmubK{\t}{K}=0$ for all parahorics $K$.
		\item
The projection $\XmubK{\t}{K} \to \XmubK{\t}{K'}$ has finite fibers for all   $K\subsetneqq K'$.
\item The projection $\XmubK{\t}{K} \to \XmubK{\t}{K'}$ is a bijection  for all  $K\subsetneqq K'$.
	\end{enumerate}
\end{theorem}

\begin{proof}
	$(3) \Rightarrow (2)$ and $(5)\Rightarrow (4)$ are obvious. 

    $(1)\Rightarrow (3) \& (5)$:  This follows from Remark~\ref{rmk-lubin-tate} below.

    $(2) \Rightarrow (1)$: This is Theorem \ref{thm-dimension-0}.

    $(4) \Rightarrow (1)$: By Theorem \ref{thm-fixed-s},  the Dynkin type is $ \Res_{F_d/F}(\tilde{A}_{n-1}, \s_d)$, with $\sigma_d=\id$ or $\s_d=\varsigma_0$ (up to isomorphism). Moreover, as we may take $K=\{s_0\}$, Condition \ref{star} implies that $\s_d$   cannot be $\varsigma_0$. Hence $\s=\id$.
	
\end{proof}

\begin{remark}\label{rmk-lubin-tate}
    Properties (3) and (5) in Theorem \ref{thm-all-k} are well-known in the Lubin-Tate case, and we explain this in terms of lattices in Section~\ref{s:lattice-interpretation}. Alternatively, one can apply the methods of~\cite{GH} (Section 6.3, Case 1 for $i=1$), cf.~also~\cite{GHN}. There is only one basic EKOR stratum in this case. (Note that EKOR strata were called EO strata in~\cite{GH}.) Let ${\bf J}={\bf J}_\tau$ be the $\s$-centralizer of $\tau$, cf. \eqref{scentrali}. The index set for the
stratification in a single connected component is a quotient of ${\bf J}(F)^1$ by a parahoric subgroup (where ${\bf J}(F)^1$ is the kernel of the Kottwitz homomorphism). Since ${\bf J}(F)^1$ is anisotropic, this quotient is a single point, so the EKOR stratification has  a single stratum. This stratum is attached to the length $0$ element $\tau$, thus the corresponding classical Deligne-Lusztig variety is just a point. Note that this argument can be applied to arbitrary parahoric level structure, not only maximal parahoric as in the setting of~\cite{GH}. By either of the two methods, we obtain the more precise statement that $\XmubK{\t}{K}$ has only one point in each connected component of the affine flag variety.

Using the construction in Section~\ref{subsec:ghn-2.4-construction}, the result can be generalized to the extended Lubin-Tate case, where a restriction of scalars is allowed.
\end{remark}

\subsection{The exotic case}

The second alternative in Theorem \ref{thm-fixed-s}, where Condition~\ref{star} is relevant, will be studied in detail in Section \ref{sec:fibers-group-theoretic} in group-theoretic terms, and in Section~\ref{sec:exotic} in terms of lattices. Using either approach, we will determine the cardinalities of the fibers of the map $\pi_{K, K'}$. If $\#(K_1'\setminus K_1) = 1$, then the fiber cardinalities are $1$, $2$, and $q^d+1$. If $\#(K_1'\setminus K_1) = 2$, then each fiber is naturally a product of two sets as in the first case, so the cardinalities which occur are $1$, $2$, $4$, $q^d+1$, $2(q^d+1$), and $(q^d+1)^2$. We give precise criteria in group-theoretic terms as well as in lattice terms which case occurs when, see Section~\ref{exotic-case-gp-thy} and Proposition ~\ref{fibers-lattice-descr}.

\section{Proof of $(1) \Rightarrow (2)$ in Theorems~\ref{thm-dimension-0} and~\ref{thm-fixed-s}}\label{s:onedirzero}

In this section, we prove the implications $(1) \Rightarrow (2)$ in
Theorem~\ref{thm-dimension-0} and Theorem~\ref{thm-fixed-s}.  We will handle both
theorems simultaneously by allowing $K=\tSS$, with the convention that
$\XmubK{\t}{\tSS}=\breve G/\breve G$ is a single point. Hence the condition
that the map $\pi_{K, \tSS}$ has discrete fibers is equivalent to the condition
that $\dim \XmubK{\t}{K}=0$.  

We assume that $\mu$ is not central.

\subsection{Preparations}
We start with some properties of the admissible set. 

\begin{lemma}\label{s} \cite[Lem. 6.5]{HeZhou}
	For any $s \in \tSS$, $s \t \in \Admmu$. \qed
\end{lemma}

\begin{proposition}\label{supp-tss}
    Suppose that $\bG$ is quasi-simple over $\brF$ and that $\mu$ is non-central. If $(\tilde{\Delta}, \mu) \neq (\tilde A_{n-1}, \omega^\vee_1)$ or $(\tilde A_{n-1}, \omega^\vee_{n-1})$ for some $n$, then there exists $w \in \Admmu$ such that $\supp(w \t \i)=\tSS$.
\end{proposition}

\begin{remark}
(1) Note that $(\tilde A_{n-1}, \omega^\vee_1)$ and $(\tilde A_{n-1}, \omega^\vee_{n-1})$ are isomorphic. Below, we often only mention one of these two isomorphic pairs.

(2) In Theorem \ref{critadm}, we will prove a stronger statement by a different method. We decided to keep the present proof, because it is simpler and uses only the combinatorics of the affine Weyl group.
\end{remark}
\begin{proof}
	Let $w_0$ be the longest element in $W_0$ and $K=\{s \in \BS\mid s w_0(\underline \mu)=w_0(\underline \mu)\}$. By \cite[Thm. 2.2]{HL}, we have $\ell(w_K w_0 t^{w_0(\underline \mu)})=\ell(t^{w_0(\underline \mu)})-\ell(w_K w_0)$ and $w_K w_0 t^{w_0(\underline \mu)} \in {}^{\BS} \tW$. Here $w_K$ denotes the longest element in $W_K$. Then we have $$\supp(t^{w_0(\underline \mu)} \t \i)=\supp(w_K w_0) \cup \supp(w_K w_0 t^{w_0(\underline \mu)} \t \i).$$
	
	Since $\umu$ is non-central, we have $K \subsetneqq \BS$ and thus $\supp(w_K w_0)=\BS$. If $\umu$ is non minuscule, we have $w_K w_0 t^{w_0(\underline \mu)} \t \i \neq 1$. Since $w_K w_0 t^{w_0(\underline \mu)} \in {}^{\BS} \tW$, we have $\tSS\setminus \BS \subset \supp(w_K w_0 t^{w_0(\underline \mu)} \t \i)$. Thus $\supp(t^{w_0(\underline \mu)} \t \i)=\tSS$.
	
	Now we assume that $\umu$ is minuscule. Then $t^{w_0(\underline \mu)} =w_K w_0 \t$. Moreover, $K=\BS\setminus \{s\}$ for some $s \in \BS$ which corresponds to an endpoint of the Dynkin diagram of $\BS$. Let $s'=w_0 s w_0 \in \BS$ and let $s_0$ be the unique element in $\tSS\setminus \BS$. Then we have $$t^{s_0 w_0(\underline \mu)}=s_0 w_K w_0 s' \t.$$ 
	
	If $\tW$ is of type $\tilde A_{n-1}$ and $\mu \notin \{\o^\vee_1,\o^\vee_{n-1}\}$ then, by direct computation, $\supp(w_K w_0 s')=\BS$ and thus $\supp(t^{s_0 w_0(\underline \mu)}\t\i)=\tSS$. If $\tW$ is not of type $\tilde A$ then, by the explicit formula for the reduced expressions of $w_K w_0$ given in \cite[\S 1.5]{He-0}, we still have $\supp(w_K w_0 s')=\BS$ and $\supp(t^{s_0 w_0(\underline \mu)}\t\i)=\tSS$.
\end{proof}

\begin{lemma}\label{forA}
	Let $\tW$ be the Iwahori-Weyl group of type $\tilde A_{n-1}$. If $\mu$ is non-central, and not equal to $\o^\vee_1$ or $\o^\vee_{n-1}$, then for any $s, s' \in \tSS$, $s s' \t \in \Admmu$. 
\end{lemma}

\begin{proof}
	If $s$ commutes with $s'$, then by Proposition \ref{supp-tss}, there exists $w \in \Admmu$ such that $s, s' \in \supp(w \t \i)$ and hence $s s' \le w \t \i$. So $s s' \t \le w$ and $s s' \t \in \Admmu$. 
	
	Let $\t_1$ be the automorphism of $\tW$ sending $s_0$ to $s_1$, $s_1$ to $s_2$, \ldots, $s_{n-1}$ to $s_0$. Then the conjugation action of $\t_1$ preserves $\mu$ and we have that $\t_1 \Admmu \t_1 \i=\Admmu$. Since $\t_1$ acts transitively on $\tSS$, it suffices to show that there exists $j$ with $0 \le j \le n-1$ such that $s_j s_{j+1} \t, s_{j+1} s_j \t \in \Admmu$. Here by convention, we set $s_n=s_0$. 
	
	Let $\k: \tW \to \BZ/n\BZ$ be the Kottwitz map, cf. \eqref{kotmap}. Let $i=\k(\mu)$. If $i \notin \{0, 1, n-1\}$, then $\mu_+ \ge \omega^\vee_i$. By direct computation, $s_0 s_1 \t, s_1 s_0 \t \le t^{\omega^\vee_i}$ and hence $s_0 s_1 \t, s_1 s_0 \t \in \Adm(\omega^\vee_i) \subset \Admmu$. 
	
	If $i=0$, then $\mu_+ \ge \omega^\vee_1+\omega^\vee_{n-1}$. By direct computation, $s_1 s_2\t, s_2 s_1\t \le t^{\omega^\vee_1+\o^\vee_{n-1}}$ and hence $s_1 s_2\t, s_2 s_1\t \in \Adm(\omega^\vee_1+\o^\vee_{n-1}) \subset \Admmu$. 
	
	If $i=1$ and $\mu_+ \neq \omega^\vee_1$, then $\mu_+ \ge \omega^\vee_2+\omega^\vee_{n-1}$. By direct computation, $s_0 s_1 \t, s_1 s_0 \t \le t^{\omega^\vee_2+\o^\vee_{n-1}}$ and hence $s_0 s_1 \t, s_1 s_0 \t \in \Adm(\omega^\vee_2+\o^\vee_{n-1}) \subset \Admmu$. 
	
	If $i=n-1$, and $\mu_+ \neq \o^\vee_{n-1}$, then $\mu_+ \ge \o^\vee_1+\o^\vee_{n-2}$. By direct computation, $s_0 s_1 \t, s_1 s_0 \t \le t^{\omega^\vee_1+\o^\vee_{n-2}}$ and hence $s_0 s_1 \t, s_1 s_0 \t \in \Adm(\omega^\vee_1+\o^\vee_{n-2}) \subset \Admmu$. 
\end{proof}

\begin{proposition}\label{prop-fixed-s}
Let $K \subsetneqq K' \subseteq \tSS$ be $\s$-stable. If  $s\t\s(s) \in\Admmu$  for some $s\in K'\setminus K$, then the projection $\pi_{K, K'}: \XmubK{\t}{K} \to \XmubK{\t}{K'}$ has non-discrete fibers.
\end{proposition}

\begin{proof}
	 Let $\brK_s$ be the standard parahoric subgroup generated by $\brI$ and $s$. We then have
	\[
	\brK_s \cdot_\s \brI\t \brI \subseteq \brI s\brI\t \brI\s(s) \brI \subseteq \brI\t \brI \cup \brI s\t \brI \cup \brI\t\s(s) \brI \cup \brI s\t\s(s) \brI \subseteq \brK \Admmu\brK.
	\]
By definition, $\t \in \Admmu$. By Lemma \ref{s}, $s \t, \t \s(s) \in \Admmu$. By assumption, $s \t \s(s) \in \Admmu$. Hence $\brK_s \brK/\brK \subseteq \XmubK{\t}{K}$, and this is a subset of dimension $1$ which maps to a point in $\XmubK{\t}{K'}$.	 
\end{proof}

\subsection{Reduction to the case where $\bG$ is quasi-simple over $\brF$} 	
From now on we assume that condition~(1) in either Theorem~\ref{thm-dimension-0}
or Theorem~\ref{thm-fixed-s} holds for $K \subsetneqq K' \subseteq \tSS$. We may
assume that $\bG$ is adjoint, so we can write $\bG_{\brF} =G_1 \times \cdots
G_d$ for $\brF$-simple groups $\bG_i$.

Correspondingly, $\tW$ is of the form
\[
    \tW=\tW_1 \times \tW_2 \times \cdots \times \tW_d,
\]
where $\tW_1 \cong \tW_2 \cong \cdots \cong \tW_m$ are the extended affine Weyl groups with connected Dynkin diagram. Since $\bG$ is quasi-simple over $F$, we have (up to renumbering, if necessary) $\s(\tW_1)=\tW_2, \ldots, \s(\tW_{d-1})=\s(\tW_d), \s(\tW_d)=\s(\tW_1)$. 

Write $\mu=(\mu_1, \ldots, \mu_d)$ and $\t=(\t_1, \ldots, \t_d)$. Since by assumption $\mu$ is non-central, at least one of the $\mu_i$ is non-central in $\tW_i$. Suppose that there is more than one non-central $\mu_i$. Without loss of generality, we may assume that $\mu_1$ is noncentral in $\tW_1$ and that $i$ is the smallest positive integer $>1$ such that $\mu_i$ is noncentral in $\tW_i$. Then $\Ad(\t_j)$ is the identity group automorphism on $\tW_j$ for $1<j<i$. 

Let $s$ be a simple reflection of $\tW_1$ that is contained in $K' \setminus K$. Let 
\[
    Z=\{(g, \s(g), \ldots, \s^{i-2}(g), 1, \ldots, 1)\mid g \in \brK_s\}.
\]
Then $Z \subset \brK'$ and $Z \brK/\brK \subset \brK'/\brK$ is $1$-dimensional. By direct computation, $Z \cdot_\s \t \subset \brI s \t \s^{i-1}(s) \brI$. By Lemma~\ref{s}, $s \t_1 \in \Adm(\mu_1)$ and $\t_i \s^{i-1}(s) \in \Adm(\mu_i)$. Therefore $s \t \s^{i-1}(s) \in \Admmu$. Hence $Z \brK/\brK \subseteq \XmubK{\t}{K}$, and this is a subset of dimension $1$ which maps to a point in $\XmubK{\t}{K'}$.	

It follows  that $\mu_i$ is noncentral $\tW_i$ for a unique $i$, say $i=1$. We can thus carry out the construction in Section~\ref{subsec:ghn-2.4-construction}, and find an algebraic group $\bG'$ over $F_d$ and a commutative diagram
\[
    \xymatrix{X^{\bG}(\mu, \t)_K \ar[d]^-{\pi_{K, K'}} \ar[r]^-{\cong}& X^{\bG'}(\mu', \t')_{K_1} \ar[d] ^-{\pi_{K_1, K'_1}} \\
        X^{\bG}(\mu, \t)_{K'} \ar[r]^-{\cong} & X^{\bG'}(\mu', \t')_{K'_1}.
}
\]
It is then enough to show property~(2) in Theorem~\ref{thm-dimension-0} or Theorem~\ref{thm-fixed-s}, respectively, for the $\brF$-simple group $\bG'$.

\subsection{Reduction  to the $(\tilde A_{n-1}, \o^\vee_1)$ case} Now we assume that $\bG$ is quasi-simple over $\brF$. Let $s \in K'\setminus K$. Suppose that the projection $\pi_{K, K'}: \XmubK{\t}{K} \to \XmubK{\t}{K'}$ has discrete fibers. By Proposition \ref{prop-fixed-s}, we then have $s \t \s(s) \notin \Admmu$. We distinguish cases.

\smallskip

\noindent {\it Case (I): $s$ commutes with $\tau \s(s) \tau \i$. }

\smallskip	
By Proposition \ref{supp-tss}, if $(\tilde{\Delta}, \mu) \neq (\tilde A_{n-1}, \omega^\vee_1)$ or $(\tilde A_{n-1}, \omega^\vee_{n-1})$ for some $n$, then there exists $w \in \Admmu$ with $\supp(w \t \i)=\tSS$. Hence $s \t \s(s) \le w$ and $s \t \s(s) \in \Admmu$,  a contradiction.

\smallskip

\noindent{\it Case (II): $s$ does not commute with $\tau \s(s) \tau \i$. }

\smallskip

Then $\tW$ is of type $\tilde A_n$, $\tilde C_{2 n+1}$ or  $\tilde D_{2n+1}$. 
If $\tW$ is of type $\tilde C_{2n+1}$ or $\tilde D_{2n+1}$, then $\{s, \tau \s(s) \tau \i\}=\{s_n, s_{n+1}\}$. Then by direct computation, $s_n s_{n+1} \tau, s_{n+1} s_n \tau \in \Admmu$ for any minuscule or quasi-minuscule coweight $\mu$. For general $\mu$, there exists a minuscule or quasi-minuscule coweight $\mu'$ such that $\mu \ge \mu'$. Hence $\Adm(\mu') \subset \Admmu$ and $s_n s_{n+1} \tau, s_{n+1} s_n \tau \in \Admmu$, a contradiction. 
If $\tW$ is of type $\tilde A_{n-1}$ but  $\mu_+$ is not $\o^\vee_1$ or $\o^\vee_{n-1}$ then, by Lemma \ref{forA},  $s \t \s(s) \in \Admmu$, a contradiction. 

In summary, we may now assume that $(\tilde{\Delta}, \mu)=(\tilde A_{n-1}, \omega^\vee_1)$. 

\subsection{The $(\tilde A_{n-1}, \o^\vee_1)$ case} If $(\tilde{\Delta}, \mu)=(\tilde A_{n-1}, \omega^\vee_1)$, then $s$ does not commute with $\tau \s(s) \tau \i$.   Indeed, assume that   $s$ does  commute with $\tau \s(s) \tau \i$. The maximal elements in $\Admmu$ are $\tau s_{n-1} s_{n-2} \cdots s_1$, $\tau s_{n-2} s_{n-3} \cdots s_0, \ldots, \tau s_0 s_{-1} \cdots s_{-(n-2)}$. If $s=\t \s(s) \t \i$, then $s \t \s(s)=\t \in \Admmu$, a contradiction to Proposition \ref{prop-fixed-s}. If $s \neq \t \s(s) \t \i$ then, since  $s$ commutes with $\t \s(s) \t \i$, we have $n \ge 3$ and hence $t_1 t_2 \tau \in \Admmu$ for any $t_1, t_2 \in \tSS$ with $t_1 t_2=t_2 t_1$. We again have $s \t \s(s) \in \Admmu$, a contradiction to Proposition \ref{prop-fixed-s}. 

We deduce that  $\s=\id$,  or $\s=\varsigma_0$ (for $n \ge 3$), or $\s= \Ad(\tau_{n-2})$. Now  $\Ad(\t_{n-2})$ acts on the affine Dynkin diagram by sending $s_2$ to $s_0$, $s_3$ to $s_1$, \ldots, $s_1$ to $s_{n-1}$. By direct computation, if $\s=\Ad(\tau_{n-2})$, then $s \tau \s(s) \in \Admmu$, a contradiction to Proposition \ref{prop-fixed-s}. 

If $\s=\varsigma_0$, then $s \tau \s(s) \notin \Admmu$ if and only if $s=s_0$ for $n$ odd and $s=s_0$ or $s=s_m$ for $n=2m$ even. Now assume that $K'\setminus K \subset \{s_0, s_{\frac{n}{2}}\}$, and let us check Condition \ref{star}  on $(K, K')$. We argue by contradiction. 

If $s_0 \in K'\setminus K$ and $s_1 \in K$, then $\brK_{s_0, s_1} \subset \brK'$, where $\brK_{s_0, s_1}$ is the standard parahoric subgroup generated by $\brI$ and $s_0, s_1$. We have $$\brI s_0 \t \brI \subset \brK_{s_0, s_1} \cdot_\s \t \subset \brK' \cdot_\s \t.$$ Since $s_0 \t \in \Admmu$, 
the set
$$\{g \in \brK'/\brK\mid g \i \t \s(g) \in \brK \cdot_\s \brI s_0 \t \brI\}
$$ is a one-dimensional subvariety of $\XmubK{\t}{K}$ in the fiber over $\brK'/\brK' \in \XmubK{\t}{K'}$: contradiction.  

If $n=2m$ is even and $s_m \in K'\setminus K$,  and $s_{m+1} \in K$, then $\brK_{s_m, s_{m+1}} \subset \brK'$ and 
$$\brI s_{m} \t \brI \subset \brK_{s_m, s_{m+1}} \cdot_\s \t \subset \brK' \cdot_\s \t.$$ Since $s_m \t \in \Admmu$, 
the set
$$\{g \in \brK'/\brK\mid g \i \t \s(g) \in \brK \cdot_\s \brI s_{m} \t \brI\}
$$ is a one-dimensional subvariety of $\XmubK{\t}{K}$ in the fiber over  $\brK'/\brK' \in \XmubK{\t}{K'}$: contradiction.

\section{Proof of $(2) \Rightarrow (1)$ in Theorem~\ref{thm-fixed-s}}\label{s:otherdirzero}

Similarly as before, we may assume that $\bG$ is quasi-simple over $\brF$.

\subsection{Compatibility of the map $p_{K, \t}$}
Assume that we are in the following situation
\begin{situation}\label{sitn-thm-4.2}
Let $(\bf G, \mu)$ and $K \subsetneqq K' \subsetneqq \tSS$ be $\s$-stable and such that 
we are in either of the following two cases. 
\begin{itemize}
\item (The Lubin-Tate case) The associated Coxeter datum is isomorphic to $(\tilde{A}_{n-1}, \id, \omega_1^\vee)$, or
\item (The exotic case) The associated Coxeter datum is isomorphic to $(\tilde{A}_{n-1}, \varsigma_0,
\omega_1^\vee)$, $n \ge 3$ and Condition \ref{star} is satisfied. 
\end{itemize}
\end{situation}

Then by Theorem \ref{f-HN}, the pair $(G, \mu)$ is fully Hodge-Newton decomposable.  By Theorem~\ref{fine-dec},
\[
\XmubK{\t}{K}=\bigsqcup_{w \in \KAdmmu_0} X_{K, w}(\t),
\]
and we define the map $p_{K, \t}\colon \XmubK{\t}{K} \to \KAdmmu_0$ by mapping all points in $X_{K, w}(\t)$ to $w$.
We prove the following compatibility result for the maps $p_{K, \t}$ when $K$ varies. 

\begin{theorem}\label{pi'}
Let $({\bf G}, \mu, K \subsetneqq K')$ be as in Situation~\ref{sitn-thm-4.2}. 
	
	There exists a unique map $\pi'_{K, K'}\colon \KAdmmu_0 \to {}^{K'} \Admmu_0$ such that the following diagram commutes
	\[
	\xymatrix{	
\XmubK{\t}{K} \ar[r]^-{p_{K, \t}} \ar[d]_-{\pi_{K, K'}} & \KAdmmu_0 \ar[d]^-{\pi'_{K, K'}} \\
\XmubK{\t}{K'} \ar[r]^-{p_{K', \t}}  & {}^{K'} \Admmu_0 
		,}
	\]
i.e., for each EKOR stratum in $\XmubK{\t}{K}$, the projection to $\XmubK{\t}{K'}$ is a \emph{single} EKOR stratum.
Moreover, the projection map $\pi_{K, K'}: \XmubK{\t}{K} \to \XmubK{\t}{K'}$ has finite fibers.
\end{theorem}

\subsection{Partial conjugation}
To give the definition of $\pi'_{K, K'}$, we use the partial conjugation method. 

Let $w, w' \in \tW$ and $s \in \tSS$. We write $w \xrightarrow{s}_\s w'$ if $w'=s w \s(s)$ and $\ell(w') \le \ell(w)$.
Let $K \subset \tSS$. We write $w \to_{K, \s} w'$ if there exists a sequence $w=w_0, w_1, \ldots, w_n=w'$ such that for any $k$, $w_k \xrightarrow{s}_{\s} w_{k+1}$ for some $s \in K$. We write $w \approx_{K, \s} w'$ if $w \to_{K, \s} w'$ and $w' \to_{K, \s} w$.

\begin{proposition}\label{w-w}
Let $({\bf G}, \mu, K \subsetneqq K')$ be as in Situation~\ref{sitn-thm-4.2}. 
For any $w \in \KAdmmu_0$, there exists a unique $w' \in {}^{K'} \Admmu_0 $ such that $w \approx_{K', \s} w'$. 
\end{proposition}

\begin{proof} The uniqueness of $w'$ follows from \cite[Cor. 2.5]{He-Min}. Now we prove the existence. 
	
	If $\s$ acts as id on the affine Dynkin diagram, then $\KAdmmu_0=\{\t\}$ for any $K$. Now we consider the case where $\s=\varsigma_0$. Note that the maximal elements in $\Admmu$ are 
$$
s_0 s_{n-1} s_{n-2} \cdots s_2 \t, s_1 s_0 s_{n-1} \cdots s_3 \t, \ldots, s_{n-1} s_{n-2} \cdots s_1 \t .
$$
 Therefore 
	\begin{enumerate}
		\item if $w \in \Admmu$, then each simple reflection appears at most once in a reduced expression of $w \t \i$; 
		
		\item for any $0 \le i \le n-1$, $s_i s_{i+1} \t \notin \Admmu$. Here by convention, we set $s_n=s_0$. 
	\end{enumerate}

We consider here the case where $n=2m$ for some $m \ge 2$ and $K'\setminus K=\{s_0, s_m\}$; the other cases follow from a similar (but simpler) argument. 
	Let $w \in \KAdmmu_0$. 
	
	If $s_0 w>w$ and $s_m w>w$, then $w \in {}^{K'} \Admmu_0$ and $w':=w$ is the desired element. If $s_0 w<w$ and $s_m w>w$, then $s_0$ commutes with $s_m$ and $s_m (s_0 w)>s_0 w$. So $s_0 w \in {}^{K'} \tW$. Since $s_0 w<w$ and $w \in \Admmu$, $s_1$ does not occur in any reduced expression of $w \t \i$. Thus $$s_0 w \s(s_0)=s_0 w s_0=s_0 (w \t \i) s_1 \t \in {}^{K'} \tW$$ and has the same length as $w$. Moreover, by \cite[Lemma 4.5]{Haines:Bernstein}, $s_0 w s_0 \in \Admmu$. So $w':=s_0 w s_0$ is the desired element. 
	
	If $s_0 w>w$ and $s_m w<w$, then by a similar argument $s_m w \in {}^{K'} \tW$ and $w':=s_m w s_m \in {}^{K'} \Admmu_0$ is the desired element. If $s_0 w<w$ and $s_m w<w$, then by a similar argument $s_0 s_m w \in {}^{K'} \tW$ and $w':=s_0 s_m w s_m s_0 \in {}^{K'} \Admmu_0$ is the desired element. 
\end{proof}
\begin{proof}[Proof of Theorem \ref{pi'}, existence and uniqueness of $\pi'_{K, K'}$] By Theorem \ref{fine-dec}, we have 
$$\XmubK{\t}{K}=\bigsqcup_{w \in \KAdmmu_0} X_{K, w}(\t),$$ 
and all $X_{K, w}(\t)$ in the union of the right hand side are non-empty. The latter fact says that the map $p_{K, \t}$ is surjective, so $\pi'_{K, K'}$ is unique, if it exists. We define the map $\pi'_{K, K'}: \KAdmmu_0 \to {}^{K'} \Admmu_0$ by $w \mapsto w'$, where $w'$ is the unique element in ${}^{K'} \Admmu_0$ with $w \approx_{K', \s} w'$, cf. Proposition \ref{w-w}.  Now for any $g \brK \in X_{K, w}(\t)$, we have $g \i \t \s(g) \in \brK \cdot_\s \brI w \brI \subset \brK' \cdot_\s \brI w' \brI$. Therefore $\pi_{K, K'}(g \brK) \in X_{K', w'}(\t)$. This proves the commutativity of the diagram and thus shows the existence of $\pi'_{K, K'}$.
\end{proof}

\subsection{The fibers of the map $\pi'_{K, K'}$} 

Assume that our Tits datum is $(\tilde{\Delta}, \s, \mu) = (\tilde{A}_{n-1}, \varsigma_0, \omega_1^\vee)$  for $n \ge 3$, and $K'\setminus K \subset \{s_0, s_{\frac{n}{2}}\}$, and if $s_i \in K'\setminus K$, then $s_{i+1} \notin K$. By the proof of Proposition \ref{w-w}, if $K'\setminus K=\{s_j\}$ for $j \in \{0, \frac{n}{2}\}$, then for $w' \in {}^{K'} \Admmu_0$, $$(\pi'_{K, K'}) \i (w')=\begin{cases} \{w', s_j w' s_j\}, & \text{ if } w' s_j<w'; \\ \{w'\}, & \text{ if } w' s_j>w'.\end{cases}$$

If $n=2m$ and $K'\setminus K=\{s_0, s_m\}$,  then for $w' \in {}^{K'} \Admmu_0$, $$(\pi'_{K, K'}) \i (w')=\begin{cases} \{w', s_0 w' s_0, s_m w' s_m, s_0 s_m w' s_m s_0\}, & \text{ if } w' s_0<w', w' s_m<w'; \\ \{w', s_0 w' s_0\}, & \text{ if } w' s_0<w', w' s_m>w'; \\ \{w', s_m w' s_m\}, & \text{ if } w' s_0>w', w' s_m<w'; \\ \{w'\}, & \text{ if } w' s_0>w', w' s_m>w'.\end{cases}$$

\subsection{The fibers of the map $\pi_{K, K'}$} 
\label{sec:fibers-group-theoretic}

\smallskip

Next we study the fibers of the map $\pi_{K, K'}: \XmubK{\t}{K} \to \XmubK{\t}{K'}$. This will also finish the proof of Theorem \ref{pi'}.

\begin{theorem}\label{finite-fiber-final}
	Let $b \in \breve G$. Let $K \subseteq K' \subsetneqq \tSS$. Let $w \in {}^K \tW$ and $w' \in {}^{K'} \tW$. If $w \approx_{K', \s} w'$, then the natural projection map $X_{K, w}(b) \to X_{K', w'}(b)$ has finite fibers. 
\end{theorem}

We first recall the following result which relates a fine affine Deligne-Lusztig variety in the partial affine flag variety $\breve G/\brK$ to an ordinary affine Deligne-Lusztig variety in another partial affine flag variety. 

\begin{theorem}\cite[Thm. 4.1.2]{GH}\label{i-k-w}
	Let $K \subsetneqq \tSS$ and $w \in {}^K \tW$. Set 
    $$K_1=I(K, w, \s)=\max\{K' \subset K\mid \Ad(w) \circ \s(K')=K'\}.
	$$
	 Let $\brK_1$ be the associated parahoric subgroup. Then the natural projection map $\breve G/\brK_1 \to \breve G/\brK$ induces an isomorphism 
	$$X_{K_1, w}(b) \xrightarrow{\cong} X_{K, w}(b).
	$$\qed
\end{theorem}

Note that for $s\in K$, the element $w\s(s) w\i \in \tW$ is not in general
a simple reflection; it is part of the condition in the 
definition of $K_1$ that this is the case.

\begin{remark}
    Since $\Ad(w) \circ \s(K_1)=K_1$, we have $\brK_1 \cdot_\s \brI w \brI=\brK_1 w \s(\brK_1)$ and thus $X_{K_1, w}(b)=\{g \brK_1\mid g \i b \s(g)=\brK_1 w \s(\brK_1)\}$ is an ordinary affine Deligne-Lusztig variety in $\breve G/\brK_1$. 
\end{remark}

\begin{proposition}\label{finite-ff}
    Let $K \subset \tSS$ and $w \in {}^K \tW$ with $\Ad(w) \circ \s(K)=K$. Let $b \in \breve G$ with $X_w(b) \neq \emptyset$. Then each fiber of the projection map $X_w(b) \to X_{K, w}(b)$ consists of $\sharp (\brK/\brI)^{\Ad(w) \circ \s}$ elements. 
\end{proposition}

\begin{remark}
	Note that $\brK/\brI$ is the flag variety of the reductive quotient of $\brK$ and $\Ad(w) \circ \s$ induces a Frobenius morphism on the reductive quotient of $\brK$. Hence $(\brK/\brI)^{\Ad(w) \circ \s}$ is the set of rational points of a full flag variety over the finite field $k$.
\end{remark}

\begin{proof}
    Let $U_{\brK}$ be the pro-unipotent radical  of $\brK$ and  $\underline \brK \cong \brK/U_{\brK}$ the reductive quotient of $\brK$. Let $\underline B$ be the image of $\brI$ in $\underline \brK $. Then  $\underline B$ is a Borel subgroup of $\underline \brK$. Since $\Ad(w) \circ \s(K)=K$, the action of $\Ad(w) \circ \s$ stabilizes $\underline \brK$ and hence is a Frobenius morphism on $\underline \brK$. 
	
	By Lang's theorem, any element in $\underline \brK w \underline \brK=\underline \brK w$ is of the form $k w \s(k) \i$ for some $k \in \underline \brK$. Let $g \brI \in X_w(b)$. Then the elements in the same fibers as $g \brI$ are $g k \brI$ for $k \i g \i b \s(g) \s(k) \in \brI w \brI$. Note that $g \brI \in X_w(b)$. So $g \i b \s(g)=u k_1 \i w \s(k_1) u'$. Thus the condition $k \i g \i b \s(g) \s(k) \in \brI w \brI$ is equivalent to $\underline k \i k_1 \i w \s(k_1) \s(\underline k) \in \underline B w \s(\underline B)$, where $\underline k \in \underline \brK$ such that $k \in \underline k U_{\brK}$. Note that 
 $$
 \{\underline k \underline B \in \underline \brK/\underline B\mid \underline k \i k_1 \i w \s(k_1) \s(\underline k) \in \underline B w \s(\underline B)\} \cong \{\underline k \underline B \in \underline \brK/\underline B\mid \underline k \i w \s(\underline k) w \i \in \underline B\}.
 $$ The statement is proved. 
\end{proof}

\begin{proposition}\label{D-L}
	Let $w, w' \in \tW$ and $K \subset \tSS$ such that $w \approx_{K, \s} w'$ and such that $w \in {}^K \tW$. Then there is a commutative diagram 
	\[
	\xymatrix{X_{w}(b) \ar[rr]^-\cong \ar[rd] & & X_{w'}(b) \ar[ld] \\ & X_{K, w}(b) &.
	}
	\]
\end{proposition}

\begin{proof}
	It suffices to consider the case when $\ell(w')=\ell(w)$ and $w'=s w \s(s)$ for some $s \in K$. Without loss of generality, we may assume furthermore that $s w<w$. 
	
	By Deligne-Lusztig reduction \cite[Thm. 1.6]{DL}, for any $g \brI/\brI \in X_w(b)$, there exists a unique element $g' \brI/\brI \in g \brK_s/\brI$ such that $g' \brI \in X_{w'}(b)$. Moreover, the map $g \brI \to g' \brI$ induces a homeomorphism $X_w(b) \to X_{w'}(b)$. As $g \i g' \in \brK_s \subset \brK$, the diagram in the statement of the proposition is commutative. 
\end{proof}

\subsection{Proof of Theorem \ref{finite-fiber-final}}
	Let $K_1=I(K, w, \s)$ and $K'_1=I(K', w', \s)$. Then we have the following commutative diagram 
	\[
	\xymatrix{
		X_{w'}(b) \ar[r]^-\cong \ar[dd] & X_w(b) \ar[d] & \\
		& X_{K_1, w}(b) \ar[r]^-\cong & X_{K, w}(b) \ar[d] \\
		X_{K'_1, w'}(b)  \ar[rr]^-\cong & & X_{K', w'}(b).
	}
	\] Here the vertical maps are the projection maps. The isomorphisms $X_{K_1, w}(b) \cong X_{K, w}(b)$ and $X_{K'_1, w'}(b) \cong X_{K', w'}(b)$ follow from Theorem \ref{i-k-w}. The homeomorphism $X_{w'}(b) \cong X_w(b)$ and the commutativity of the diagram follow from Proposition  \ref{D-L}. By Proposition \ref{finite-ff}, the maps $X_{w'}(b) \to X_{K'_1, w'}(b)$ and $X_w(b) \to X_{K_1, w}(b)$ have finite fibers. Hence the map $X_{K, w}(b) \to X_{K', w'}(b)$ has finite fibers. Moreover, each fiber consists of $\frac{\sharp (\brK'_1/\brI)^{\Ad(w') \circ \s}}{\sharp (\brK_1/\brI)^{\Ad(w) \circ \s}}$ elements.  

\smallskip

Finally we determine  explicitly, in each of the two cases of Theorem \ref{pi'},  the fibers of the map $\pi_{K, K'}\colon  X^G(\mu, \t)_K \to X^G(\mu, \t)_{K'}$. 

\subsection{The $(\tilde A_{n-1}, \id, \o^\vee_1)$ case}
In this case, $\bG=\PGL_n$. Note that $\Ad(\t) \circ \s$ acts transitively on $\tSS$. For any $w \in W_a \t$, $\supp_{\s}(w) \neq \tSS$ if and only if $w=\t$. Thus by Theorem \ref{fine-dec}, $\XmubK{\t}{K}=X_{K, \t}(\t)$. We have $X_\t(\t)=\Omega \brI/\brI \subset \breve G/\brI$ is a finite subset consisting of $n$ points. And for any parahoric $K$, $X_{K, \t}(\t)$ is the image of $X_\t(\t)$ under the natural projection map $\breve G/\brI \to \breve G/\brK$. Hence $\XmubK{\t}{K}=X_{K, \t}=\Omega \brK/\brK \subset \breve G/\brK$ consists of $n$ points. More precisely, in each connected component of $\breve G/\brK$, there is precisely one point of $\XmubK{\t}{K}$. 
Moreover, for any $K \subsetneqq K' \subsetneqq \tSS$, the projection map $\XmubK{\t}{K} \to \XmubK{\t}{K'}$ is bijective.

\subsection{The $(\tilde{A}_{n-1}, \varsigma_0, \omega_1^\vee)$ case}\label{exotic-case-gp-thy}
We first discuss the case where $K'\setminus K=\{s_0\}$. By assumption, $s_{1}, s_{n-1} \notin K$ (recall that $K$ is $\s$-stable). Recall the explicit description of $\Admmu$ obtained in the proof of Proposition ~\ref{w-w}: The elements of $\Admmu$ are $\t$ and the elements of the form
\[
    s_i s_{i-i_1} \cdots s_{i-i_k} \t
\]
for $0 < i_1 < \cdots < i_r \le n-2$ (and all indices are understood in $\mathbb Z/n\mathbb Z$, $r$ could be $0$). An element $w\t \in\Adm(\mu)$ lies in $\Admmu_0$ if there exists $j$, $0\le j\le n-1$ such that $j, n-j+1\notin \supp(w)$.

Let $w \in \KAdmmu_0$ and $w'=\pi'_{K, K'}(w) \in {}^{K'}\!\!\Admmu_0$. The
proof of Proposition ~\ref{w-w} also shows that we have $w'=w$ or $w'=s_0 w s_0$. Hence at most two $K$-EKOR strata lie above the $K'$-EKOR stratum attached to $w'$, and we have two $K$-EKOR strata above the $K'$-EKOR stratum attached to $w'$ if and only if $w' \ne s_0w's_0 \in \KAdmmu$ and $\pi_{K, K'}'(s_0w's_0) = w'$. Using elementary properties of the Bruhat order and~\cite[Lemma 4.5]{Haines:Bernstein} one checks that this is equivalent to $w's_0 < w'$:
\begin{equation*}
\pi_{K, K'} \i \big(X_{K', w'}(\t)\big)=\begin{cases}X_{K, w'}(\t) \sqcup X_{K, s_0 w' s_0} (\t) & \text{ if } w' s_0<w\\
X_{K, w'}(\t) & \text{ if } w' s_0>w .
\end{cases}
\end{equation*}

From the explicit description we obtain that $I(K', w', \s)=I(K, w,
\s)$ or $I(K', w', \s)=I(K, w, \s) \sqcup \{s_0\}$, and that $s_0 \in I(K', w',
\s)$ if and only if $w's_0 = s_0 w'$. Since $s_0w' > w'$ by assumption, in this
case we have $w's_0 > w'$, and the above shows that there is a single $K$-EKOR
stratum above the $K'$-EKOR stratum for $w'$.

By the proof of Theorem \ref{finite-fiber-final}, for $g \in X_{K, w'}(\t)$, we
now obtain
\[
    \sharp \pi_{K, K'} \i(g)=
    \begin{cases}
        q+1, & \text{ if } I(K', w', \s)=I(K, w, \s) \sqcup \{s_0\};\\
        2, & \text{ if } w' s_0<w'; \\
        1, & \text{ if } w' s_0>w' \text{ and } I(K', w', \s)=I(K, w, \s).
    \end{cases}
\]
Here $q$ denotes the cardinality of the residue class field of $F$.

Let us express the condition $w's_0 = s_0 w'$ more explicitly, using once again
the explicit description of the admissible set in this case.

\smallskip

\emph{Claim:  $w's_0 = s_0 w'$ if and only if $w'\notin W_0\tau$, and in
this case $w's_0 > w'$.}

\smallskip

To prove the claim, note that for $w'\notin W_0\tau$, the explicit description
(and the assumption that $s_0 w' > w'$) show that $w'$ has the form $\cdots
s_1s_0 \cdots\t$, whence $s_0 w' s_0 = \cdots s_0 s_1 s_0 s_1\cdots \tau = w'$.
Since $s_0 w' > w'$ by assumption, it is also clear that $w's_0 > w'$ in this
case. On the other hand, if $w'\in W_0\tau$, then $s_0\t \le s_0 w'$ but
$s_0\t\not\le w'\t\i s_1 \t = w' s_0$.

Altogether we have proved:
\begin{proposition}\label{prop-exotic-minimal-gp-thy}
For $w' \in {}^{K'}\!\!\Admmu_0$ and $g \in X_{K, w'}(\t)$,
\[
    \sharp \pi_{K, K'} \i(g)=
    \begin{cases}
        q+1, & \text{ if and only if } w'\notin W_0\t;\\
        2, & \text{ if and only if } w' s_0<w'; \\
        1, & \text{ if and only if } w'\in W_0\t \text{ and } w' s_0>w'.
    \end{cases}
\]
\end{proposition}

See Proposition ~\ref{fibers-lattice-descr} for a proof of this proposition in terms of lattices.

The case $n=2m$, $j=m$ is completely analogous to the case above.
Similarly, if $n=2m$ for $m \ge 2$ and $K'\setminus K=\{s_0, s_m\}$, then for $w' \in {}^{K'} \Admmu_0$ and $g \in X_{K', w'}(\t)$, the fiber $\pi_{K, K'}\i(g)$ has  $1, 2, 4, q+1, 2(q+1)$ or $(q+1)^2$ depending on which of the conditions $w' s_0>w'$, $w' s_m>w'$, $\ell(s_0 s_1 w')=\ell(w')-2$ and $\ell(s_m s_{m+1} w')=\ell(w')-2$ are satisfied.

\begin{example}\label{exotic-GU3}
    Here we consider the case where $(\tilde{\Delta},  \s, \mu)=(\tilde A_2, \varsigma_0, \o^\vee_1)$. In this case, $$\Admmu=\{\t, s_0 \t, s_1 \t, s_2 \t, s_0 s_2 \t, s_1 s_0 \t, s_2 s_1 \t\}.$$ Let $K=\emptyset$ and $K'=\{s_0\}$. Then 
	\begin{equation*}
	\begin{aligned}
	\KAdmmu_0&=\{\t, s_0 \t, s_1 \t, s_2 \t, s_1 s_0 \t\}; \\
	{}^{K'} \Admmu_0&=\{\t, s_1 \t, s_2 \t, s_1 s_0 \t\}.
	\end{aligned}
	\end{equation*} 
	The map $\pi'_{ K, K'}$ sends $\t$ to $\t$, $s_2 \t$ to $s_2 \t$, both $s_0 \t$ and $ s_1 \t$ to $s_1 \t$, and $s_1 s_0 \t$ to $s_1 s_0 \t$. 
	
	Note that $I(K, w, \s)=\emptyset$ for $w \in \KAdmmu_0$ and $I(K', w, \s)=\emptyset$ for $w=\t, s_1 \t, s_2 \t$,  and $I(K', s_1 s_0 \t, \s)=K'$. Hence the natural projection map $\pi_{K, K'}$  induces isomorphisms
	 $$
	 X_{K, \t}(\t) \cong X_{K', \t}(\t),\, X_{K, s_2 \t}(\t) \cong X_{K', s_2 \t}(\t),\, X_{K, s_1 \t}(\t) \cong X_{K', s_1 \t}(\t),\, X_{K, s_0 \t}(\t) \cong X_{K', s_1 \t}(\t)
	 $$
	  and the projection map $X_{K, s_1 s_0 \t}(\t) \to X_{K', s_1 s_0 \t}(\t)$ is a $(q+1)$ to $1$ map, where $q+1$ is the cardinality of $(\brK'/\brK)^{\Ad(s_1 s_0 \t) \circ \s}$. 
	
	In summary, the fibers of the map $\pi_{K, K'}: X^G(\mu, \t)_K \to X^G(\mu, \t)_{K'}$ are as follows: 
	\begin{enumerate}
		\item over points in $X_{K', \t}(\t)$, each fiber consists of $1$ point;
		
		\item over points in $X_{K', s_2 \t}(\t)$, each fiber consists of $1$ point; 
		
		\item over points in $X_{K', s_1 \t}(\t)$, each fiber consists of $2$ points;
		
		\item over points in $X_{K', s_1 s_0 \t}(\t)$, each fiber consists of $q+1$ points.
	\end{enumerate}
 
\end{example}

\section{Lattice interpretation of the minimal cases}\label{s:lattice-interpretation}

In this section, we give explicit descriptions in terms of lattices for the Lubin-Tate case and the exotic case in which discrete fibers occur. To avoid too heavy notation, we do not include cases arising by restriction of scalars, but only discuss the non-extended cases.

\subsection{The Lubin-Tate case}\label{s:lattice-interpretation-LT}
In this subsection, we explain what $\XmubK{\t}{K}$ looks like in terms of a lattice description, in the Lubin-Tate case (Example~\ref{ex-lubin-tate}), as described in Theorem \ref{thm-all-k}. Let us consider first the case where $K$ is a hyperspecial maximal parahoric subgroup. In this case, we have the following description.

Let $(N, \phi)$ be an isocrystal of dimension $n$, where $\phi$ is a $\sigma$-linear automorphism isoclinic of slope $1/n$. Then we have (for $\bG=\GL_n$)
\begin{equation}
\XmubK{\t}{K}= \bigsqcup_{v\in\BZ} \{ M\mid M\supset \phi (M), \vol(M)=v\}.
\end{equation}

The decomposition indexed by $v$ corresponds to the decomposition of the affine Grassmannian, or correspondingly the space of all lattices in $N$, into connected components. Note that after passing to lattices, there is no dependence on $K$ anymore. More precisely, denote by $\Latt$ the set of all lattices in $N$. Viewing $K$ as the stabilizer of a lattice $\Lambda$, we have an identification $\GL_n(\brF)/K \cong \Latt$ mapping $g \mapsto g\Lambda$. Using this identification, we view $\XmubK{\t}{K}$ as a subset of $\Latt$. Likewise, we have an identification $\GL_n(\brF)/\tau K\tau^{-1} \cong \Latt$, now mapping $g\mapsto g\tau\Lambda$, and this is the identification we use when we want to view $\XmubK{\t}{\tau K\tau^{-1}}$ as a subset of $\Latt$. Since the bijection $\GL_n(\brF)/K \to \GL_n(\brF)/\tau K\tau^{-1}$, $g\mapsto g\tau^{-1}$, maps $\XmubK{\t}{K}$ onto $\XmubK{\t}{\tau K\tau^{-1}}$, as subsets of $\Latt$ we have $\XmubK{\t}{K} = \XmubK{\t}{\tau K \tau^{-1}}$. By iterating this, we can identify the affine Deligne-Lusztig varieties $\XmubK{\t}{K}$ for all standard hyperspecial parahorics $K$.

Note that for $M$ in $\XmubK{\t}{K}$ the index of $\phi (M)$ in $M$ is equal to $1$. 
\begin{lemma}
The chain of lattices
$$
M\supset \phi (M)\supset \phi ^2(M)\supset\ldots\supset \phi ^{n-1}(M)\supset \phi ^n(M)=pM 
$$
determines the unique fixed point under $\phi $ in $\CB(\PGL_n, \breve\BQ_p)$, i.e., the unique point in $\CB({\bf J}_{\tau,\ad}, \BQ_p)$. In particular, each connected component of $\XmubK{\t}{K}$ consists of a single point.
\end{lemma}
\begin{proof}
All we have to show is that $\phi ^n(M)=pM$: after this, the lattice chain determines an alcove in $\CB(\PGL_n, \breve\BQ_p)$ which is obviously fixed by $\phi $, i.e., lies in $\CB({\bf J}_{\tau,\ad}, \BQ_p)$. Since ${\bf J}_{\tau, \ad}$ is anisotropic, the latter building consists of only one point. 

We consider the chain of lattices
$$
M\supset \phi (M)\supset \phi ^2(M)+pM\supset \phi ^3(M)+pM\supset\ldots\supset \phi ^{n-1}(M)+pM\supset \phi ^n(M)+pM .
$$
{\it Claim: All inclusions are strict.}

\smallskip

Once the claim is proved, we conclude as follows. Since obviously all indices in this chain are $\leq 1$, the claim implies that $[M:(\phi ^n(M)+pM)]=n=[M:pM]$. Hence $\phi ^n(M)+pM=pM$, i.e., $\phi ^n(M)=pM$ (both have index $n$ in $M$). 

{\it Proof of claim:} Assume that $\phi ^r(M)+pM=\phi ^{r+1}(M)+pM$. Then $\phi ^{r+1}(M)+p\phi (M)=\phi ^{r+2}(M)+p\phi (M)$. Hence
\begin{equation*}
\begin{aligned}
\phi ^{r+1}(M)+pM=\phi ^{r+2}(M)+p\phi (M)+pM=\phi ^{r+2}(M)+pM.
\end{aligned}
\end{equation*}
We conclude that $\phi ^r(M)+pM=\phi ^j(M)+pM$, for any $j\geq r$. But $\phi $ is topologically nilpotent, hence $\phi ^j(M)\subset pM$ for large $j$. But this implies $\phi ^r(M)\subset pM$, which is absurd for $r\leq n-1$. 
\end{proof}

The lemma implies immediately  that $\XmubK{\t}{K}$ has only one element when $K$ is an arbitrary  parahoric.

\subsection{The exotic case}\label{sec:exotic}

For the setup, we follow~\cite{KRIII}, cf. also \cite{Ch}. The case of hyperspecial level structure (which in terms of the notation used below corresponds to the case $r=0$) was analyzed in detail by Vollaard~\cite{Vollaard}.

\subsubsection{The isocrystal}\label{subsec:isocrystal}
Let $\tilde F/F$ be the unramified quadratic extension contained in $\brF$.  We fix $n\ge 1$, and $1 \le s \le n-1$. We also fix the following data. 
\begin{enumerate}
\item
$N$ is a $\brF$-vector space of dimension $2n$ together with an alternating
$\brF$-bilinear pairing $\langle \,, \,\rangle\colon N\times N\to \brF$,
\item
there is a $\tilde F$-action on $N$ such that
\begin{equation}\label{J-Zp2}
\langle a\cdot x, y \rangle = \langle x, \sigma(a)\cdot y\rangle\quad\text{for all}\ x, y\in N, a\in\tilde F,
\end{equation}
\item
we have a $\s$-linear operator $\phi\colon N\to N$ which  commutes with the $\tilde F$-action and such that all slopes of $\phi$ are
equal to $\frac{1}{2}$, and which satisfies
\begin{equation}\label{J-FV}
\langle \phi(x), \phi(y)\rangle =\pi\cdot\s( \langle x, y\rangle)\quad\text{for all}\ x, y\in N ,
\end{equation}
where $\pi$ is a fixed uniformizer of $F$. 
\end{enumerate}

Via the $\tilde F$-action, $N$ is a module over
$\tilde F\otimes_{F}\brF = \brF\times \brF$, i.e., it decomposes as
$N = N^0 \oplus N^1$, where $\tilde F$ acts on $N^0$ via the inclusion
$\tilde F\subset \brF$, and on $N^1$ via $\sigma\colon \tilde F\to \brF$.
We then have $\phi(N^0) = N^1$, $\phi(N^1) = N^0$. 
The $\tilde F$-action on an element $x = (x^0, x^1)$ is given by $a(x^0, x^1)
= (ax^0, \s(a)x^1)$. By~\eqref{J-Zp2} (and using that the pairing is
alternating), we obtain that $N^0$ and $N^1$ are totally isotropic
subspaces.

We will consider $O_{\tilde F}$-invariant $O_{\brF}$-lattices $M$.  For them we obtain an analogous decomposition $M = M^0\oplus M^1$. We will impose 
 the \emph{signature condition} for $s$, i.e., $\pi M\subset \phi(M)\subset M$ with 
\begin{equation}\label{signs}
\pi M^0\subset^{n-s} \phi(M^1)\subset^{s} M^0 .
\end{equation}
Here the upper indices indicate the length as $O_{\brF}$-modules of the corresponding factor modules. 

For a lattice $M\subset N$, we denote by $M^\vee$ its dual with respect to the form $\langle\, , \, \rangle$, i.e., $M^\vee = \{ x\in N\mid \langle x, M\rangle \subseteq O_{\brF}\}$.

We will impose the following condition.
\begin{altitemize}
\item
{\it there exists a $O_{\tilde F}$-stable self-dual lattice $\BM\subset N$ such that $\pi\BM \subset \phi(\BM) \subset \BM$, and satisfying the signature condition for $s$. }
\end{altitemize}

In the setting of the following remark, this condition means that the above data
arise from a $p$-divisible group (with an $O_{\tilde F}$-action and
a $p$-principal polarization), as in~\cite{RZ96}. See
Remark~\ref{rmk:selfdual-lattice-gp-thy} for a discussion of this assumption in
terms of group theory.

\begin{remark} Let $F=\BQ_p$. Then the tuple $(N, \langle\, , \, \rangle, \phi)$ 
is the isocrystal of
a supersingular $p$-divisible group of height $2n$ over $\ov\BF_p$ with
$\BZ_{p^2}$-action which satisfies the determinant condition for signature $(s,
n-s)$, with a quasi-polarization compatible with the
$\BZ_{p^2}$-action, cf.~\cite[Def.~1.1]{Vollaard}. In loc.~cit., $p$-divisible groups are considered which admit a $p$-principal polarization. These correspond to self-dual lattices, i.e., $M^\vee=M$. Here we will consider more general parahoric level structures. In the case of maximal but non-hyperspecial level structure, the level structure can be seen as a (non-$p$-principal) polarization. 
\end{remark}

\subsubsection{The space of lattices}

Now let us fix an integer $r$, $0\le r\le n/2$. We will see below how this corresponds to a choice of maximal rational parahoric level structure.

Consider the following set of pairs of lattices in $N$.
\begin{equation}
\begin{aligned}
\CF^{\{2r\}} = \{ (\pi M_2 \subseteq M_1 \subseteq M_2)\mid \ & M_i\text{\ stable under $O_{\tilde F}$},\, M_1^0\subseteq^{2r} M_2^0, \, M_1^1\subseteq^{2r} M_2^1,\\
& M_2 = \pi^cM_1^\vee\text{\ for some}\ c\in\BZ\}.
\end{aligned}
\end{equation}
By mapping $(M_1\subseteq M_2)\in \CF^{\{2r\}}$ to $(M_1^0 \subseteq M_2^0, c)$, we obtain a bijection between $\CF^{\{2r\}}$ and the set
\begin{equation}
\CF^{\{2r\},0}:= \{ (\pi A\subseteq B\subseteq^{2r} A, c)\mid B, A\subset N^0\text{\ lattices}, c\in\BZ\}.
\end{equation}

This set of lattices will be identified below with the set of $\ov k$-points of the corresponding partial affine flag variety.

\subsubsection{The action of Frobenius}

The operator $\phi$ on $N$ induces an action on the set $\CF^{\{2r\}}$. In fact, for $(M_1\subseteq M_2)\in \CF^{\{2r\}}$ with $M_2 = \pi^cM_1^\vee$, we have $\phi(M_2) = \phi(\pi^c M_1^\vee) = \pi^{c+1} \phi(M_1)^\vee$. To describe this action in terms of the bijection $\CF^{\{2r\}} \isoarrow \CF^{\{2r\},0}$, we introduce the following notation.

Let $\boldtau=\pi^{-1}\phi^{2}$, a $\sigma^2$-linear automorphism of $N^0$ which has all slopes zero. Let $C=(N^0)^{\<\boldtau\>}$. Also, let
$$
h(x, y)=\delta^{-1}\pi^{-1}\langle x, \phi y\rangle,
$$
where $\delta\in O^\times_{\tilde F}$ is such that $\s(\delta)=-\delta$. Then the restriction of $h$ to $C$  is a hermitian form on $C$. On $N^0$, the hermitian nature of $h$ is given by
\begin{equation}\label{pseudohermit}
h(x, y)=\sigma\big(h(y, \boldtau^{-1} (x))\big).
\end{equation}

\begin{definition}
For a lattice $L\subset N^0$, we denote by
\[
L^\sharp = \{ x\in N^0\mid \pi^{-1}\langle x, \phi(L) \rangle \subseteq O_{\brF} \},
\]
the dual of $L$ with respect to the form $h$, which is again a lattice in $N^0$.
\end{definition}

Note that
\begin{equation}
(L^\sharp)^\sharp=\boldtau(L).
\end{equation}

\begin{lemma}\label{sharp-vs-F}
For $(M_1\subseteq M_2)\in \CF^{\{2r\}}$ corresponding to $(B\subseteq A, c) \in \CF^{\{2r\},0}$, the chain $(\phi(M_1)\subseteq \phi(M_2))$ corresponds to $((\pi^{-c}A)^\sharp \subseteq (\pi^{-c}B)^\sharp, c+1)$.
\end{lemma}

\begin{proof}
We need to check $\phi(M_1)^0 = (\pi^{-c}M_2^0)^\sharp$ and $\phi(M_2)^0 = (\pi^{-c}M_1^0)^\sharp$. Now $\phi(M_1)^0 = \phi(M_1^1)$, and
\[
\< \phi(M_1^1), \pi^{-c-1} \phi(M_2^0) \> =\s (\< M_1^1, \pi^{-c} M_2^0\>) = \s(\< M_1^1, (M_1^\vee)^0\>)= O_{\brF}
\]
by~\eqref{J-FV}, so $\phi(M_1^1) = (\pi^{-c}M_2^0)^\sharp$. The computation for $\phi(M_2)^0$ is similar.
\end{proof}

\subsubsection{The parahoric RZ space}

The $\overline{k}$-valued points of the (\emph{relative}) RZ space which we want to describe correspond to those points in $\CF^{\{2r\}}$ (or equivalently in $\CF^{\{2r\},0}$) which are Dieudonn\'e modules of signature $(s, n-s)$:
\begin{equation}
\CN = \CN^{\{2r\}} = \{ (M_1\subseteq M_2)\in \CF^{\{2r\}} \mid \pi M_i \subseteq \phi(M_i)\subseteq M_i, i=1, 2\}.
\end{equation}
Here $\phi(M_i)^0\subseteq M_i^0$ has co-length $s$ and  $\phi(M_i)^1\subseteq M_i^1$ has co-length $n-s$. By Lemma~\ref{sharp-vs-F}, we can identify $\CN$ with a subset of $\CF^{\{2r\},0}$, as follows:
\begin{equation}
\CN = \{ (B\subseteq A, c)\in \CF^{\{2r\},0} \mid \pi B\subseteq \pi^c A^\sharp\subseteq^s B, \pi A\subseteq \pi^c B^\sharp\subseteq^s A\}.
\end{equation}

\subsubsection{Reduction to the case $c=0$}

We have
\[
\CN = \bigsqcup_{c\in\BZ} \CN_c,
\]
where, for $c\in \BZ$, we write
\[
\CN_c = \{ (B\subseteq A\subset N) \mid (B\subseteq A, c)\in \CF^{\{2r\},0},\  \pi B\subseteq \pi ^c A^\sharp\subseteq B, \pi A\subseteq \pi ^c B^\sharp\subseteq A\}.
\]

\begin{lemma}
\begin{enumerate}
\item
If $nc$ is odd, then $\CN_c = \emptyset$.
\item
If $nc$ is even, then there exists an automorphism $j$ of $N$ compatible with
$\phi$ and the pairing $\<\,,\,\>$ (and hence with the pairing $h$ and the
$-^\sharp$ construction) such that the map $(B\subseteq A) \mapsto (jB\subseteq
jA)$ is an isomorphism $\CN_c \cong \CN_0$.
\end{enumerate}
\end{lemma}

\begin{proof}
Part (1) follows by a comparison of indices between $A$, $B$, $A^\sharp$,
$B^\sharp$ and $\BM$, similarly as in~\cite[Lemma 1.7]{Vollaard}. Part~(2) is
proved in~\cite[Lemma 1.17]{Vollaard}.
\end{proof}

From now on we assume $c=0$, so we consider the set
\begin{equation}
\CN_0 = \CN_0^{\{2r\}} = \{ \pi A\subseteq B\subseteq^{2r} A \subset N^0 \mid \pi B\subseteq A^\sharp\subseteq^s B, \pi A\subseteq B^\sharp\subseteq^s A\}.
\end{equation}
This is the description given in \cite{KRIII}, comp. \cite{Ch}\footnote{In loc.~cit. also pairs $M_1\subset M_2$ are considered where $M_1^0\subset M_2^0$ has \emph{odd} co-length.}. Note that the Hasse invariant of $C$ is given by $\inv(C)=(-1)^s$. 

\subsubsection{Non-maximal level structure}

Combining the above data for more than one $r$, we get analogous descriptions of the RZ spaces $\CN^{R}$, $\CN_0^R$ with more general parahoric level structure $R\subseteq \{0, \dots, [n/2]\}$. For instance, combining the cases $r=0$ and $r=1$, we obtain a non-maximal parahoric case, given as the set of diagrams
\begin{equation}
\begin{matrix}
B_1&\subset B_0&\subset&A_1\\
\cup&\,\cup&{}&\,\cup\\
\nass
A_1^\sharp&\subset B_0^\sharp&\subset &B_1^\sharp
\end{matrix}
\end{equation}
Here all horizontal inclusions have index $1$ and it is understood that $\pi A_1\subseteq B_1$.
The index of the vertical inclusions in the above diagram is equal to $s$.

\subsubsection{Description of fibers: ``forgetting $L_0$''}\label{subsec:forget-L0}

From now on we restrict to the case $s=1$, i.e., to signature $(1, n-1)$.  Let
us describe explicitly, in terms of lattices, the projection
\[
\CN^{R\cup \{0\}}_0 \to \CN^R_0
\]
for a level structure $R\subseteq\{ 1,\dots, [n/2]\}$ (i.e., $0\not\in R)$ such
that $1\in R$, between spaces with parahoric level structures which is given by
forgetting the lattice at position $0$. In terms of the group-theoretic description to be discussed below, this case corresponds to $K'\setminus K = \{ s_0\}$. In other words, we need to describe,
for a diagram of lattices in $N^0$ with all inclusions of index $1$ and
$\pi A_1\subseteq B_1$,
\begin{equation}\label{iwahori}
\begin{matrix}
B_1&\subset B_0&\subset&A_1\\
\cup&\,\cup&{}&\,\cup\\
\nass
A_1^\sharp&\subset B_0^\sharp&\subset &B_1^\sharp 
\end{matrix}
\end{equation}
how many choices there are for $B_0$ when $A_1$ and $B_1$ are fixed. (All the other positions which might be present in $R$ are irrelevant for determining the fiber.)

We distinguish cases, depending on whether $B_1\subseteq B_1^\sharp$, or not.

\textbf{First case: $B_1\not\subseteq B_1^\sharp$.}
In this case, we have $A_1^\sharp = B_1\cap B_1^\sharp \supseteq \pi A_1$. Thus $A_1/A_1^\sharp$ is an $\overline{k}$-vector space with a ``hermitian'' form, and $B_0^\sharp/A_1^\sharp \subset B_1^\sharp/A_1^\sharp$ is an isotropic line.

\smallskip

\emph{Claim: There are exactly $q+1$ such lines.}

\smallskip

\emph{Proof of claim.} 
By assumption, $A_1/A_1^\sharp = B_1/A_1^\sharp \oplus B_1^\sharp/A_1^\sharp$, and the restriction of the pairing to $B_1^\sharp/A_1^\sharp\times B_1^\sharp/A_1^\sharp$ is non-degenerate. The entirety of all non-trivial subspaces of $B_1^\sharp/A_1^\sharp$ is a projective
line. Mapping a line $L$ to $L^\sharp \subset B_1^\sharp/A_1^\sharp$ defines
a twisted Frobenius on this projective line over $\overline{k}$, i.e., a $k$-structure on this projective line (cf.~\cite[Lemma 2.12]{Vollaard}).
The isotropic lines correspond to the rational points with respect to this
$k$-structure. Over a finite field, every form of $\BP^1$ is $\BP^1$,
so there are $q+1$ points.

\textbf{Second case: $B_1\subseteq B_1^\sharp$.} In this case, the only possibilities for $B_0$ are $B_0=B_1^\sharp$ or $B_0=\boldtau^{-1}B_1^\sharp$ (which can equivalently be expressed as $B_0^\sharp = B_1$). In fact, if $B_0\ne B_1^\sharp$, then $B_0 + B_1^\sharp = A_1$, and similarly, if $B_1\ne B_0^\sharp$, then $B_1 + B_0^\sharp = B_0$, so from both inequalities together we obtain $B_1^\sharp = B_1 + B_0^\sharp + B_1^\sharp = A_1$, an obvious contradiction.

Depending on whether $B_1 = \boldtau(B_1)$, or not, we have one or two points in the fiber.

\subsubsection{Description of fibers: general case}

If $n$ is odd, then the case considered in the previous section is the only
possible case. If $n=2m$ is even, the case of forgetting $L_{m}$ is completely
analogous to the case of forgetting $L_0$.

Finally, if $n$ is even, there is the case of forgetting $L_0$ and $L_{m}$.
This case corresponds to the case $K'\setminus K = \{ s_0, s_{m}\}$.
Since forgetting $L_0$ and forgetting $L_{m}$ is independent of each other,
the fibers in this case are just products of fibers arising in the case of
forgetting one lattice of the chain. In particular, we see that the possible
cardinalities of fibers are $1$, $2$, $4$, $q+1$, $2(q+1)$, $(q+1)^2$.

\subsubsection{Connection with group theory}

For this subsection, the condition $s=1$ plays no role.
Let $V$ be an $n$-dimensional $\tilde F$-vector space with an alternating
bilinear form $\<\,, \,\>\colon V\times V\to F$ such that $\< av, w\> = \< v,
\s(a)w\>$ for all $a\in \tilde F$, $v, w, \in V$, and let $\bG$ be the
associated group of similitudes of this pairing,
cf.~\cite{Vollaard-Wedhorn}~2.1. As before, we write $\breve G = \bG(\brQp)$. Setting $N = V\otimes_{F}\brF$ and
extending the pairing, we obtain a $2n$-dimensional $\brF$-vector space $N$ with an action of $\tilde F$ 
and a pairing which satisfy properties (1), (2) in
Section~\ref{subsec:isocrystal}. Conversely, starting with $N$ and a pairing
satisfying (1), (2) and choosing a $\tilde F$-subvector space $V\subset N$ such
that $V\otimes_{F}\brF = N$ and such that the pairing restricted to
$V\times V$ takes values in $F$, we obtain data as above.

We assume that $V$ contains a self-dual $O_{\tilde F}$-lattice $L_0$, and we fix  a self-dual ``standard lattice chain'' of $O_{\tilde F}$-lattices in $V$ containing $L_0$. This gives us a standard Iwahori subgroup. As in the previous sections, we have the extended affine Weyl group $\tW$, the set $\tSS$ of simple affine reflections, etc.

By restricting to part of the standard lattice chain, we can identify each $\CF^{\{2r\}}$ as a quotient of $\breve G$ by the standard parahoric subgroup of type $K=K^{\{r\}}=\{ 0, \dots, n-1 \} \setminus \{ r, n-r \}$ if $r>0$, or $K=K^{\{0\}}=\{ 1, \dots, n-1 \}$ if $r=0$. We obtain analogous identifications for non-maximal parahoric level structure.

Now suppose that $N = V\otimes_{F}\brF$ comes equipped with an operator $\phi$, as in~\ref{subsec:isocrystal}~(3). We write $F = b\s$ where $b\in \GL(N)$ and $\s = \id\otimes \s$. Then~\eqref{J-FV} amounts to saying that $b\in \breve G$ with multiplier $c(b)=\pi$. The condition that  $\phi$ be isoclinic  is equivalent to requiring that $b$ is basic. Conversely, starting with a basic element $b\in \breve G$ with multiplier $\pi$, we can define $\phi = b\s$.

According to the choice of the integer $s$, $1\le s\le n-1$, which defines the signature condition, we define the cocharacter $\mu_+ = \omega_s^\vee$. We denote by $\mu$ its conjugacy class.

\begin{remark}\label{rmk:selfdual-lattice-gp-thy}
\begin{altenumerate}
\item
Given the vector space $V$ with the pairing $\<\,, \,\>$, the existence of
a selfdual lattice is equivalent to the existence of a \emph{hyperspecial}
parahoric subgroup in $\breve G$ defined over $F$. This in turn is equivalent to
$\bG$ being quasi-split (over $F$).
\item
We have $[b]\in \BGmu$ if and only if $\XmubK{b}{K} \ne \emptyset$ (for
any/every $K$), see~\cite{Wintenberger}, \cite{He-KR}.
Since there is a unique basic element in $\BGmu$, we see that the $\s$-conjugacy
class $[b]$ is uniquely determined by $s$ under the condition $\XmubK{b}{K} \ne
\emptyset$.
\item
The following proposition says that $\XmubK{b}{K} \ne \emptyset$ if and only if
there exists a self-dual Dieudonn\'e module satisfying the signature condition
corresponding to $\mu_+$. The latter condition is the condition which we imposed
in Section~\ref{subsec:isocrystal}.
\end{altenumerate}
\end{remark}

The map $g= (g^0, g^1) \mapsto (g^0, c(g))$ gives an isomorphism $\bG_{\breve F} \isoarrow \GL(N^0)\times\BG_{m,\breve F}$ of algebraic groups over $\breve F$.
Via this isomorphism, we can also view $\CF^{\{2r\}}$ as a partial affine flag variety for the group $\GL(N^0)\times\BG_{m,\breve F}$. This corresponds to the identification $\CF^{\{2r\}} = \CF^{\{2r\}, 0}$.

Consider the space $\CN\subset\CF$ as defined above, for level structure corresponding to $K\subset \tSS$.

\begin{proposition}
In the setting outlined above, 
\[
\CN^{\{2r\}} = X(\mu, b)_{K^{\{r\}}}
\]
as subsets of the corresponding partial affine flag variety $\CF$ over $\breve F$.\qed
\end{proposition}

\begin{proof}
Inside the partial flag variety, for both these sets, their definition can be expressed by imposing conditions on the relative position between the partial lattice chain and its image under Frobenius. For $\CN^{\{2r\}}$, the condition is that this relative position be $\mu$-permissible in the sense of~\cite{kottwitz-rapoport}. For $X(\mu, b)_{K^{\{r\}}}$, the condition is that it must be $\mu$-admissible. By loc.~cit., the two conditions coincide. (Note that because of the identification $\bG_{\breve F} \isoarrow \GL(N^0)\times\BG_{m,\breve F}$ it is enough to know this for $GL_n$.)
\end{proof}

In analogy with the decomposition $\CN^{\{2r\}} = \sqcup_c \CN^{\{2r\}}_c$, the
space $\XmubK{b}{K^{\{r\}}}$ decomposes as a union of spaces of the form $\XmubK{b}{K^{\{r\}}}$
for a unitary group, rather than a group of unitary similitudes.

The group $\BJ_b$, the $\s$-centralizer of $b$, in this context can be
identified with the unitary similitude group of the hermitian space $C$.

\subsubsection{Description of fibers and the EKOR stratification}

Let us discuss the case of ``forgetting $L_{0}$'' with the connection to group theory in mind. As before, we assume $s=1$.  (The other cases can be handled similarly.) As above, fix a level structure $R\subseteq\{ 1,\dots, [n/2]\}$ such that $1\in R$.

Recall our terminology of KR and EKOR strata, see Section~\ref{subsec:EKOR}.
In terms of lattices, the KR stratification on the Iwahori level space $\CN^{\rm Iw} \cong
\XmubK{b}{}$ is given by the relative position of $L_\bullet$,
$L_\bullet^\sharp$.  The EKOR stratification on $\XmubK{b}{K}$ likewise induces
a stratification on the corresponding $\CN$ space which we can describe as the
coarsest stratification such that the projection of every KR stratum is a union
of EKOR strata, cf.~\cite{HR}. For $w\in \KAdmmu$, the index set for the EKOR
stratification, the projection of the KR stratum for $w$ is equal to the EKOR
stratum for $w$, i.e., the partial lattice chains in the EKOR stratum for $w$
are precisely those chains which can be extended to a full lattice chain
$L_\bullet$ such that the relative position of $L_\bullet$ and
$L_\bullet^\sharp$ is equal to $w$.

As the standard lattice chain we choose
\[
\Lambda_\bullet = \cdots \subset \diag(p, 1, \dots, 1) \subset \diag(1, \dots, 1) \subset \diag(1, \dots, 1, p^{-1}) \subset \cdots
\]
where $\diag(\,)$ denotes a diagonal matrix and a matrix is understood as
a lattice by taking the lattice generated by its column vectors.

Let $\tau$ be the matrix
\[
\left(
\begin{array}{cccc}
& &  & p\\
1 & & & \\
& \ddots & & \\
& & 1 &
\end{array}
\right),
\]
so that $\tau \Lambda_i = \Lambda_{i+1}$. We can also view $\tau$ as a length $0$ element of the Iwahori-Weyl group of $\breve G$.

The simple reflections are given as follows:
\[
s_1 =
\left(
\begin{array}{ccccc}
& 1 & & &\\
1 & & & & \\
&  & 1 & & \\
& & & \ddots & \\
& & & & 1
\end{array}
\right), 
s_2 =
\left(
\begin{array}{cccccc}
1 &  & & & &  \\
& & 1 & & & \\
& 1 & & & &\\
& &  & 1 & & \\
& & & & \ddots & \\
& & & & & 1
\end{array}
\right),\dots,
s_0 =
\left(
\begin{array}{ccccc}
& &  & & p \\
& 1 & &&  \\
& & \ddots &&  \\
& & & 1& \\
p^{-1} & &  & &
\end{array}
\right).
\]

\begin{proposition}
With notation as in diagram~\eqref{iwahori}, each of the conditions
\begin{enumerate}
\item
$B_1\subseteq B_1^\sharp$, $B_1 = \boldtau(B_1)$,
\item
$B_1\subseteq B_1^\sharp$, $B_1 \ne \boldtau(B_1)$,
\item
$B_1\not\subseteq B_1^\sharp$
\end{enumerate}
describes a union of EKOR strata. The fibers of the projection $\pi_{K, K'}$ have cardinality $1$ in Case (1), cardinality $2$ in Case (2) and cardinality $q+1$ in Case (3).
\end{proposition}

As before, $q$ denotes the cardinality of the residue class field of $F$.

\begin{proof}
Via our choice of standard lattice chain, the alcove expression for the identity element of $\tW$ and of $\tau$ are
\[
\begin{array}{lcccccc}
{\alc}(\id)_\bullet: &\quad \dots, & (1^{(2)}, 0^{(n-2)}),& (1, 0^{(n-1)}),& (0^{(n)}),& (0^{(n-1)}, -1), & \dots \\
{\alc}(\tau)_\bullet: & \quad\dots, & (1^{(3)}, 0^{(n-3)}),& (1^{(2)}, 0^{(n-2)}),& (1, 0^{(n-1)}),& (0^{(n)}), & \dots
\end{array}
\]
respectively. Here we use the ``alcove notation'' of~\cite{kottwitz-rapoport}.
Similarly, any $w'\in \KAdmmu$ gives rise to such an alcove expression $({\alc}(w')_i)$ with each ${\alc}(w')_i\in\BZ^n$, and $w'$ is determined by this datum. The fact that $w'\in \Admmu$ translates to the condition $\alc(\id) \le \alc(w') \le \alc(\id) + (1^{(n)})$, where $\le$ means that for each index, the respective entries are $\le$. The condition $B_1\subset B_1^\sharp$ translates to $(1, 0^{(n-1)}) \ge \alc(w')_{-1}$, which together with the admissibility implies that $\alc(w')_{-1} = (0^{(n)})$ or $\alc(w')_{-1} = (1, 0^{(n-2)}, -1)$. The latter case is not possible because $w'\in {}^K\tW$. 

Now assume that  $B_1\not\subset B_1^\sharp$, then $\alc(w')_{-1}$ has the form $(0, 0^{(i)}, 1, 0^{(n-i-3)}, -1)$ for some $i\ge 0$. Since these conditions are constant on each KR stratum, and are phrased in terms of the indices $1$, $-1$ of the lattice chain only, they describe unions of EKOR strata.

Now assume that $B_1\subset B_1^\sharp$, so $\alc(w')_{-1} = (0^{(n)})$. Then $B_1^\sharp = B_0$, so the condition $B_1 = \boldtau(B_1)$ becomes $B_1 = B_0^\sharp$ which is equivalent to $\alc(w')_0 = (1, 0^{(n-1)})$. Again, this clearly describes a union of EKOR strata. (Note that at this point $B_1\subset B_1^\sharp$ implies $B_1^\sharp = B_0$, i.e., we do not see the possibility $B_0 = \boldtau\i B_1^\sharp$ in the second case of Section~\ref{subsec:forget-L0}. This is because we are not considering the full fiber here, but only the EKOR strata for $w'$, for level $K$ and $K'$.)
\end{proof}

We now recover the characterization of the loci of different fiber cardinalities
as unions of EKOR strata which we proved group-theoretically as
Proposition ~\ref{prop-exotic-minimal-gp-thy}. (But note that in the lattice context we did not reprove Theorem~\ref{pi'} because we did not separate the unions of EKOR strata where the fiber cardinality is constant into individual EKOR strata.)

\begin{proposition}\label{fibers-lattice-descr}
Fix a point in a parahoric RZ space $\CN_0$ given by a diagram
\begin{equation}
\begin{matrix}
\cdots\subset &B_1&\subset&A_1 & \subset\cdots\\
& \cup&{}&\,\cup &\\
\nass
\cdots\subset &A_1^\sharp&\subset &B_1^\sharp & \subset\cdots
\end{matrix}
\end{equation}
which lies in the EKOR stratum for $w'\in \KAdmmu$. Then 
\begin{itemize}
\item
$B_1\subseteq B_1^\sharp$, $B_1 = \boldtau(B_1)$
if and only if $w'\in W_0\t$ and $w's_0 > w'$, if and only if the fiber cardinality is $1$,
\item
$B_1\subseteq B_1^\sharp$, $B_1 \ne \boldtau(B_1)$
if and only if $w's_0 < w'$, if and only if the fiber cardinality is $2$,
\item
$B_1\not\subseteq B_1^\sharp$
if and only if $w'\not\in W_0\t$, if and only if the fiber cardinality is $q+1$.
\end{itemize}
\end{proposition}

\begin{proof}
First note that by the proof of the previous proposition, $B_1\subseteq
B_1^\sharp$ is equivalent to $\alc(w')_{-1} = (0^{(n)})$ or, in other words,
$w'\in W_0\tau$. This already proves the third statement. Now if $w'\in
W_0\tau$, then $\ell(w')$ is the number of inversions of the permutation $v:=
w'\t\i$. We have $B_1 = \boldtau(B_1)$ if and only if $\alc(w')_{0} = (1,
0^{(n-1)})$, if and only if $v(1) = 1$. In this case, $w's_0 = vs_1 \t$ has
length $\ell(w')+1$. On the other hand, if $v(1) \ne 1$, then by the
admissibility of $w'$, $v(1)=2$, so $w's_0 = vs_1 \t$ has length $\ell(w')-1$.

It remains to prove that $w'\not\in W_0\t$ implies $w's_0 > w'$. As we found
above, $w'\not\in W_0\t$ means that $\alc(w')_0 = (0, \dots, 0, 1, 0\dots, 0)$
with the $1$ in position $i>1$. We have $\alc(w')_i =\alc(w's_0)_i$ for all $i \ne 0$,
and hence $\alc(w')_1\le \alc(w's_0)_0 \le \alc(w')_{-1}$ and $\alc(w's_0)_0 \ne
\alc(w')_{-1}$. Thus the only possibility for $\alc(w's_0)_0$ is $(1, 0, \dots, 0, 1, 0, \dots, 0, -1)$
whence $w's_0\notin \Admmu$. This is only possible if $w's_0 > w'$.
\end{proof}

\subsubsection{The EKOR stratification in the case of signature $(1, 2)$}

In the case $n=3$, we can describe explicitly all the KR and EKOR strata, cf.~Example~\ref{exotic-GU3}. As a preparation, we write down explicitly the KR
strata in terms of lattices. In this section, we consider the full affine flag
variety for $\GL_3$ over $\brF$.  The set of $\overline{k}$-valued points is the set of
full periodic lattice chains $L_\bullet$. Since all lattice chains are
periodic, we usually only consider degrees $1$, $0$, $-1$.

\begin{lemma}\label{relpos}
Let $L_\bullet$, $L'_\bullet$ be lattice chains and denote by $\mathop{\rm inv}(L_\bullet, L'_\bullet)\in \tilde W$ their relative position.
\begin{enumerate}
\item
$\mathop{\rm inv}(L_\bullet, L'_\bullet) = \tau$ if and only if $L'_i = L_{i+1}$ for $i=1, 0, -1$ (equivalently: for all $i$),
\item
$\mathop{\rm inv}(L_\bullet, L'_\bullet) \in \{ s_0 \tau, \tau \}$ if and only if $L'_1 = L_{2} (=\pi L_{-1})$ and $L'_0 = L_1$,
\item
$\mathop{\rm inv}(L_\bullet, L'_\bullet) \in \{ s_1 \tau, \tau \}$ if and only if $L'_1 = L_{2} (=\pi L_{-1})$ and $L'_{-1} = L_0$,
\item
$\mathop{\rm inv}(L_\bullet, L'_\bullet) \in \{ s_2 \tau, \tau \}$ if and only if $L'_{-1} = L_{0}$ and $L'_{0} = L_1$,
\item
$\mathop{\rm inv}(L_\bullet, L'_\bullet) \in \{ s_1s_0 \tau, s_0\tau, s_1\tau, \tau \}$ if and only if $L'_{1} = L_{2} (=\pi L_{-1})$.
\end{enumerate}
\end{lemma}

The lemma describes all KR strata for $w\in \Admmu_0$. We omit the easy proof. As a consequence, we obtain the following description of the EKOR strata in $\CN_0^{\{2\}}$. (It is possible to characterize the EKOR strata by other conditions, in the style of the original definition of the EO stratification in the Siegel case, see for instance~\cite{Oort}, \cite{GY}; we have made a choice which is close to the criteria we have found above for the cardinality of the fibers of the projection from the Iwahori space.)

\begin{proposition}
A point in $\CN_0^{\{2\}}$, given by a diagram
\begin{equation}
\begin{matrix}
B_1&\subset&A_1\\
\cup&{}&\,\cup\\
\nass
A_1^\sharp&\subset &B_1^\sharp
\end{matrix}
\end{equation}
lies in the EKOR stratum attached to
\begin{enumerate}
\item
$\tau$ if and only if $pA_1 = A_1^\sharp$, $B_1\subseteq B_1^\sharp$, $B_1 = \boldtau(B_1)$,
\item
$s_1\tau$ if and only if $B_1\subseteq B_1^\sharp$, $B_1 = \boldtau(B_1)$, (and on this stratum $\pi A_1 = A_1^\sharp$),
\item
$s_2\tau$ if and only if $\pi A_1 \ne A_1^\sharp$ (and on this stratum  $B_1\subseteq B_1^\sharp$, $B_1 = \boldtau(B_1)$)
\item
$s_1s_0\tau$ if and only if $B_1\not\subseteq B_1^\sharp$ (and on this stratum  $\pi A_1 = A_1^\sharp$). \qed
\end{enumerate}
\end{proposition}

\section{Proof of Theorems \ref{Main1} and \ref{Main2}}\label{finalpart2}

In this section, we deduce Theorems \ref{Main1}, resp. \ref{Main2} from Theorems \ref{thm-dimension-0}, resp. \ref{thm-fixed-s}. Let $(\bG, \mu)$ be such that $\bG$ is quasi-simple and $\mu$ non-central. Write $\bG=\Res_{\tilde F/F}\tilde\bG$, for a finite field extension $\tilde F$ and an absolutely quasi-simple group $\tilde\bG$ over $\tilde F$. We also write $\mu=(\mu_\varphi)$, where $\mu_\varphi$ are cocharacters of $\tilde\bG$. Here $\varphi$ runs over $\Hom_F(\tilde F, \ov F)$. Let $F_d$ be the maximal unramified subextension of $\tilde F$, $d=[F_d:F]$,  and fix an embedding of $F_d$ into $\ov F$. Let $\bG_d=\Res_{\tilde F/F_d}\tilde\bG$. Then $\bG=\Res_{F_d/F}(\bG_d)$, and the Tits datum over $F$ of $(\bG, \mu)$ is equal to $(\Res_{F_d/F}(\tilde \Delta_{\bG_d},\sigma_d), (\umu_{d, i})_i)$, where  $\tilde \Delta_{\bG_d}$ is the absolute Dynkin diagram of $\bG_d\otimes_{F_d}\brF$ with its action $\sigma_d$ of the Frobenius over $F_d$, and where, for $i=0, \ldots, d-1$,  we denote by $\umu_{d, i}$  the element in the translation lattice corresponding to $\mu_{d,i}=(\mu_\varphi)_\varphi$. Here $\varphi$ runs over those elements of $\Hom_{F}(\tilde F, \ov F)$ whose restriction to $F_d$ is equal to $\sigma^i$. Note that $\tilde\Delta_{\bG_d}$ coincides with the absolute local Dynkin diagram $\tilde\Delta_{\tilde\bG}$ of $\tilde\bG\otimes_{\tilde F} \breve{\tilde F}$, 
where $\breve{\tilde F}=\tilde F\otimes_{F_d}\brF$ is the completion of the maximal unramified extension of $\tilde F$, cf. \cite[1.13]{Tits:Corvallis}. 

Let now $(\bG, \mu)$ satisfy the conclusions of Theorems \ref{thm-dimension-0}, resp. \ref{thm-fixed-s}. In the case of Theorem \ref{thm-dimension-0}, it follows that  $(\tilde \Delta_{\bG_d},\sigma_d)=(\tilde A_{n-1}, \id)$. Furthermore, by changing the embedding of $F_d$ into $\ov F$, we deduce   from  
$\umu_d=(\o_1^\vee, 0,\ldots,0)$ that, for $i\neq 0$, $\umu_{d, i}$ is central and then that $\mu_{d, i}$ is central, cf.~Lemma~\ref{lemma-mu-central}. From $\umu_{d, 0}=\o_1^\vee$, we similarly deduce that there exists a unique $\varphi_0\in\Hom_{F_d}(\tilde F, \ov F)$ such that $\mu_{\varphi_0}=\o_1^\vee$ and such that $\mu_\varphi$ is central for all $\varphi\in\Hom_{F_d}(\tilde F, \ov F)\setminus\{\varphi_0\}$, comp. Lemma \ref{lemma-mu-minuscule} and the table right before Lemma 5.4 in \cite{HPR}.
It also follows that $\tilde\bG_\ad=\PGL_n$, and Theorem \ref{Main1} follows. 

In the case of Theorem \ref{thm-fixed-s}, and excluding the case treated in Theorem \ref{thm-dimension-0}, it follows that $(\tilde \Delta_{\bG_d},\sigma_d)=(\tilde A_{n-1}, \varsigma_0)$. Analogously to the case treated before, we obtain that there exists a unique $\varphi_0\in\Hom_F(\tilde F, \ov F)$ such that $\mu_{\varphi_0}=\o_1^\vee$ and such that $\mu_\varphi$ is central for all $\varphi\neq\varphi_0$, comp. Lemma \ref{lemma-mu-minuscule}. It follows that $\tilde\bG_\ad$ is an outer twist of $\PGL_n$ by an unramified quadratic extension $\tilde F'$ of $\tilde F$.  Hence  $\tilde\bG_\ad=\U(V)_\ad$, for a  $\tilde F'/\tilde F$-hermitian vector space $V$. The condition on $(K, K')$ in Theorem \ref{Main2} follows directly from Theorem \ref{thm-dimension-0}, and implies that the hermitian space $V$ is split (existence of a lattice which is self-dual or self-dual up to a scalar). Theorem \ref{Main2} is proved.

\part{Maximal dimension}
In this part, we consider the  problem opposite to the one of the last part: when is $\XmubK{b}{K}$  of maximal dimension? 

\section{Dimension of ADLV}\label{s:dimADLV}
\subsection{Admissible sets}
In this subsection, we introduce a dimension notion for certain subsets of $\breve G$. We follow \cite[\S2.5]{He-CDM}.
We view $\breve G$ as the set of $\overline k$-valued points of the loop group of $\bf G$ and equip it with the ind-topology. Then the closure $\overline{\brI x\brI}$ is equal to the (perfect) scheme $\bigcup_{x'\le x} \brI x'\brI$, and a subset $V$ is closed if and only if its intersection with $\overline{\brI x\brI}$ is closed for the Zariski topology, for all $x\in \tW$.

A subset $V$ of $\breve G$ is called {\it admissible}\footnote{This notion of admissibility is not related to the $\mu$-admissible set.} if for any $w \in \tW$, the set $V \cap \brI  w \brI$ is stable under the right action of an open compact subgroup $\brK_w$ which contains a congruence  subgroup $\brI_n$ of $\breve G$. This is equivalent to asking that for any $w \in \tW$,  the set $V \cap \overline{\brI  w \brI}$ is stable under the right action of an open compact subgroup $\brK_w$ which contains a congruence  subgroup $\brI_n$ of $\breve G$.  We say that $V$ is \emph{bounded} if $V \cap \breve \CI  w \breve \CI=\emptyset$ for all but finitely many $w \in W$.

For any compact open subgroup $\brK$ of $\breve G$, we define
\[
\dim_{\brK} V=\sup_w \dim((V \cap \overline{\breve \CI  w \breve \CI})/\brK_w)-\dim(\brK/\brK_w),
\]
where $\brK_w$ is chosen as above and such that $\brK_w\subseteq \brK$.

The previous definition is applicable in our case because of the following fact. 
\begin{theorem}\cite[Thm. A.1]{He-KR}
	Any $\s$-conjugacy class in $\breve G$ is an admissible subset. \qed
\end{theorem}
We also recall the following fact. Note that in \cite{He-CDM} the notation $X_{K, w}(b)$ has a different meaning than here.
\begin{theorem}\cite[Thm. 2.23]{He-CDM}\label{dim-1-1}
Let $[b] \in B(\bG)$.
Then for every $w\in \Admmu$,
\[
\dim_{\brI}(\brI w \brI \cap [b])=\dim X_w(b) + \<\nu_b, 2 \rho\>.
\]
Furthermore, for  a $\s$-stable parahoric subgroup $\brK$ of $\breve G$,
\[
\dim_{\brK}(\brK \Admmu \brK \cap [b])=\dim \XmubK{b}{K}+\<\nu_b, 2 \rho\>.
\]
\qed
\end{theorem}

\subsection{Closure relations of fine affine Deligne-Lusztig varieties}
We recall from \cite[\S 4]{He-Min} the partial order on ${}^K \tilde W$. 
Let $w, w' \in {}^K \tilde W$. Then $w' \preceq_{K, \sigma} w$ if there exists $x \in W_K$ such that $x w' \sigma(x)^{-1} \le w$.
The relation to the closure relation is given by the following fact.

\begin{theorem}[{\cite[Prop. 2.5]{He-11}, \cite[Thm.~2.11]{He-CDM}}]\label{He-CDM-2.11}
For $w \in {}^K \tW$, the closure of $\brK \cdot_\s \brI w\brI$ is given as follows:
\[
\overline{\brK \cdot_\s \brI w\brI}=\bigsqcup_{\{w' \in {}^K \tW\mid w' \preceq_{K, \s} w\}} \brK \cdot_\s \brI w'\brI.
\]
\qed
\end{theorem}

We also need the following fact. 
\begin{theorem}[{\cite[Thm. 2.5]{He-CDM}}]\label{He-CDM-2.5}
There is the disjoint sum decomposition into locally closed subsets, 
\[
\brK \Admmu \brK=\bigsqcup_{x \in \KAdmmu} \brK \cdot_\s \brI x \brI .
\]
Furthermore, $\dim_{\brK} (\brK \cdot_\s \brI x \brI)=\ell(x)$, for any $x \in \KAdmmu$.\qed
\end{theorem}

From these facts we can now deduce the following statement. 
\begin{proposition}\label{Irr-KAdmK}
The admissible set $\brK \Admmu \brK$ is equi-dimensional with $$\dim_{\brK}(\brK \Admmu \brK )=\langle\mu, 2 \rho\rangle.$$ The irreducible components of $\brK \Admmu \brK$ are the $\overline{\brK t^\l \brK}=\overline{\brK \cdot_\s \brI t^\l \brI}$ for $\lambda\in W_0(\umu)$ with $t^\lambda\in{}^K\tW$. 
\end{proposition}

\begin{proof}
If $t^\l \in {}^K \tW$, then the maximal element in $W_K t^\l W_K$ is $w_K t^\l$, where $w_K$ is the longest element in $W_K$. In this case, $\overline{\brK t^\l \brK}=\overline{\brI w_K t^\l \brI}$ and $\ell(w_K t^\l)=\ell(w_K)+\ell(t^\l)=\ell(w_K)+\ell(t^{\umu})$. Hence $\dim_{\brK}(\brK t^\l \brK)=\ell(t^{\umu})=\langle\mu, 2 \rho\rangle$. Moreover, $\brK \cdot_\s \brI t^\l \brI \subset \brK t^\l \brK$ and $\dim_{\brK}(\brK \cdot_\s \brI t^\l \brI)=\ell(t^\l)=\ell(t^{\umu})$. Thus $\overline{\brK t^\l \brK}=\overline{\brK \cdot_\s \brI t^\l \brI}$. 

We have $\brK \Admmu\brK = \cup_{\l \in W_0(\umu)} \overline{\brK t^\l \brK}$, and each $\overline{\brK t^\l \brK}$ is irreducible. If $\l' \in W_K(\l)$, then $\brK t^\l \brK=\brK t^{\l'} \brK$. It remains to show that for any $\l$, there exists $\l' \in W_K(\l)$ with $t^{\l'} \in {}^K \tW$. 

Let $w \in W_K$ such that $w t^\l \in {}^K \tW$. Then by definition, for any simple root $\a$ in $K$, we have that $(w t^\l) \i(\a)$ is a negative root in the affine root system. Hence $\<\l, w \i(\a)\><0$. This is equivalent to saying that $\<w(\l), \a\><0$. Hence $(t^{w(\l)}) \i(\a)$ is a negative root. Thus $t^{w(\l)} \in {}^K \tW$. This finishes the proof. 
\end{proof}

\begin{corollary}\label{cor:bound}
The dimension of $X(\mu, b)_K$ is bounded as 
$$\dim X(\mu, b)_K\leq \langle\mu, 2 \rho\rangle.$$
 If equality holds, then $b$ is basic. 
\end{corollary}
\begin{proof}
 By Theorem \ref{dim-1-1}, we have 
\begin{align*} \dim \XmubK{b}{K} & = \dim_{\brK}(\brK \Admmu \brK \cap [b])-\<\nu_b, 2 \rho\> \\ & \le \dim_{\brK}(\brK \Admmu \brK)-\<\nu_b, 2 \rho\> \\ &=\<\mu, 2 \rho\>-\<\nu_b, 2 \rho\>,\end{align*} 
where we used Proposition \ref{Irr-KAdmK} in the last line. If $\dim \XmubK{b}{K}=\<\mu, 2 \rho\>$, we have $\<\nu_b, 2 \rho\>=0$ and thus $[b]$ is the unique basic $\s$-conjugacy class in $\BGmu$.
\end{proof}
\begin{remark}
Whereas $\brK \Admmu \brK$ is equi-dimensional, the corresponding statement is not true for $ \XmubK{b}{K}$. 
\end{remark}
\section{Statement of results}\label{s:resmax}

\subsection{Criterion for maximal dimension}\label{subsec:criterion-maximal}
We introduce
\begin{equation}\label{Wfin}
\begin{aligned}
W(\umu)_{K,\fin}&=\{\lambda\in W_0(\umu)\mid t^\lambda\in{}^K\tilde W, W_{\supp_\sigma(t^\lambda)}\text{ is finite}\} \\
&=\{\lambda\in W_0(\umu)\mid t^\lambda\in\KAdmmu_0\} , 
\end{aligned}
\end{equation}
where we use  the notation of \eqref{def-sigma-support} in the first line, and of \eqref{adm0} in the last line. We simply write $W(\umu)_{\fin}$ for $W(\umu)_{\emptyset, \fin}$.  Note that  since $t^\l$ is an element of $\Admmu$ of maximal length, it is a maximal element of $\KAdmmu_0$ with respect to the partial order $\preceq_{K, \s}$.
The following theorem gives a classification of those cases when equality holds in the inequality in Corollary \ref{cor:bound}.

\begin{theorem}\label{max-dim}
    Let $\brK$ be a $\s$-stable parahoric subgroup of $\breve G$ of type $K$,
    and $[b] \in \BGmu$. If $\dim \XmubK{b}{K}=\<\mu, 2 \rho\>$, then
    $[b]=[\t]$ is basic,  ${\bf J}_\t$ is quasi-split and $\umu$ is minuscule.
    When $\brK$ is an Iwahori subgroup, then the converse holds. 

For general $\brK$, $\dim \XmubK{b}{K}=\<\mu, 2 \rho\>$ if and only if   $[b]$ is basic and $W(\umu)_{K,\fin}\neq \emptyset$.  In this case, the irreducible components of $\XmubK{b}{K}$ of dimension $\<\mu, 2 \rho\>$ are the irreducible components of $\overline{X_{K, t^\l}(b)}$, where $\l \in W(\umu)_{K,\fin}$.
\end{theorem}
The proof is given in Section \ref{s:max}. 

\subsection{Classification of maximal equi-dimensional cases}
The following theorem gives a classification of all   cases when $\XmubK{\t}{K}$ is equi-dimensional of maximal dimension. 
\begin{theorem}\label{thmmaxequi}
    Assume that $\bG$ is quasi-simple over $F$ and that $\mu$ is not central.
    Write the Tits datum of $(\bG, \mu)$ as $(\Res_{F_d/F}(\tilde{\Delta}, \s_d), (\mu_1,
    \dots, \mu_d))$.

    Then $\XmubK{\t}{K}$ is equi-dimensional of dimension equal to $\<\mu,
    2 \rho\>$ if and only if we are in one of the following cases:
	\begin{enumerate}
		\item
the tuple $(\tilde\Delta, \s_d)$ is 
$(\tilde A_{n-1},  \varrho_{n-1})$ (where $\varrho_{n-1}$ denotes rotation by $n-1$ steps), and precisely one $\mu_i$ is non-central (say $\mu_1$ is non-central), and $\mu_1 = \o^\vee_1$. Furthermore,  $K = \emptyset$,
		\item
the tuple $(\tilde\Delta, \s_d)$ is
            $(\tilde A_3, \varrho_2, \emptyset)$ (where $\varrho_2$ denotes rotation by $2$ steps)
and precisely one $\mu_i$ is non-central (say $\mu_1$ is non-central), and $\mu_1 = \o^\vee_2$. Furthermore,  $K = \emptyset$,
		\item
the tuple $(\tilde\Delta, \s_d)$ is
$(\tilde A_{n-1}, \id)$, there exist $i\ne i'$ such that $\mu_j$ is central for all $j\ne i, i'$, and $(\mu_i, \mu_{i'}) = (\o^\vee_1, \o^\vee_{n-1})$. Furthermore, $K= \emptyset$.
\end{enumerate}
\end{theorem}

The proof is given in  Section \ref{s:equimax}. 
\begin{example}
Here we consider the example of Stamm in \cite[Thm. 3]{St}. The corresponding Tits datum is $(\tilde\Delta, \{\l\})$, where $\tilde\Delta$ is of type $\tilde A_1 \times \tilde A_1$, $\tilde \BS=\{s_0, s_1, s_{0'}, s_{1'}\}$, $\l=((1, 0), (1, 0))$ and we consider the Iwahori level structure $ K=\emptyset$. The Frobenius morphism $\s$ induces a bijective map on $\tilde \BS$, which permutes $s_0$ with $s_{0'}$, and permutes $s_1$ with $s_{1'}$. Let $\t$ be the length $0$ element in $\tW$ with $\k(\t)=\k(\l)$. Then the action of $\Ad(\t)$ on $\tilde \BS$ permutes $s_0$ with $s_1$, and permutes $s_{0'}$ with $s_{1'}$. Therefore the action of $\Ad(\t) \circ \s$ permutes $s_0$ with $s_{1'}$, and permutes $s_1$ with $s_{0'}$. We have $$\Admmu=\{\t, s_0 \t, s_1 \t, s_{0'} \t, s_{1'} \t, s_0 s_{0'} \t, s_0 s_{1'} \t, s_1 s_{0'} \t, s_1 s_{1'} \t\}.$$
	
In this case, $ \brI \Admmu\brI \cap[\t] =\overline{\brI s_0 s_{1'} \t \brI} \cup \overline{\brI s_1 s_{0'} \t \brI}$ and $\overline{\brI s_0 s_{1'} \t \brI} \cap \overline{\brI s_1 s_{0'} \t \brI}=\brI \t \brI$. Hence $\XmubK{\t}{}$ has two irreducible components, both of dimension $2$ and their intersection is of dimension $0$. 
	
On the other hand, if $K=\{s_0, s_{0'}\}$, then $$\KAdmmu=\{\t, s_1 \t, s_{1'} \t, s_1 s_{1'} \t\}.$$ In this case, $ \brK \Admmu\brK \cap[\t]=\overline{\brK \cdot_\s \brI s_1 \t \brI} \cup \overline{\brK \cdot_\s \brI s_{1'} \t \brI}$ and $\overline{\brK \cdot_\s \brI s_1 \t \brI} \cap \overline{\brK \cdot_\s \brI s_{1'} \t \brI}=\brK \cdot_\s \brI \t \brI$. Hence $\XmubK{\t}{K}$ has two irreducible components, both of dimension $1$ and their intersection is of dimension $0$.
\end{example}

\begin{example}
Here we consider the case $(\tilde A_{n-1} \times \tilde A_{n-1} {}^1 \varsigma_0,, (\o^\vee_1, \o^\vee_{n-1}), \emptyset)$ for $n \ge 3$, where ${}^1 \varsigma_0$ is the automorphism of $\tilde A_{n-1} \times \tilde A_{n-1}$ which exchanges the two factors. By Theorem \ref{max-dim}, if $\tSS\setminus K$ contains $\{s_i, s_{i+1}, s_{i'}, s_{(i+1)'}\}$ for some $i$, then $\XmubK{\t}{K}$ has dimension $\<\mu, 2 \rho\>$. But only when $K=\emptyset$ is $\XmubK{\t}{K}$  equi-dimensional of dimension $\<\mu, 2 \rho\>$. 
\end{example}

\section{Critical index set}\label{s:critind}

\subsection{Critical index set} Recall that $\fka$ denotes the base alcove. For any $x \in \tW$, we define the \emph{critical index set} for $x$ by 
\begin{equation}
\text{Crit}(x)=\{v\mid \text{$v$ is a common vertex of } \fka \text{ and } x(\fka)\}.
\end{equation}

Note that if $x=w \t$ for $w \in W_a$ and $\t \in \Omega$, $\text{Crit}(x)=\text{Crit}(w)$ and this is a nonempty set if and only if $W_{\supp(w)}$ is finite. 

\subsection{Quasi-rigid set}
Let $\t \in \Omega$, i.e., a length-zero element in $\tW$. We introduce the \emph{quasi-rigid} set for $\t$ as follows, 
\begin{equation}
\text{Q-Rig}(\t)=\{w \t \text{ with } w \in W_a\mid  W_{\supp(w)} \text{ is finite}\}.
\end{equation}

In other words, $\text{Q-Rig}(\t)=\text{Q-Rig}(1) \t$ consists of all elements $x$ in $W_a \t$ such that the critical index set for $x$ is nonempty. 

\begin{figure}[h]
	\begin{tikzpicture}[line join = round, line cap = round]
	\useasboundingbox[clip] (-2.5, -2.5) rectangle (3.5, 3.5);
	
	\draw[fill=black!20] (-1, -1) -- (2, -1) -- (2, 2) -- (1, 1) -- (0, 2) -- (-1, 1) -- (0, 0) -- (-1, -1);
	
	\begin{scope}[shift={(0, 0)}]
	\foreach \i in {-3, ..., 3} {
		\draw[color=gray] (\i, -10) -- (\i, 10);
		\draw[color=gray] (-10, \i) -- (10, \i);
	}
	\end{scope}
	
	\begin{scope}[rotate=45, shift={(0, 0)}]
	\foreach \i in {-3, ..., 3} {
		\draw[color=gray] (\sqrttwo * \i, -10) -- (\sqrttwo * \i, 10);
		\draw[color=gray] (-10, \sqrttwo * \i) -- (10, \sqrttwo * \i);
	}
	\end{scope}
	
	\draw[fill=black] (0, 0) circle[radius=0.08];
	
	\draw[ultra thick] (0, 0) -- (1, 0) -- (1, 1) -- (0, 0) node[shift={(.58, .25)}] {\fka};
	
	\draw[thick] (-1, -1) -- (1, -1) -- (2, 0) -- (2, 2) -- (0, 2) -- (0, 1) -- (-1, 1) -- (-1, -1);
	
	\end{tikzpicture}
\caption{\it Admissible set (shaded gray) for $\tilde B_2$, $\mu = \omega_1^\vee$  and quasi-rigid set for $\t=\t(t^{\umu})$ (inside the thick lines).}
\end{figure}
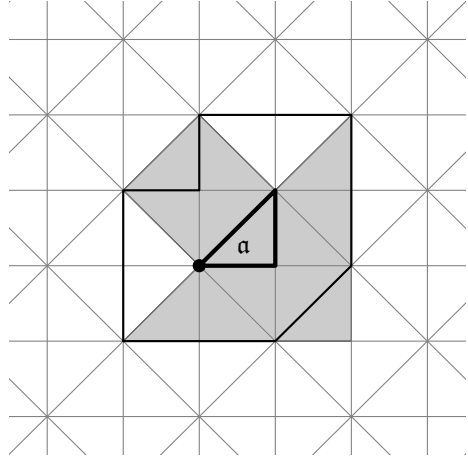

For any length preserving automorphism $\th$ of $\tW$, we  introduce the $\theta$-\emph{rigid set} for $\t$,  
 \begin{equation}\text{Rig}(\t, \th)=\{x \in W_a \t\mid W_{\supp_{\th}(x)} \text{ is finite}\},
 \end{equation}
 cf.~\cite{CH}.
Note that 
\begin{equation*}
\begin{aligned}
\supp(w) &\subset \supp_{\th}(w \t)=\cup_{i \in \BZ} (\Ad(\t) \circ \th)^i \supp(w), \\
\supp(w)&=\supp_{\Ad(\t) \i} (w \t) .
\end{aligned}
\end{equation*}
 Hence 
\begin{enumerate}
	\item For any length preserving automorphism $\th$ of $\tW$, $\text{Q-Rig}(\t) \supset \text{Rig}(\t, \th)$;
	
	\item $\text{Q-Rig}(\t)=\text{Rig}(\t, \Ad(\t) \i)$. 
\end{enumerate}

The following theorem compares $\KAdmmu$ and $\text{Q-Rig}(\t)$. 

\begin{theorem}\label{critadm}
	Assume that $\tW$ is irreducible. Let $K \subset \tSS$ with $W_K$ finite, i.e., $K\neq \tilde\BS$. Then $\KAdmmu \subset \text{Q-Rig}(\t)$ if and only if $(\tilde\Delta, \s,\mu)=(\tilde A_{n-1},\varsigma_1, \o^\vee_1)$ (up to isomorphism), in which case   $\KAdmmu = \text{Q-Rig}(\t) \cap {}^K \tW$. 
\end{theorem}

\begin{remark}
The case where $K=\emptyset$ is Proposition~\ref{supp-tss}. The proof of that proposition does not show the general case since there are less elements in $\KAdmmu$ as $K$ becomes larger. Therefore we have to use more advanced techniques here.
\end{remark}

\begin{proof}
	Let $\bf H$ be a connected reductive group over $F$ with Iwahori-Weyl group over $\brF$ isomorphic to $\tW$ and where the induced action of the Frobenius on $\tW$ equals  $\Ad(\t) \i$. 
	Then we have by (2) above $\text{Q-Rig}(\t)=\text{Rig}(\t, \s)$. Hence, by assumption,  for any $x \in \KAdmmu$, $W_{\supp_\s(x)}$ is finite. Hence, by Proposition \ref{GHN-4.6}, $\brK \cdot_\s \brI x \brI \subset [\t]$. By \eqref{decXmu}, we see that $X(\mu, b)_K=\emptyset$ if $b$ is not basic. By Theorem \ref{KR-conj}, $\BGmu=\{[\t]\}$ is then  a singleton. Then by \cite[\S 6]{Ko2}, $(\tilde\Delta, \mu)=(\tilde A_{n-1},\varsigma_1, \o^\vee_1)$ (up to isomorphism). 
\end{proof}
\begin{remark}
The  concept of critical index is due to Drinfeld \cite{Dr}. The fact that in the Drinfeld case $(\tilde A_{n-1},\varsigma_1, \o^\vee_1)$ any element of $\KAdmmu$ has a critical index is crucial in his proof of $p$-adic uniformization of the Drinfeld RZ-space. The proof in loc.~cit. is by linear algebra. Note that Theorem \ref{critadm} answers the question raised in \cite[\S 3]{RZ:indag}. 
\end{remark}

Note that the study of $\text{Q-Rig}(\t)$ can be reduced to the case where $G$ is adjoint and $\tW$ is irreducible. The following result describes the translation elements in $\text{Q-Rig}(\t)$ in the case where $\tW$ is irreducible.

\begin{proposition}\label{l-supp}
Suppose that $\tW$ is irreducible. Let $t^\l$ be a translation element in $\tW$,
and let $\t \in \Omega$ with $t^\l \in W_a \t$. Then $t^\l \in \text{Q-Rig}(\t)$
if and only if there exists a length preserving automorphism $\th$ of $\tW$ such
that $\th(\l)$ is a dominant minuscule coweight.

Furthermore, if $t^\l$ is non-central, then $t^\l$ has exactly one critical index, and the critical index corresponds to a special vertex.
\end{proposition}

As the proof will show, if $\bG$ is adjoint and $\theta$ exists, then it can be chosen as conjugation by a length $0$ element of $\tW$.

\begin{proof}
Let $\fka'=t^\l(\fka)$ be the alcove obtained from the base alcove $\fka$ by translation. Then $t^\l \in \text{Q-Rig}(\t)$ if and only if $\fka$ and $\fka'$ have a common vertex, say $v$. 
	
	Note that the vertices of $\fka$ are $\frac{\o^\vee_i}{\<\o^\vee_i, \b\>}$ for $i \in \BS$ and $0$. Here $\b$ is the highest root and $\o^\vee_i$ is the fundamental coweight associated to $i$. Thus the vertices of $\fka'$ are $\frac{\o^\vee_j}{\<\o^\vee_j, \b\>}+\l$ for $j \in \BS$ and $\l$. Then we have 
	\begin{enumerate}
		\item Either $v=\frac{\o^\vee_i}{\<\o^\vee_i, \b\>}$ and $\l=\frac{\o^\vee_i}{\<\o^\vee_i, \b\>}-\frac{\o^\vee_j}{\<\o^\vee_j, \b\>}$ for some $i \neq j \in \BS$;
		
		\item or $v=\l=\frac{\o^\vee_i}{\<\o^\vee_i, \b\>}$; 
		
		\item or $v=0$ and $\l=-\frac{\o^\vee_j}{\<\o^\vee_j, \b\>}$; 
		
		\item or $v=\l=0$. 
	\end{enumerate}
	
	In case (1), we have $\frac{1}{\<\o^\vee_i, \b\>}=\<\l, \a_i\> \in \BZ$ and $\frac{1}{\<\o^\vee_j, \b\>}=-\langle\l, \a_j\rangle \in \BZ$, where $\a_i$ is the simple root associated to the simple reflection $s_i$. Thus both $\o^\vee_i$ and $\o^\vee_j$ are minuscule coweights. Hence both $v$ and $v-\l$ are special vertices in the base alcove. In cases (2)-(4), one may show by a similar (but easier) argument that $v$ and $v-\l$ are still special vertices in the base alcove. 
	
    The group of length $0$ elements acts transitively on the set of special vertices of $\mathfrak a$, so after applying the length preserving automorphism of $\tW$ induced by such an element, we may assume that $v-\l$ is the origin in the base alcove. In other words, $v=\l$ is a special vertex in the base alcove and hence $\l$ is a minuscule coweight (recall that we excluded the possibility that $\l$ is central in our assumptions).
\end{proof}
\begin{corollary}\label{8-4}
Assume that $\tW$ is irreducible and $t^\l$ is non-central. If $K \subset \tSS$ with $K\supsetneq \supp(t^\l\tau^{-1})$, then $K=\tSS$. 
\end{corollary}
\begin{proof}
By Proposition \ref{l-supp}, $\supp(t^\l \t \i)=\tSS$ or $\tSS\setminus\{s\}$ for some simple reflection $s$, corresponding to a special vertex. Thus if $K \supsetneq \supp(t^\l\tau^{-1})$, then $\supp(t^\l \t \i)=\tSS\setminus\{s\}$ and $K=\tSS$.  
\end{proof}
\section{Maximal dimension}\label{s:max}
In this section, we prove Theorem \ref{max-dim}. 
\subsection{Preparations} 
The following result gives an explicit description of the set $W(\umu)_{\fin}$ introduced in Section~\ref{subsec:criterion-maximal}.

\begin{proposition}\label{W-fin}
Suppose that $G$ is quasi-simple over $F$, i.e., $\s$ acts transitively on the set of irreducible components of $\tW$. Suppose that $\umu$ is non-central in $G$, i.e., the restriction of $\umu$ to some irreducible component of $\tW$ is non-central. Then 
\[
    W(\umu)_{\rm fin} = \{ \l \in W_0(\umu) \mid t^\l \text{ has an $\Ad(\t)\circ\s$-stable critical index}\}.
\]

In particular, for any $\l \in W(\umu)_{\rm fin}$, $\lambda$ is minuscule, $t^\l$ has a unique $\Ad(\t)\circ\s$-stable critical index, and the corresponding vertex is special. 
\end{proposition}

\begin{proof}
Without loss of generality, we may assume that $G$ is adjoint. In this case $\tW=\tW_1 \times \tW_2 \times \cdots \times \tW_d$ and $\tSS=\tSS_1 \times \tSS_2 \times \cdots \tSS_d$, where $\tW_1 \cong \tW_2 \cong \ldots  \cong \tW_d$ are irreducible. We have $\umu=(\mu_1, \mu_2, \ldots, \mu_d)$. We may assume that $\mu_1$ is non-central in $\tW_1$. Let $\t=(\t_1, \t_2, \ldots, \t_d)$. 

For any subset $K\subseteq \tSS$, $W_K$ is finite if and only if in each component of the Dynkin diagram, there is at least one vertex not contained in $K$. Hence, as we have remarked before, $\l$ has a critical index if and only if $W_{\supp(t^\l\t^{-1})}$ is finite.
In case the critical index is unique, we have that $\supp(t^\l) = \supp_\s(t^\l)$ if and only if the critical index is $\Ad(\t) \circ \s$-stable.

Since $\umu$ is non-central, elements of $W_0(\umu)$ have at most one critical index, and we obtain that the right hand side is a subset of $W(\umu)_{\fin}$.

Conversely, let $\l=(\l_1, \l_2, \ldots, \l_d) \in W_0(\umu)$ be an element in $W(\umu)_{\rm fin}$. By Proposition~\ref{l-supp}, $\mu_1$ is minuscule, $\l_1$ is of the form $\th_1(\mu_1)$ and $t^{\l_1}$ has a unique critical index. Note that $\supp(t^{\l_1} \t_1 \i)=\tSS_1\setminus\{s_1\}$ for some simple reflection $s_1$ that corresponds to the critical index of $t^{\l_1}$. For $1 \le i \le d$, let $s_i=(\Ad(\t) \circ \s)^{i-1}(s_1) \in \tW_i$. Then $\tSS\setminus \{s_1, s_2, \ldots, s_d\} \subset \supp_{\s}(t^\l)$. Note that for any $K \supsetneq \tSS\setminus\{s_1, s_2, \ldots, s_d\}$, $W_K$ is an infinite group. Thus we have $\tSS\setminus\{s_1, s_2, \ldots, s_d\}=\supp_{\s}(t^\l)$. In particular, $(\Ad(\t) \circ \s)^d(s_1)=s_1$. And for each $1 \le i \le d$, either $\l_i$ is central or $\l_i$ is minuscule non-central and $s_i$ is the simple reflection corresponding to the critical index of $t^{\l_i}$. Hence $t^\l$ has a critical index which corresponds to $s_1 s_2 \cdots s_d$. Moreover, by  construction, this is the unique $\Ad(\t)\circ \s$-stable critical index.

The final part follows from Proposition~\ref{l-supp}, or from the equality of the two sets, since all elements of the right hand side have these properties.
\end{proof}

\begin{proposition}\label{J-quasisplit}
    The set $W(\umu)_{\fin}$ is nonempty if and only if $\bJ_\t$ is quasi-split
    and $\umu$ minuscule.
\end{proposition}

\begin{proof}
Since $[\t]$ is basic, $\bJ_\t$ is an inner form of $\bG$. It is quasi-split if
and only if there exists a collection $\Pi\subset\tSS$ of special vertices, one
in each connected component of the affine Dynkin diagram, such that
$\Ad(\t)\circ\s(\Pi) = \Pi$, i.e., the subset is fixed by the twisted Frobenius
corresponding to $\bJ_\t$. If $W(\umu)_{\fin}$ is nonempty then $\umu$ is minuscule and
Proposition~\ref{W-fin} implies that $\bJ_\t$ is quasi-split.

Conversely, suppose that $\bJ_\t$ is quasi-split and that $\umu$ is minuscule, so that $t^{\umu}$ has a critical index. Applying
Proposition~\ref{W-fin}, it is enough to show that
with $\Pi\subset\tSS$ as above, there exists a length preserving automorphism
$\th$ of $\tW$ and $\l \in W_0(\umu)$ such that $\theta(\l)=\umu$ and $\Pi = \tSS
\setminus \supp(t^\l)$. We may assume that $\bG$ is adjoint.  Then the subgroup of
length $0$ elements of $\tW$ acts transitively on the set of special vertices of
the base alcove. Let $\th$ be induced by a length $0$ element and such that $\l
:= \th^{-1}(\umu)$ satisfies $\Pi = \tSS \setminus \supp(t^\l)$. Then $\l \in
W_0(\umu)$ and hence $\l\in W(\umu)_{\fin}$.
\end{proof}

\subsection{Proof of Theorem~\ref{max-dim}}
First assume that $b$ is basic and $W(\umu)_{K, \fin}\ne
\emptyset$. By Proposition~\ref{GHN-4.6}, $\brK \cdot_\s \brI t^\l \brI \subset [\t]$ for $\l\in W(\umu)_{K, \fin}$.  By Theorems \ref{dim-1-1} and  \ref{He-CDM-2.5}, we see that  $\dim \XmubK{b}{K} = \<\mu, 2\rho\>$.
For $K=\emptyset$, if $\bJ_\t$ is quasi-split and $\umu$ is minuscule, then
Proposition ~\ref{J-quasisplit} shows $W(\umu)_{\fin}\ne \emptyset$ and hence $\dim
\XmubK{b}{K} = \<\mu, 2\rho\>$.

Now suppose that  $\dim \XmubK{b}{K} = \<\mu, 2\rho\>$. By Corollary~\ref{cor:bound},  $[b] = [\t]$ is basic. 
We next claim that the irreducible components of $\XmubK{\t}{K}$ of dimension $\<\mu, 2\rho\>$ are the irreducible components of the $\overline{X_{K, t^\l}(\t)}$ of dimension $\<\mu, 2\rho\>$, where $\l\in W(\umu)_{K, \fin}$.
Indeed,  by~\eqref{decXmu},
\[
\XmubK{\t}{K}=\bigsqcup_{x \in \KAdmmu} X_{K, x}(\t).
\]
Now for $x \in \KAdmmu$, $\dim X_{K, x}(\t) \le \dim X_x(\t) = \dim_{\brI}(\brI x \brI \cap [\t]) \le \dim_{\brI}(\brI x \brI)=\ell(x)$, using Theorem~\ref{dim-1-1} for the first and Theorem~\ref{He-CDM-2.5} for the final equality, which proves the claim. In particular, $W(\umu)_{K, \fin}\ne\emptyset$.
On the other hand, $\overline{X_{K, t^\l}(\t)}$ is equi-dimensional. In fact, $X_{K, t^\l}(\t)$ is a disjoint union of copies of a classical Deligne-Lusztig variety by~\cite[Prop.~5.7]{GHN}, \cite[Thm.~4.1.1, Thm.~4.1.2]{GH}.

Finally, the map $X(\mu,b)\to X(\mu,b)_{K}$ is surjective, cf.~\cite[Thm. 1.1]{He-KR}. Hence we deduce from $\dim \XmubK{b}{K} = \<\mu, 2\rho\>$ that $\dim X(\mu, b) = \<\mu, 2\rho\>$. The previous reasoning applied to $K=\emptyset$ implies  $W(\umu)_{\fin}\neq\emptyset$, and hence we deduce from Proposition~\ref{J-quasisplit} that $\bJ_\t$ is quasi-split and $\umu$ minuscule. Theorem~\ref{max-dim} is proved.\qed

\begin{remark}
For any $(\bG, \mu)$ such that $\umu$ is minuscule, there exists an inner form $\bf H$ of $\bG$ such that $\dim X^{\bf H}(\mu,
\t)=\<\mu, 2 \rho\>$, namely the one with Frobenius $\Ad(\t)\circ \s$. In
particular, this applies when $\bG$ splits over $\breve F$, because then $\umu = \mu$. 
\end{remark}
\section{Maximal equi-dimension}\label{s:equimax}
In this section, we prove Theorem \ref{thmmaxequi}. 
\subsection{Reduction to the fully Hodge-Newton decomposable case}\label{10.1(a)}
Suppose that $\XmubK{b}{K}$ is equi-dimensional of dimension equal to $\<\mu, 2 \rho\>$. By Theorem \ref{max-dim}, $[b]=[\t]$ is basic and 
\[
\XmubK{\t}{K}=\bigcup_{\l \in W(\umu)_{K,\fin}} \overline{X_{K, t^\l}(\t)}.
\]

We claim that $({\bf G}, \mu)$ is of fully Hodge-Newton decomposable type. In fact, by~Theorem~\ref{GHN-2.3} it is enough to show that whenever $w\in \KAdmmu$ satisfies $X_{K, w}(\t)\ne \emptyset$, then $W_{\supp_\s(w)}$ is finite.  But then $X_{K, w}(\t)\subseteq \XmubK{\t}{K}$ and the above gives $X_{K, w}(\t) \subseteq \overline{X_{K, t^\l}(\t)}$ for some $\l\in W(\umu)_{K, \fin}$.
Now Theorem~\ref{He-CDM-2.11} shows that
\[
\overline{\brK \cdot_\s \brI t^\l \brI}=\bigsqcup_{\{x \in {}^K \tW\mid x \preceq_{K, \s} t^\l\}} \brK \cdot_\s \brI x\brI,
\]
and this implies that
\[
\overline{X_{K, t^\l}(\t)} \subseteq \bigsqcup_{\{x \in {}^K \tW\mid x \preceq_{K, \s} t^\l\}} X_{K, x}(b).
\]
We obtain that $w \preceq_{K, \s} t^\l$, for some $\l\in W(\umu)_{K, \fin}$. This implies $\supp_\s(w)\subseteq \supp_\s(t^\l)$, so $W_{\supp_\s(w)}$ is finite.

Hence by Theorem~\ref{fine-dec},
\[
\XmubK{\t}{K}=\bigsqcup_{x \in \KAdmmu_0} X_{K,x}(b).
\]
In particular, we have that $\XmubK{b}{K}$ is equi-dimensional of dimension equal to $\<\mu, 2 \rho\>$ if and only if the following condition is satisfied. 
\smallskip

\emph{$(\star)$ The set of maximal elements of $\KAdmmu_0$ with respect to the partial order $\preceq_{K, \s}$ is equal to $\{t^\l\mid  \l \in W(\umu)_{K, \fin}\}$.}

We first check which cases satisfy $(\star)$ under the additional assumption that $\mu$ is non-central in every irreducible component: In Sections~\ref{subsec:star-irred}, \ref{subsec:star-irred-2} we go through the irreducible cases, and in Section~\ref{ss:reduc} we check the remaining case, the Hilbert-Blumenthal case. Finally, in Section~\ref{subsec:star-general} we explain how to deduce the general case where $\mu$ is allowed to have central components.

\subsection{Candidates for the irreducible cases}\label{subsec:star-irred}
We first consider the case where $\tW$ is irreducible. Since $\XmubK{\t}{K}$ has dimension $\<\mu, 2\rho\>$,  we have $W(\umu)_{K, \fin} \neq \emptyset$. By Proposition \ref{W-fin}, $\Ad(\t) \circ \s$ fixes a special vertex in the affine Dynkin diagram of $\tW$. The fully Hodge-Newton decomposable cases with $\tW$ irreducible and  where $\Ad(\t) \circ \s$ fixes a special vertex can be extracted from the table in Theorem \ref{f-HN}, and are as follows (see the explanation after Theorem~\ref{f-HN-2} for the notation):
	
	\begin{enumerate}[label=(\roman*)]
		\item $(\tilde A_{n-1}, \varrho_{n-1}, \o^\vee_1)$ for $n \ge 2$;
		
		\item $(\tilde A_{2m}, \varsigma_0, \o^\vee_1)$ for $m \ge 1$;
		
		\item $(\tilde A_3, \varsigma_0, \o^\vee_2)$;
		
		\item $(\tilde A_3, \varrho_2, \o^\vee_2)$;
		
		\item $(\tilde B_n, \Ad(\t_1), \o^\vee_1)$ for $n \ge 3$;
		
		\item $(\tilde C_2, \Ad(\t_2), \o^\vee_2)$;
		
		\item $(\tilde D_n, \varsigma_0, \o^\vee_1)$ for $n \ge 4$.
	\end{enumerate}

\smallskip

Next we check when the condition \S \ref{10.1(a)} $(\star)$  is satisfied. 	
	
\smallskip

\subsection{Case-by-case analysis}\label{subsec:star-irred-2}
\subsubsection{$(\tilde A_{n-1}, \Ad(\t_{n-1}), \o^\vee_1)$ for $n \ge 2$} Here the only possible $K$ is $\emptyset$ and $\brK=\brI$. This is the Drinfeld case and $\BGmu$ consists of a single element, namely, $[\t]$. In this case $\brI \Admmu \brI \subset [\t]$ and $\XmubK{\t}{}$ is equi-dimensional of dimension equal to $\<\mu, 2 \rho\>$. 

\subsubsection{$(\tilde A_{2m}, \varsigma_0, \o^\vee_1)$ for $m \ge 1$} In this case, $\tSS^{\Ad(\t) \circ \s}=\{s_{m+1}\}$. Thus the only translation element in $\Admmu_0$ is $t^\l$, where $\l=\Ad(\t_n)(\o^\vee_1) \in {}^{\tSS\setminus \{s_m\}} \tW$ and $\supp(t^\l \t \i)=\tSS\setminus \{s_{m+1}\}$. Therefore if $\l \in W(\umu)_{K, \fin}$, then $K \subset \tSS\setminus \{s_m\}$. Since $K=\s(K)$, we have $K \subset \tSS\setminus \{s_m, s_{m+1}\}$. In this case, $s_{m+1} \t \in \KAdmmu_0$ and $s_{m+1} \t \npreceq_{K, \s} t^\l$. This contradicts \S \ref{10.1(a)} $(\star)$ .
	
\subsubsection{$(\tilde A_3, \varsigma_0, \o^\vee_2)$} In this case, $\tSS^{\Ad(\t) \circ \s}=\{s_1, s_3\}$. Thus the only translation elements in $\Admmu_0$ are $s_1 s_2 s_0 s_1 \t$ and $s_3 s_2 s_0 s_3 \t$. Therefore if $W(\umu)_{K, \fin} \neq \emptyset$, then $s_1 \notin K$ or $s_3 \notin K$. Since $K=\s(K)$, both $s_1$ and $s_3$ are not in $K$. In this case, $s_1 s_3 \t \in \KAdmmu_0$ and $s_1 s_3 \t \npreceq_{K, \s} s_1 s_2 s_0 s_1 \t$ and $s_1 s_3 \t \npreceq_{K, \s} s_3 s_2 s_0 s_3 \t$. This contradicts \S \ref{10.1(a)} $(\star)$ .
	
\subsubsection{$(\tilde A_3, \Ad(\t_2), \o^\vee_2)$} We first consider the case where $K=\emptyset$. In this case, the maximal elements in $\KAdmmu_0$ are $s_2 s_1 s_3 s_2 \t$, $s_3 s_2 s_0 s_3 \t$, $s_0 s_1 s_3 s_0 \t$ and $s_1 s_2 s_0 s_1 \t$ and the condition \S \ref{10.1(a)} $(\star)$  is satisfied.
    
If $K=\{s_0, s_2\}$, then the maximal elements in $\KAdmmu_0$ are $s_3 s_2 s_0 s_3 \t$, $s_1 s_2 s_0 s_1 \t$, $s_1 s_3 s_0 \t$ and $s_1 s_3 s_2 \t$. This contradicts \S \ref{10.1(a)} $(\star)$ .
	
\subsubsection{$(\tilde B_n, \Ad(\t_1), \o^\vee_1)$ for $n \ge 3$} By Proposition \ref{W-fin}, $W(\umu)_{\fin}=\{\o^\vee_1, \Ad(\t_1)(\o^\vee_1)\}$. Note that \begin{gather*} 
		\text{$t^{\o^\vee_1} \in {}^{\BS} \tW$ and $\supp(t^{\o^\vee_1} \t \i)=\tSS\setminus \{s_1\}$}; \\
		\text{$t^{\Ad(\t_1)(\o^\vee_1)} \in {}^{\tSS\setminus \{s_1\}} \tW$ and $\supp(t^{\Ad(\t_1)(\o^\vee_1)} \t \i)=\BS$}.
	\end{gather*} 
		
	Thus if $\KAdmmu_0$ contains some of these translations elements and $K=\s(K)$, then $K \subset \tSS\setminus \{s_0, s_1\}$. In this case, $s_0 s_1 \t \in \KAdmmu_0$ and $s_0 s_1 \t \npreceq_{K, \s} t^{\o^\vee_1}$ and $s_0 s_1 \t \npreceq_{K, \s} t^{\Ad(\t_1)(\o^\vee_1)}$. This contradicts \S \ref{10.1(a)} $(\star)$ .
	
\subsubsection{$(\tilde C_2, \Ad(\t_2), \o^\vee_2)$} In this case, $\tSS^{\Ad(\t) \circ \s}=\{s_0, s_1, s_2\}$. The only translation elements in $\Admmu_0$ are $s_0 s_1 s_0 \t$ and $s_2 s_1 s_2 \t$. Therefore if $W(\umu)_{K, \fin} \neq \emptyset$, then $s_0 \notin K$ or $s_2 \notin K$. Since $K=\s(K)$, both $s_0$ and $s_2$ are not in $K$. In this case, $s_0 s_2 \t \in \KAdmmu_0$ and $s_0 s_2 \t \npreceq_{K, \s} s_0 s_1 s_0 \t$ and $s_0 s_2 \t \npreceq_{K, \s} s_2 s_1 s_2 \t$. This contradicts \S \ref{10.1(a)} $(\star)$ .
	
\subsubsection{$(\tilde D_n, \varsigma_0, \o^\vee_1)$ for $n \ge 4$} In this case, the special vertices that are fixed by $\Ad(\t) \circ \s$ are $n-1$ and $n$. By Proposition~\ref{l-supp} and Proposition \ref{W-fin}, the elements of $W(\umu)_{\fin}$ are of the form $\th(\umu)$, where $\th$ runs over length preserving automorphism such that $\th \circ \Ad(\t)(\BS)$ is $\Ad(\t) \circ \s$-stable. In this case, $\th$ sends the vertices $\{0, 1\}$ to the vertices $\{n-1, n\}$. We have that $K \subset \tSS\setminus \{s_{n-1}\}$ or $K \subset \tSS\setminus \{s_n\}$. Since $K=\s(K)$, we have $K \subset \tSS\setminus \{s_{n-1}, s_n\}$. Then we have $s_{n-1} s_n \t \in \KAdmmu_0$. On the other hand, we have $\supp(t^{\th(\umu)} \t \i) \subset \tSS\setminus \{s_{n-1}\}$ or $\supp(t^{\th(\umu)} \t \i) \subset \tSS\setminus \{s_n\}$. Thus $s_{n-1} s_n \t \npreceq_{K, \s} t^{\th(\umu)}$. This contradicts \S \ref{10.1(a)} $(\star)$ .
	
\subsection{Reducible case}\label{ss:reduc} We consider the case where $\tW$ is reducible, cf. Theorem \ref{f-HN-2}. Let us first assume that $\mu$ is non-central in each factor, so it is of type $(\tilde A_{n-1} \times \tilde A_{n-1}, {}^1 \varsigma_0, (\o^\vee_1, \o^\vee_{n-1}))$. There are two copies of the affine Dynkin diagram of type $\tilde A_{n-1}$, and we label the vertices by $i$ and $i'$ respectively, where $i, i' \in \BZ/n \BZ$. The Frobenius $\s$ acts by ${}^1 \varsigma_0$, which exchanges the vertex $i$ with $i'$ for any $i$. The $\Ad(\t) \circ\sigma$-orbits on $\tSS$ are $\{s_i, s_{(i-1)'}\}$ for $i \in \BZ/n \BZ$. If $K=\emptyset$, then the maximal elements in $\KAdmmu_0$ are $(s_i s_{i-1} \cdots s_{i-n+2}) (s_{(i-n+1)'} \cdots s_{(i-2)'} s_{(i-1)'}) \t$ for $i \in \BZ/n \BZ$. They are all translation elements. Hence the condition \S \ref{10.1(a)} $(\star)$  is satisfied.
	
	Now suppose that $K \neq \emptyset$. Without loss of generality, we may assume that $\{s_0, s_{0'}\} \subset K$. Then $(s_{n-1} s_{n-2} \cdots s_2) (s_{1'} s_{2'} \cdots s_{(n-1)'}) \t$ is a maximal element in $\KAdmmu_0$. This contradicts \S \ref{10.1(a)} $(\star)$ .

\subsection{The general case}\label{subsec:star-general}

Finally, let us reduce the general case to the case where $\umu$ is non-central in each component. Given $(\bG, \mu)$, we may assume that $\bG$ is adjoint, and we construct $(\bG', \mu')$ as in Section~\ref{subsec:ghn-2.4-construction}. Since we have already shown that $(\bG, \mu)$ is fully Hodge-Newton decomposable, $\mu$ is minute. This implies that $\mu'$ is minute, and hence we see that the Dynkin type of $(\bG', \mu')$ is one of the types in Theorem~\ref{f-HN}. The only possibilities for $(\bG, \mu)$ then are
\begin{itemize}
    \item 
        All $\mu_\varphi$, except for one, are central, and the component where $\mu$ is non-central is as in Theorem~\ref{f-HN}, or
    \item
        All $\mu_\varphi$, except for two, are central, and the two components where $\mu$ is non-central give rise to the Hilbert-Blumenthal case $(\tilde A_{n-1} \times \tilde A_{n-1}, (\o^\vee_1, \o^\vee_{n-1}))$.
\end{itemize}

The components where $\mu$ is central do not contribute to the set $\KAdmmu_0$, so that the analysis whether condition $(\star)$ is satisfied is exactly the same as in the previous sections.

\section{Lattice interpretation of the maximal equi-dimensional cases}\label{s:lattice-interpretation-max}
In this section, we go through the list of Theorem \ref{thmmaxequi} under the assumption that $\umu$ is non-central in each factor of $\tilde W$ and give lattice interpretations of $\XmubK{\t}{K}$ in each case. 
\subsection{The Drinfeld case}\label{ss:Drinfeld}
Let $(N, \phi )$ be a $\breve F$-vector space of dimension $n$, equipped with  a $\sigma$-linear automorphism isoclinic of slope $0$. Then we have
\begin{equation}\label{DrlatticeinN}
\XmubK{\t}{K}= \bigsqcup_{v\in\BZ} \{ M_\bullet\mid M_{i+1}\supset \phi( M_i),\forall i, \vol(M_0)=v\}.
\end{equation}

Here $M_\bullet$ is a periodic $O_{\breve F}$-lattice chain with period $n$. The decomposition indexed by $v$ corresponds to the decomposition of the affine flag variety into connected components. 

In this case, we obtain a $\pi$-adic formal scheme, as follows. We fix the following \emph{relative} rational RZ-data $\CD$ of EL-type. Let $B$ be a central division algebra over $F$ with invariant $1/n$. Let $V$ be a free $B$-module of rank one. Let $\breve V=V\otimes_F\breve F$. Then $b\in\GL_B(\breve V)$ is such that the relative isocrystal $(\breve V, b (\id\otimes \sigma))$ is isoclinic of slope $1/n$. The conjugacy class $\mu$ is given by $(1,0,\ldots,0)$ in an identification of $\GL_B(V)$ with $\GL_n$ after extension of scalars to $\ov F$.  The relative integral RZ-data $\CD_{O_F}$ are given by the maximal order $O_B$ of $B$ and the periodic lattice chain $\CL=\{\Pi^iO_B\mid i\in\BZ\}$. Here $\Pi$ denotes a uniformizer in $O_B$. 

In this case, there is a unique \emph{special formal $O_B$-module} of $F$-height $n^2$ over any algebraically closed extension of the residue field $k$ of $F$, cf. \cite[Lem. 3.60]{RZ96}. Taking any one of these as a \emph{framing object}  over $\ov k$, we obtain an RZ-space $\CN_{\CD_{O_F}}$ over $\Spf O_{\breve F}$ which parametrizes special formal $O_B$-modules together with a quasi-isogeny framing. It is a $\pi$-adic formal scheme \cite[Prop. 3.62]{RZ96}, flat over $O_{\breve F}$ \cite[3.69, Thm. 3.72]{RZ96}. Setting $\bG=\GL_B(V)$, we obtain a tuple $(\bG, \mu, b, K)$, where $K$ is the parahoric subgroup stabilizing the lattice chain $\CL$ (note that $\bG$ is the algebraic group over $F$ associated to $B^\times$). To identify the connected component $\CN^o_{\CD_{O_F}}(\ov k)$ of height zero elements with  \eqref{DrlatticeinN} for $v=0$, let $\tilde F/F$ be an unramified subfield of $B$ of degree $n$, with a fixed embedding $\tilde F\hookrightarrow \breve F$, and assume that $\Pi$ satisfies $\Pi^n=\pi$ and that $\Pi$ normalizes $\tilde F$ and induces on $\tilde F$ the Frobenius generator of the Galois group $\Gal(\tilde F/F)$. Let 
$$
\breve V=\bigoplus_{k\in\BZ/n}\breve V_k
$$
be the eigenspace decomposition under $\tilde F$. Then $\sigma$ is an endomorphism of degree $1$, and so is $\Pi$. Then set $N=\breve V_0$, $\phi=\Pi^{-1} \big(b(\id\otimes\sigma)\big)$. Similarly, the decomposition $O_{\tilde F}\otimes_{O_F}O_{\breve F}=\oplus_{k\in\BZ/n}O_{\breve F}$ induces for each $i\in\BZ$ a decomposition of $\breve \CL_i=\CL_i\otimes_{O_F} O_{\breve F}$, 
$$
\breve\CL_i=\bigoplus_{k\in\BZ/n} \breve\CL_{i, k} .
$$ 
Then the lattice chain $M_\bullet$ in \eqref{DrlatticeinN} is given as $M_i=\breve \CL_{i, 0}$. 
 
\subsection{The $D_{2/4}$-case}
Let $(N, \phi )$ be an isocrystal of dimension $4$, where $\phi $ is a $\sigma$-linear automorphism isoclinic of slope $0$. Then we have
\begin{equation}
\XmubK{\t}{K}= \bigsqcup_{v\in\BZ} \{ M_\bullet\mid M_{i+2}\supset \phi (M_i),\forall i, \vol(M_0)=v\}.
\end{equation}
Here $M_\bullet$ is a periodic lattice chain with period $4$. The decomposition indexed by $v$ corresponds to the decomposition of the affine flag variety into connected components.

\subsection{The Hilbert-Blumenthal case}
Let $(N, \phi )$ be a $\sigma^{2}$-isocrystal of dimension $n$, where $\phi $ is a $\sigma^{2}$-linear automorphism isoclinic of slope $0$.  Then we have
\begin{equation}
\XmubK{\t}{K}= \bigsqcup_{v\in\BZ} \{ (M_\bullet\ M'_\bullet) \mid \pi\phi( M_i)\subset M'_{i}\subset^1 M_i,\forall i, \vol(M_0)=v\}.
\end{equation}
Here $M_\bullet$ and $M'_\bullet$  are maximal  periodic lattice chains in $N$.  The decomposition indexed by $v$ corresponds to the decomposition of the affine flag variety into connected components.

\section{Application to $p$-adic uniformization}\label{s:padic}
As explained in Subsection \ref{ss:Drinfeld}, the RZ-space corresponding to the case (1) of Theorem \ref{thmmaxequi} is $\pi$-adic. In this section we explain various criteria which show that in the cases (2) and (3) of Theorem \ref{thmmaxequi}, the corresponding RZ-spaces are not $\pi$-adic formal schemes. Here, we implicitly appeal to the uniqueness result 
\cite[Prop. 4.4]{HPR} that the RZ-space (which a priori depends on \emph{ integral {\rm RZ}-data} $\CD_{\BZ_p}$, cf.~loc.~cit) only depends on the tuple $(\bG, \mu, b, K)$. To apply this result, we assume that $\bG$ splits over a tamely ramified extension of $F$. 

\subsection{Via change of parahoric}\label{ss:change} 
We note the following consequence of Theorem \ref{thmmaxequi}.
\begin{corollary}\label{charDrin}
    Assume that $\bG$ is quasi-simple over $F$ and that $\mu$ is non-central. Then $\XmubK{\t}{K}$ is equi-dimensional of dimension equal to $\<\mu, 2 \rho\>$ for \emph{every} parahoric subgroup $K$ if and only if the pair $(\tilde\Delta, \s)$ is isomorphic to $\Res_{F_d/F}(\tilde A_{n-1}, \varrho_{n-1})$ (where as before $\varrho_{n-1}$ denotes rotation by $n-1$ steps, and $F_d/F$ is unramified of degree $d$). Writing $\mu = (\mu_1, \dots, \mu_d)$, there is a unique $i$ such that $\mu_i$ is non-central, and $\mu_i = \o^\vee_1$. In this case $K=\emptyset$ corresponds to the unique parahoric subgroup. \qed
\end{corollary}

The significance of this corollary  is given by the following fact. Let $E$ be the reflex field of $(\bG, \mu)$, i.e., the field of definition of $\mu$. Let $\mathfrak X$  be a formal scheme over $\Spf O_{\breve E}$ with underlying reduced scheme $\XmubK{\t}{K}$. We assume that $\mathfrak X$ is  flat over $\Spf O_{\breve E}$, and that its generic fiber, i.e., the associated rigid space $\mathfrak X^{\rm rig}$,  is smooth of dimension $\<\mu, 2 \rho\>$. Let $\pi$ be a uniformizer of $O_{\breve E}$. Assume that the formal scheme $\mathfrak X$ is $\pi$-adic, i.e., $\pi$ generates an ideal of definition of $\mathfrak X$. Equivalently, the ideal $\CJ$ of  $\XmubK{\t}{K}$ satisfies $\CJ= \rad(\pi\CO_{\mathfrak X})$ (radical ideal). Then $\XmubK{\t}{K}$ is equi-dimensional of dimension $\<\mu, 2 \rho\>$. Indeed, then $\XmubK{\t}{K}$ coincides  with the special fiber of $\mathfrak X$,  which is equi-dimensional of the same dimension as its generic fiber.  

Let $K\subset K'$. Let $\mathfrak X$ and $\mathfrak X'$ be two normal  flat formal schemes over $\Spf O_{\breve E}$ with underlying reduced scheme $\XmubK{\t}{K}$, resp. $\XmubK{\t}{K'}$,  and let $f\colon \mathfrak X\to \mathfrak X'$ be a proper morphism inducing the natural map $\XmubK{\t}{K}\to \XmubK{\t}{K'}$ and such that $f$ is a finite morphism in the generic fibers.    Let $\CJ$, resp. $\CJ'$, be the ideals of definitions of $\mathfrak X$, resp. $\mathfrak X'$.
\begin{lemma}
The equality $\CJ= \rad(\pi \CO_{\mathfrak X})$ holds if and only if $\CJ'=\rad(\pi \CO_{\mathfrak X'})$.
\end{lemma}
In other words, $\mathfrak X$ is a $\pi$-adic formal scheme  if and only if $\mathfrak X'$ is.
\begin{proof}  Assume $\CJ'=\rad(\pi\CO_{\frak X'})$. 
The morphism $f$ is adic, hence $f^*(\CJ')$ is an ideal of definition of $\mathfrak X$ which is contained in $\CJ$, as the latter is a maximal ideal of definition.  Hence $\CJ= \rad(\pi \CO_{\mathfrak X})$  is clear. For the other direction, let $\tilde f\colon \mathfrak X\to\tilde{ \mathfrak X}'$ be the Stein factorization of $f$. Then   the normality of $\tilde{\mathfrak X'}$ implies $\tilde f_*(\CO_{\mathfrak X})=\CO_{\tilde{\mathfrak X'}}$. On the other hand, for the maximal ideal of definition $\tilde\CJ'$ of $\tilde{\mathfrak X'}$, we have $\tilde \CJ'\subset \tilde f_*(\tilde f^*(\tilde\CJ'))\subset \tilde f_*(\CJ)$. Hence $\tilde\CJ'\subset \tilde f_*(\CJ)=\tilde f_*(\rad(\pi\CO_{\mathfrak X}))=\rad( \pi\tilde f_*(\CO_{\mathfrak X}))=\rad(\pi\CO_{\tilde{\mathfrak X'}})$, hence $\tilde {\mathfrak X'}$ is $\pi$-adic. But the normality of $\mathfrak X'$ implies that $\CO_{\mathfrak X'}\cap \pi\CO_{\tilde{\mathfrak X'}}=\pi\CO_{\tilde{\mathfrak X'}}$. Hence, since $\tilde{\mathfrak X'}$ is a $\pi$-adic formal scheme, so is $\mathfrak X'$. 
\end{proof}
\subsection{Via formal branches}\label{ss:formbranch}
In this subsection, we argue via the local structure of RZ-spaces.  Let $(\bG, \mu, K)$ be the corresponding local model triple over $F$, and $\BM^\loc(\bG, \mu)_K$ the local model over $O_E$,  in the sense of \cite{HPR}. Then the special fiber $\ov{\BM}^\loc(\bG, \mu)_K$ is a closed subset of the loop group partial affine flag variety $L\bG'/L^+\brK'$, 
\begin{equation}
\CA(\mu, \tau)_{K}=\{g \brK \in \breve G'/\brK'\mid g  \in \brK'\Admmu\brK'\}.
\end{equation}
By the \emph{local model diagram},  the singularities of the RZ-space $\CM(\bG, \mu, b)_K$ corresponding to $(\bG, \mu, b, K)$ are modeled by $\BM^\loc(\bG, \mu)_K$. More precisely, for any  $x\in\CM(\bG, \mu, b)_K(\ov k)$, there exists $y\in\BM^\loc(\bG, \mu)_K(\ov k)$ such that the strict henselizations at $x$ and at $y$ are isomorphic. Furthermore, for $b=\tau$, under the identification $\CM(\bG, \mu, \t)_K(\ov k)=\XmubK{\t}{K}$, the point $x_0=e\brK$ is realized by the point $y_0=\tau\in \CA(\mu, \tau)_{K}$.  Hence we have an identification
\begin{equation}
\begin{aligned}
\{\text{formal branches}&\text{ of the special fiber of $\CM(\bG, \mu, \t)_K$ through $x_0$}\}=\\
&\{\text{extreme elements of $\KAdmmu$}\}
\end{aligned}
\end{equation}
On the other hand, the extreme elements of $\KAdmmu$ can be identified with 
\begin{equation}
\KAdmmu^o:=\{ \lambda\in W_0(\umu)\mid t^\lambda\in {}^K\tilde W\}.
\end{equation}
Therefore, we deduce from Theorem \ref{max-dim} the following criterion.
\begin{theorem}
The RZ-space $\CM(\bG, \mu, \t)_K$ is $\pi$-adic if and only if  the inclusion $W(\umu)_{K,\fin}\subset \KAdmmu^o$ is an equality. 
\end{theorem}
This theorem again excludes the cases (2) and (3) of Theorem \ref{thmmaxequi}. Indeed, in these cases $K=\emptyset$ and the following elements are in $ \Admmu^o\setminus W(\umu)_{\fin}$:

\smallskip 

\noindent Case (2): $s_1 s_3 s_2 s_0 \t$.

\smallskip

\noindent Case (3): $s_0s_{n-1} \cdots s_2 s_{0'} s_{(n-1)'} \cdot s_{2'} \t$.

Here, in the last line, we use the notation from Subsection \ref{ss:reduc}. 

\subsection{Via non-archimedean uniformization}\label{ss:unif}
To put the above results into context, let us explain how to derive the above
statement using global methods, i.e., the theory of Shimura varieties. This
allows us to ``see'' all Newton strata at once, which is not possible within
one fixed RZ space. In this subsection, to simplify notations, we assume
$F=\BQ_p$. 

In each case of Theorem \ref{thmmaxequi} one can construct a Shimura pair $(\BG, \{h_{\BG}\})$ of PEL-type which yields after localization at $p$ the pair $(\bG, \mu)$. Let $\BK=\BK^p \BK_p\subset \BG(\BA_f)=\BG(\BA_f^p)\times\BG(\BQ_p)$, with $\BK_p=K$. Let $\BE=\BE(\BG, \{h_{\BG}\})$ be the global Shimura field and fix an embedding $\ov\BQ\subset\ov\BQ_p$ which determines a $p$-adic place $\nu$ of $\BE$ with $E=\BE_\nu$. 

Let $\CS_\BK=\CS(\BG, \{h_{\BG}\})_\BK$ be the Pappas-Zhu model of the Shimura variety $S(\BG, \{h_{\BG}\})_\BK$ over $O_E$. Then the \emph{Newton map}
\begin{equation}
\delta_\BK\colon \CS_\BK(\ov\BF_p)\to \BGmu
\end{equation}
is surjective, cf. \cite[\S 9]{HeZhou}.  In case (1) of  Theorem \ref{thmmaxequi}, the set $ \BGmu$ consists only of the unique basic element $[\t]$ of  $ \BGmu$, cf.~\cite{Ko2}; in cases (2) and (3), there are additional elements besides $[\t]$ (in case (2), one additional element). It follows that in cases (2) and (3), the closed subset $\CS_{\BK, {\rm basic}}$ with  $\CS_{\BK, {\rm basic}}(\ov\BF_p)=\delta_\BK^{-1}([\t])$ is a proper closed subset of the special fiber $\ov\CS_\BK$ of $\CS_\BK$.  Hence, in cases (2) and (3),  the formal completion $\CS^\wedge_{\BK}/_{\ov\CS_{\BK,{\rm basic}}}$ is a formal scheme over $\Spf O_E$ that is not $\pi$-adic. However, by non-archimedean uniformization \cite[Ch.~6]{RZ96}, there is an isomorphism of formal schemes over $\Spf O_{\breve E}$,
$$
\CS^\wedge_{\BK}/_{\ov\CS_{\BK,{\rm basic}}}\times_{\Spf O_E}\Spf O_{\breve E}\simeq \BG(\BQ)\backslash \big[\CM(\bG, \mu, \t)_K\times\BG(\BA_f^p)/\BK^p\big] . 
$$
It follows in cases (2) and (3) that the formal scheme $\CM(\bG, \mu, \t)_K$ is not $\pi$-adic.

\section{Proof of Theorems \ref{Main3} and \ref{Main4}}\label{finalpart3}
For Theorem \ref{Main3}, all that remains to be shown after Theorem \ref{max-dim} is the assertion that $W(\umu)_{K,\fin}$ parametrizes the orbits of ${\bf J}_b(F)$ on the set of irreducible components of dimension $\langle\mu, 2\rho\rangle$ of $X(\mu, b)_K$.

By Theorem \ref{max-dim}, the union of the irreducible components of maximal dimension is equal to $\cup_{\l \in W(\umu)_{K,\fin}} \overline{X_{K, t^\l}(b)}$. Note that each $\overline{X_{K, t^\l}(b)}$ is stable under the action of ${\bf J}_b(F)$. Moreover, the natural map from the set of irreducible components of $X_{K, t^\l}(b)$ to the set of irreducible components of $\overline{X_{K, t^\l}(b)}$ is bijective and ${\bf J}_b(F)$-equivariant. It remains to show that for any $\l \in W(\umu)_{K,\fin}$, ${\bf J}_b(F)$ acts transitively on the set of irreducible components of $X_{K, t^\l}(b)$. 
	
The natural projection map $\breve G/\brI \to \breve G/\brK$ induces the surjection $X_{t^\l}(b) \to X_{K, t^\l}(b)$ and this map is ${\bf J}_b(F)$-equivariant. Moreover, since $\l \in W(\umu)_{K,\fin}$, $W_{\supp_{\s}(t^\l)}$ is finite. By \cite[Prop.~2.2.1]{GH}, we have $X_{t^\l}(b) \cong {\bf J}_b(F) \times^{{\bf J}_b(F) \cap \brK} Y(w)$, where $\brK$ is the parahoric subgroup associated to $\supp_\s(t^\l)$ and $Y(w)$ is the classical Deligne-Lusztig variety associated to $w$ in the finite dimensional flag variety $\brK/\brI$. By~\cite[Ex.~3.10 d)]{Lusztig} (comp.~also~\cite[Cor.~1.2]{Go})
$Y(w)$ is irreducible. Hence ${\bf J}_b(F)$ acts transitively on the set of irreducible components of $X_{t^\l}(b)$, and hence transitively on the set of irreducible components of $X_{K, t^\l}(b)$.

Theorem~\ref{Main4} is deduced from Theorem~\ref{thmmaxequi} just as Theorems~\ref{Main1} and~\ref{Main2} are deduced from Theorems~\ref{thm-dimension-0} and~\ref{thm-fixed-s}. Corollary~\ref{MainvaryingK} follows from Theorem~\ref{Main4} by the observation that in cases (2) and (3) there are $F$-rational parahoric level structures other than the Iwahori level, comp.~Corollary~\ref{charDrin}.

Theorem~\ref{charpadicunif} follows from the fact that the integral RZ-data $\CD_{\BZ_p}$ are of extended Drinfeld type if $(\bG, \mu, K)$ is of type (1) in Theorem~\ref{Main3} (here the key is the fact that we assume that the first entry of a rational RZ-datum is a field extension of $\BQ_p$, so that the \emph{fake unitary group}  case is excluded).


\begin{thebibliography}{MMMxx}


\bibitem[BS17]{BS} B. Bhatt, P. Scholze, \emph{Projectivity of the Witt vector Grassmannian}, Invent. Math. \textbf{209} (2017), no. 2, 329--423.


\bibitem[Bo81]{Bour}
N. Bourbaki, \emph{Groupes et alg\`ebres de Lie}, Ch. 4, 5, 6, Masson 1981.


\bibitem[Ch]{Ch}  
S. Cho, \emph{The basic locus of the unitary Shimura variety with parahoric level structure, and special cycles},  arXiv:1807.09997 (2018).

\bibitem[CH17]{CH}
D. Ciubotaru, X. He, \emph{Cocenters and representations of affine Hecke algebra}, J. Eur. Math. Soc. \textbf{19} (2017), 3143--3177.

\bibitem[CKV15]{CKV} M.~Chen, M.~Kisin, E.~Viehmann, \emph{Connected components of affine Deligne-Lusztig varieties in mixed characteristic}, Compos.~Math.~\textbf{151} (2015), 1697--1762.

\bibitem[DL76]{DL} P.~Deligne, G.~Lusztig, \emph{Representations of reductive groups over finite fields}, Ann.~Math.~(2) \textbf{103}  (1976), 103--161.

\bibitem[Dr76]{Dr} V.~Drinfeld, \emph{Coverings of p-adic symmetric domains}, Funkcional. Anal. i Prilozen. {\bf 10} (1976), no. 2, 29--40  (Russian).


 \bibitem[Go09]{Go} U.~G\"{o}rtz, \emph{On the connectedness of {D}eligne-{L}usztig varieties}, Represent. Theory {\bf 13} (2009), 1--7. 

\bibitem[GHo12]{GHo12} U.~G\"{o}rtz, M.~Hoeve, \emph{Ekedahl-Oort strata and Kottwitz-Rapoport strata}, J.~Alg.~\textbf{351} (2012), 160--174.

\bibitem[GHKR]{GHKR1}
U.~G\"{o}rtz, T.~Haines, R.~Kottwitz, D.~Reuman, {\em  Dimensions of some affine Deligne-Lusztig varieties},
Ann. Sci. Ecole Norm. Sup. , \textbf{39} (2006), 467--511.


\bibitem[GH10]{GH0}
U.~ G\"ortz,  X.~ He, \emph{Dimension of affine Deligne-Lusztig varieties in affine flag varieties}, Doc. Math. \textbf{15} (2010),  1009--1028. 

\bibitem[GH15]{GH}
U. G\"ortz, X. He, \emph{Basic loci in Shimura varieties of Coxeter type}, Cambridge J. Math.~\textbf{3} (2015), no.~3, 323--353.


\bibitem[GHN]{GHN} U.~G\"ortz, X.~He, S.~Nie, \emph{Fully Hodge-Newton decomposable Shimura varieties}, to appear in Peking Math.~J., {\tt arxiv:1610.05381v2} (2019).

\bibitem[GY12]{GY} U.~G\"ortz, C.-F.~Yu, \emph{The supersingular locus in Siegel modular varieties with Iwahori level structure}, Math.~Ann.~\textbf{353} (2012),  465--498.



\bibitem[Ha01]{Haines:Bernstein}
T.~Haines, \emph{The combinatorics of Bernstein functions},  Trans. Amer. Math. Soc. {\bf 353} (2001), 1251--1278.


\bibitem[HR08]{Haines-Rapoport} T.~Haines, M.~Rapoport, \emph{On parahoric subgroups}, Appendix to: G.~Pappas, M.~Rapoport, Twisted loop groups and their affine flag varieties, Adv.~math.~\textbf{219} (2008), 118--198.


\bibitem[HT01]{HT}
M.~Harris and R.~Taylor, {The geometry and cohomology of some simple Shimura varieties}, Annals of Mathematics Studies \textbf{151}, Princeton Univ. Press, Princeton, NJ, 2001.


\bibitem[Ha11]{Hartwig} P.~Hartwig, \emph{On the reduction of the Siegel moduli space of abelian varieties of dimension 3 with Iwahori level structure}, M\"{u}nster J.~Math.~\textbf{4} (2011), 185--226.


\bibitem[He07]{He-Min}
X. He, {\em Minimal length elements in some double cosets of Coxeter groups},  Adv. Math. {\bf 215} (2007), 469--503. 

\bibitem[He09]{He-0}
X. He, {\em A subalgebra of $0$-Hecke algebra}, J. Algebra {\bf 322} (2009), 4030--4039.

\bibitem[He11]{He-11}
X.~He, \emph{Closure of Steinberg fibers and affine Deligne-Lusztig varieties},
Int. Math. Res. Not. IMRN 2011, 3237--3260. 

\bibitem[He14]{He14}
X. He, \emph{Geometric and homological properties of affine Deligne-Lusztig varieties}, Ann. Math. {\bf 179} (2014), 367--404.

\bibitem[He16a]{He-KR}
X. He, \emph{Kottwitz-Rapoport conjecture on unions of affine Deligne-Lusztig varieties},  Ann. Sci. Ecole Norm. Sup. {\bf 49} (2016), 1125--1141.

\bibitem[He16b]{He-CDM}
X. He, \emph{Hecke algebras and $p$-adic groups}, Current developments in mathematics 2015, 73--135, Int. Press, Somerville, MA, 2016.

\bibitem[HL15]{HL}
X. He, T. Lam, \emph{Projected Richardson varieties and affine Schubert varieties}, Ann. Inst. Fourier (Grenoble) {\bf 65} (2015), 2385--2412.

\bibitem[HN14]{HN}
X. He, S. Nie, \emph{Minimal length elements of extended affine Weyl group}, Compos. Math. {\bf 150} (2014), 1903--1927.

\bibitem[HPR] {HPR} X.~He, G.~Pappas, M.~Rapoport, \emph{Good and semi-stable reductions of Shimura varieties}, {\tt arXiv:1804.09615
} (2018) 
\bibitem[HR17]{HR}
X. He, M. Rapoport, \emph{Stratifications in the reduction of Shimura varieties}, Manuscripta Math. {\bf 152} (2017), 317--343.

\bibitem[HZ]{HeZhou}
X.~He, R.~Zhou, \emph{On the connected components of affine Deligne-Lusztig varieties}, {\tt arxiv:1610.06879} (2016).



\bibitem[Ko85]{Ko1} R.~Kottwitz, \emph{Isocrystals with additional structure}, Compos.~Math.~\textbf{56} (1985), 201--220.

\bibitem[Ko97]{Ko2} R.~Kottwitz, \emph{Isocrystals with additional structure. {II}}, Compos.~Math. \textbf{109} (1997), 255--339.

\bibitem[KR00]{kottwitz-rapoport} R.~Kottwitz, M.~Rapoport, \emph{Minuscule alcoves for $GL_n$ and $GSp_{2n}$}, Manuscripta Math.~\textbf{102} (2000), 403--428.

\bibitem[KR03]{KR03}
R.~ Kottwitz,  M.~ Rapoport, \emph{On the existence of F -crystals}, Comment. Math. Helv. \textbf{78} (2003), 153--184.


\bibitem[KRIII]{KRIII} S.~Kudla, M.~Rapoport, \emph{Notes on special cycles}, version Nov. 2012.


\bibitem[Lu]{Lusztig} G.~Lusztig, \emph{Representations of finite Chevalley groups}, Expository lectures from the CBMS
Regional Conference held at Madison, Wis., 1977. CBMS Regional Conf. Series in Math., \textbf{39}.  Amer. Math. Soc., 1978.


\bibitem[Oo01]{Oort} F.~Oort, \emph{A stratification of a moduli space of abelian varieties}, in: Moduli of Abelian Varieties (Texel Island, 1999). Progr.~Math.~\textbf{195}, Birkhäuser (2001), 345--416.

\bibitem[PRS]{PRS} G.~Pappas, M.~Rapoport, B.~Smithling, \emph{Local models of Shimura varieties, I. Geometry and combinatorics},
Handbook of moduli (eds. G. Farkas and I. Morrison), vol. III, 135--217,
Adv. Lect. in Math. \textbf{26}, International Press (2013),
 
\bibitem[Ra05]{Ra} M.~Rapoport, \emph{A guide to the reduction modulo $p$ of Shimura varieties}, Ast\'erisque \textbf{298} (2005), 271--318.


\bibitem[RR96]{Rapoport-Richartz} M.~Rapoport, M.~Richartz, \emph{On the classification and specialization of $F$-isocrystals with additional structures}, Compos.~Math.~\textbf{103}  (1996), 153--181.


 \bibitem[RV14]{RV} M.~Rapoport, E.~Viehmann, \emph{Towards a theory of local Shimura varieties}, M\"unster J. Math. {\bf 7} (2014), 273--326.
 
\bibitem[RZ96]{RZ96} M.~Rapoport, Th.~Zink, Period Spaces for $p$-divisible Groups, Annals of Mathematics Studies, {\bf 141}. Princeton University Press, Princeton, NJ, 1996. 

 \bibitem[RZ99]{RZ:indag} M.~Rapoport, Th.~Zink, \emph{A finiteness theorem in the Bruhat-Tits building: an application of Landvogt's embedding theorem}, Indag.~Math., N.S., \textbf{10} (3) (1999), 449--458.


\bibitem[St97]{St}
H.~Stamm, \emph{On the reduction of the Hilbert-Blumenthal-moduli scheme with $\G_0(p)$-level structure}, Forum Math. \textbf{9} (1997), no. 4, 405--455.


\bibitem[Ti79]{Tits:Corvallis} J.~Tits, \emph{Reductive Groups over Local Fields}, in: Automorphic Forms, Representations, and $L$-functions, A.~Borel, W.~Casselman (eds.), Proc.~Symp.~Pure Math \textbf{33}, vol.~1 (1979), 29--69.

\bibitem[Vi18]{HV} E.~Viehmann, \emph{Moduli spaces of local $G$-shtukas}, Proc. Int. Cong. of Math. 2018, vol. 2, 1443--1464.
\bibitem[Vo10]{Vollaard} I.~Vollaard, \emph{The supersingular locus of the Shimura variety for $GU(1,s)$}, Canad. J. Math. \textbf{62} (2010),  668--720.

\bibitem[VW11]{Vollaard-Wedhorn} I.~Vollaard, T.~Wedhorn, \emph{The supersingular locus of the Shimura variety of GU(1,n-1) II}, Invent.~Math.~\textbf{184} (2011), 591--627.

\bibitem[Wi05]{Wintenberger} J.-F.~Wintenberger, \emph{Existence de $F$-cristaux avec structures suppl\'ementaires}, Adv.~math.~\textbf{190} (2005), 196--224.

\bibitem[Zh17]{Zhu} X.~Zhu, \emph{Affine Grassmannians and the geometric Satake in mixed characteristic}, Ann.~Math.~\textbf{185} (2017),  403--492.


\end{thebibliography}
\end{document}